\documentclass[11pt,a4paper]{article}
\usepackage[utf8]{inputenc}
\usepackage[english]{babel}
\usepackage[T1]{fontenc}
\usepackage{amsmath}
\usepackage{amsfonts}
\usepackage{amssymb}
\usepackage{amsthm}
\usepackage{mathrsfs}
\usepackage{dsfont}
\usepackage[left=2cm,right=2cm,top=3cm,bottom=3cm]{geometry}
\usepackage{accents}
\usepackage{enumitem}

\RequirePackage[colorlinks,citecolor=blue,urlcolor=blue]{hyperref}

\usepackage{tikz}
\usepackage{pgfplots}
\usetikzlibrary{arrows,arrows.meta,calc,quotes,angles,shapes,plotmarks}

\newtheorem{theorem}{Theorem}[section]
\newtheorem{lemma}[theorem]{Lemma}
\newtheorem{proposition}[theorem]{Proposition}
\newtheorem{remark}[theorem]{Remark}

\numberwithin{equation}{section}

\DeclareMathOperator{\argmin}{argmin}

\newcommand{\ubar}[1]{\underaccent{\bar}{#1}}

\allowdisplaybreaks

\providecommand{\keywords}[1]
{
	\textbf{\textit{Keywords--}} #1
}

\providecommand{\MSC}[1]
{
	\textbf{\textit{MSC Classification--}} #1
}

\makeatletter
\newcommand{\owntag}[2][\relax]{
  \ifx#1\relax\relax\def\owntag@name{#2}\else\def\owntag@name{#1}\fi
  \refstepcounter{equation}\tag{\theequation, #2}%
  \expandafter\ltx@label\expandafter{eq:\owntag@name}%
  \def\@currentlabel{\theequation, #2}\expandafter\ltx@label\expandafter{Eq:\owntag@name}%
  \def\@currentlabel{#2}\expandafter\ltx@label\expandafter{tag:\owntag@name}%
}
\makeatother

\title{Convergence of Langevin-Simulated Annealing algorithms with multiplicative noise}
\author{Pierre Bras\footnote{Sorbonne Universit\'e, Laboratoire de Probabilit\'es, Statistique et Mod\'elisation, UMR 8001, case 158, 4 pl. Jussieu, F-75252 Paris Cedex 5, France. E-mail: \texttt{pierre.bras@sorbonne-universite.fr} and \texttt{gilles.pages@sorbonne-universite.fr}.} \footnote{Corresponding author.} $\ $ and Gilles Pag\`es\footnotemark[1]}
\date{}

\begin{document}

\maketitle

\begin{abstract}
We study the convergence of Langevin-Simulated Annealing type algorithms with multiplicative noise, i.e. for $V : \mathbb{R}^d \to \mathbb{R}$ a potential function to minimize, we consider the stochastic differential equation $dY_t = - \sigma \sigma^\top \nabla V(Y_t) dt + a(t)\sigma(Y_t)dW_t + a(t)^2\Upsilon(Y_t)dt$, where $(W_t)$ is a Brownian motion, where $\sigma : \mathbb{R}^d \to \mathcal{M}_d(\mathbb{R})$ is an adaptive (multiplicative) noise, where $a : \mathbb{R}^+ \to \mathbb{R}^+$ is a function decreasing to $0$ and where $\Upsilon$ is a correction term. This setting can be applied to optimization problems arising in Machine Learning; allowing $\sigma$ to depend on the position brings faster convergence in comparison with the classical Langevin equation $dY_t = -\nabla V(Y_t)dt + \sigma dW_t$.
The case where $\sigma$ is a constant matrix has been extensively studied however little attention has been paid to the general case.
We prove the convergence for the $L^1$-Wasserstein distance of $Y_t$ and of the associated Euler scheme $\bar{Y}_t$ to some measure $\nu^\star$ which is supported by $\argmin(V)$ and give rates of convergence to the instantaneous Gibbs measure $\nu_{a(t)}$ of density $\propto \exp(-2V(x)/a(t)^2)$. To do so, we first consider the case where $a$ is a piecewise constant function. We find again the classical schedule $a(t) = A\log^{-1/2}(t)$. We then prove the convergence for the general case by giving bounds for the Wasserstein distance to the stepwise constant case using ergodicity properties.
\end{abstract}


\MSC{62L20, 65C30, 60H35}

\section{Introduction}

Langevin-based algorithms are used to solve optimization problems in high dimension and have gained much interest in relation with Machine Learning. The Langevin equation is a Stochastic Differential Equation (SDE) which consists in a gradient descent with noise. More precisely, let $V :  \mathbb{R}^d \rightarrow \mathbb{R}^+$ be a coercive potential function, then the associated Langevin equation reads
$$ dX_t = -\nabla V(X_t)dt + \sigma dW_t , \ t \ge 0,$$
where $(W_t)$ is a $d$-dimensional Brownian motion and where $\sigma > 0$. Under standard assumptions, the invariant measure of this SDE is the Gibbs measure of density proportional to $e^{-2V(x)/\sigma^2}$ and for small enough $\sigma$, this measure concentrates around $\argmin(V)$ \cite{dalalyan2014} \cite{bras2021}.
Adding a small noise to the gradient descent allows to explore the space and to escape from traps such as local minima or saddle points appearing in non-convex optimization problems \cite{lazarev1992} \cite{dauphin2014}. This noise may also be interpreted as coming from the approximation of the gradient in stochastic gradient descent algorithms.
Such methods have been recently brought up to light again with Stochastic Gradient Langevin Dynamics (SGLD) algorithms \cite{welling2011} \cite{li2015}, especially for the deep learning and the calibration of large artificial neural networks, which is a high-dimensional non-convex optimization problem.

The Langevin-simulated annealing SDE is the Langevin equation where the noise parameter is slowly decreasing to $0$, namely
\begin{equation}
\label{eq:intro:1}
dX_t = -\nabla V(X_t) dt + a(t) \sigma dW_t , \ t \ge 0,
\end{equation}
where $a : \mathbb{R}^+ \to \mathbb{R}^+$ is non-increasing and converges to 0. The idea is that the "instantaneous" invariant measure $\nu_{a(t)\sigma}$ which is the Gibbs measure of density $\propto \exp(-2V(x)/(a(t)^2\sigma^2))$ converges itself to $\argmin(V)$.
This method indeed shares similarities with the original simulated annealing algorithm \cite{laarhoven1987}, which builds a Markov chain from the Gibbs measure using the Metropolis-Hastings algorithm and where the parameter $\sigma$, interpreted as a temperature, slowly decreases to zero over the iterations.

In \cite{chaing1987} and \cite{royer1989} is shown that choosing $a(t) = A \log^{-1/2}(t)$ for some $A>0$ in \eqref{eq:intro:1} guarantees the convergence of $X_t$ to $\nu^\star$ defined as the limit measure of $(\nu_{a(t)})$ as $t \to \infty$ and which is supported by $\argmin(V)$.
\cite{miclo1992} proves again the convergence of the SDE using free energies inequalities. These studies deeply rely on some Poincar\'e and log-Sobolev inequalities and require the following assumptions on the potential function:
$$ \lim_{|x| \rightarrow \infty} V(x) = \lim_{|x| \rightarrow \infty} |\nabla V(x)| = \infty \quad \text{and} \quad \forall x \in \mathbb{R}^d, \ \Delta V(x) \le C + |\nabla V(x)|^2 .$$
\cite{zitt2008} proves that the convergence still holds under weaker assumptions, in particular where the gradient of the potential is not coercive, using weak Poincaré inequalities.
In \cite{gelfand-mitter} is proved the convergence of the associated stochastic gradient descent algorithm.

All these results are established in the so-called additive case, i.e. they highly rely on the fact that $\sigma$ is constant, whereas little attention has been paid to the multiplicative case, i.e. where $\sigma : \mathbb{R}^d \to \mathcal{M}_d(\mathbb{R}) $ is not constant and depends on $X_t$. Allowing $\sigma$ to be adaptive and to depend on the position highly extends the range of applications of Langevin algorithms and such adaptive algorithms are already widely used by practitioners and prove to be faster than non-adaptive algorithms and competitive with standard non-Langevin algorithms or even faster. See Section \ref{subsec:practitioner} where various specifications for $\sigma(x)$ that can be found in the Stochastic Optimization literature are briefly presented, and Section \ref{sec:experiments} where we show results of simulations of the training of an artificial neural network for various choices of $\sigma$. However, to our knowledge, a general result of convergence for Langevin algorithms with multiplicative noise is yet to be proved.
\cite[Proposition 2.5]{pages2020} gives a general formula on $b$ and $\sigma$ so that the associated Gibbs measure is still the invariant measure of the SDE $dX_t = b(X_t)dt + \sigma(X_t)dW_t$; a simple example of acceleration of convergence using non-constant $\sigma$ is then given in \cite[Section 2.4]{pages2020}. More generally, \cite{ma2015} gives a characterization of any SGMCMC (Stochastic Gradient Markov Chain Monte Carlo) algorithm with multiplicative noise and with the corresponding Gibbs measure as a target.
In practice, the matrix $\sigma$ is often chosen so that $\sigma \sigma^\top \simeq (\nabla^2 V)^{-1}$ but approximations are needed because of the high dimensions of the matrix (e.g. only considering diagonal matrices). Still, our results hold also for non-diagonal $\sigma$, which opens the way to algorithms with such $\sigma$.

\medskip

In this paper, we consider the following SDE:
\begin{align*}
& dY_t = -(\sigma \sigma^\top \nabla V)(Y_t)dt + a(t) \sigma(Y_t) dW_t + \left(a^2(t) \left[\sum_{j=1}^d \partial_j(\sigma \sigma^\star)(Y_t)_{ij} \right]_{1 \le i \le d}\right) dt \\
& a(t) = \frac{A}{\sqrt{\log(t)}},
\end{align*}
where the expression of the drift comes from \cite[Proposition 2.5]{pages2020} and where the second drift term is interpreted as a correction term so that $\nu_{a(t)}$ is still the the "instantaneous" invariant measure. This last term boils down to $0$ if $\sigma$ is constant.
The aim of this paper is to prove the convergence for the $L^1$-Wasserstein distance of the law of $Y_t$ to $\nu^\star$ in the setting adopted in \cite{pages2020}, assuming in particular the convex uniformity of the potential outside a compact set and the ellipticity and the boundedness of $\sigma$.
We also prove the convergence of the corresponding Euler-Maruyama scheme with decreasing steps.

Considering the convex condition outside a compact set is in fact quite different from the convex setting and turns out to be more demanding. This setting often appears in optimization problems (see Section \ref{sec:neural_networks}), where a characteristic set - the compact set - contains the interesting features of the model with traps such as local minima, and where outside of this set the loss function is coercive and convex. We give classic examples of neural networks where this setting applies.

\medskip
We adopt a \textit{domino strategy} like in \cite{pages2020}, inspired by proofs of weak error expansion of discretization schemes of diffusion processes, see \cite{talay1990} and \cite{bally1996}. In \cite{pages2020} is proved the convergence of the Euler-Maruyama scheme $\bar{X}$ with decreasing steps $(\gamma_n)$ of an ergodic and homogeneous SDE $X$ with non constant $\sigma$, to the invariant measure of $X$. It then appears that the multiplicative case is much more demanding than the additive case. For a function $f : \mathbb{R}^d \to \mathbb{R}$, the \textit{domino strategy} consists in a step-by-step decomposition of the weak error to produce an upper bound as follows:
\begin{align}
| \mathbb{E}f(\bar{X}_{\Gamma_n}^x) - \mathbb{E}f(X_{\Gamma_n}^x)| & = | \bar{P}_{\gamma_1} \circ \cdots \circ \bar{P}_{\gamma_n} f(x) - P_{\Gamma_n}f(x) | \nonumber \\
\label{eq:domino_strategy}
& \le \sum_{k=1}^n \left|\bar{P}_{\gamma_1} \circ \cdots \circ \bar{P}_{\gamma_{k-1}} \circ (\bar{P}_{\gamma_k}-P_{\gamma_k})\circ P_{\Gamma_n-\Gamma_k}f(x) \right|,
\end{align}
where $P$ and $\bar{P}$ are the transition kernels associated to $X$ and $\bar{X}$ respectively and where $\Gamma_n=\gamma_1+\cdots+\gamma_n$. Then two terms appear: first the "error" term, for large $k$, where the error is controlled by classic weak and strong bounds on the error of an Euler-Maruyama scheme, and the "ergodic" term, for small $k$, where the ergodicity of $X$ is used.

However, we cannot directly apply this strategy of proof to our problem since we consider a non homogeneous SDE $Y$, so we proceed as follows: we consider instead the SDE $X$ where the coefficient $a(t)$ is non-increasing and piecewise constant and where the successive plateaux $[T_{n-1},T_{n})$ of $a$ are increasingly larger time intervals. On each plateau we obtain a homogeneous and uniformly elliptic SDE with an invariant Gibbs distribution $\nu_{a_n}$ where $a_n$ is the constant value of $a$ on $[T_{n-1},T_{n})$, to which a \textit{domino strategy} can be applied. This ellipticity fades with time since $a_n$ goes to $0$ and we need to carefully control its impact on the way the diffusion $X$ gets close to its "instantaneous" invariant Gibbs distribution $\nu_{a_n}$. To this end we have to refine several one step weak error results from \cite{pages2020} and ergodic bounds from \cite{wang2020}. Doing so we derive by induction an upper-bound for the distance between $X_t$ and $\nu^\star$ after each plateau and prove that a coefficient $(a(t))$ of order $\log^{-1/2}(t)$ is a sufficient and generally necessary condition for convergence.
Using this result, we then prove the convergence of $Y_t$ and its Euler-Maruyama scheme $\bar{Y}_t$ by bounding the distance between $X_t$ and $Y_t$ and $X_t$ and $\bar{Y}_t$. We also consider the "Stochastic Gradient case" i.e. where the true gradient cannot be computed exactly and where a noise, which is a sequence of increments of a martingale, is added to the gradient. This case was treated in \cite{gelfand-mitter} in the additive setting.
The process $X$ is used as a tool for the proof of the convergence of $Y_t$, however the convergence of $X_t$ to $\nu^\star$ also has its own interest since the "plateau" method is also used by practitioners.

We also establish a convergence rate which is somehow limited by $\mathcal{W}_1(\nu_{a(t)},\nu^\star)$, which is of order $a(t)$ under the assumption that $\argmin(V)$ is finite and that $\nabla^2 V$ is positive definite at every element of $\argmin(V)$; if $\argmin(V)$ is still finite but if $\nabla^2 V$ is not positive definite at every element of $\argmin(V)$, but if we assume instead that all the elements of $\argmin(V)$ are strictly polynomial minima, then the rate is of order $a(t)^\delta$ for some $\delta \in (0,1)$ \cite{bras2021}. We pay particular attention to the non-definite case, since it was pointed out in \cite{sagun2016} and \cite{sagun2017} that for some optimization problems arising in Machine Learning, the Hessian of the loss function at the end of the training tends out to be extremely singular. Indeed, as the dimension of the parameter which is used to minimize the loss function is large and as the neural network can be over-parametrized, many eigenvalues of the Hessian matrix are close to zero. However, this subject is still new in the Stochastic Optimization literature and needs more theoretical background. 

Still we give sharper bounds on the rate of convergence of the $L^1$-Wasserstein distance between $X$ or $Y$ and $\nu_{a(t)}$ as in practice the optimization procedure stops at some (large) $t$ and the target distribution is actually $\nu_{a(t)}$ instead of $\nu^\star$.

In a next paper, we shall prove the convergence in total variation distance. In this last case, the domino strategy is more complex to implement and requires regularization lemmas, as in \cite{pages2020} which studies the convergence for both distances.

\medskip

The article is organized as follows.
In Section \ref{sec:main-results} we first give the setting and assumptions of the problem we consider. This setting is taken from \cite{pages2020}. We then state our main results of convergence as well as convergence rates. In Section \ref{sec:optimization} we show how this setting applies to some classic optimization problems arising in Machine Learning and present several general choices for $\sigma$ that are used in practice. In Section \ref{sec:langevin} we consider the case where the coefficient $a$ is constant and give convergence rates to the invariant measure taking into account the ellipticity parameter. We also give preliminary lemmas for the rate of convergence of $\nu_{a_n}$ to $\nu^\star$. In Section \ref{sec:X} we prove the convergence of the solution of the SDE where $a$ is piecewise constant, by "plateaux". Using the dependence in $a$ of the rate of convergence to the invariant measure in the ergodic case we prove the convergence to $\argmin(V)$. In Section \ref{sec:Y} we then prove the convergence of the SDE in the case where $a$ is not by plateau but continuously decreasing. This is done by bounding the Wasserstein distance with the "plateau" case and revisiting the lemmas for strong and weak errors from \cite{pages2020}. In Section \ref{sec:bar-Y} and Section \ref{sec:bar-X} we also prove the convergence for the corresponding the Euler-Maruyama schemes. The proofs actually follow the same strategy as the previous one. In Section \ref{sec:experiments} we present experiments of training of neural networks using various specifications for $\sigma$; the algorithms with multiplicative $\sigma$ prove to be faster than the algorithm with constant $\sigma$.

\bigskip

\textsc{Notations}

We endow the space $\mathbb{R}^d$ with the canonical Euclidean norm denoted by $| \boldsymbol{\cdot} |$ and we denote $\langle \cdot, \cdot \rangle$ the associated canonical inner product. For $x \in \mathbb{R}^d$ and for $R>0$, we denote $\textbf{B}(x,R) = \lbrace y \in \mathbb{R}^d : \ |y-x| \le R \rbrace$.

For $M \in (\mathbb{R}^d)^{\otimes k}$, we denote by $\|M\|$ its operator norm, i.e. $\|M\| = \sup_{u \in \mathbb{R}^{d\times k}, \ |u|=1} M \cdot u$. If $M : \mathbb{R}^d \to (\mathbb{R}^d)^{\otimes k}$, we denote $\|M\|_\infty = \sup_{x \in \mathbb{R}^d} \|M(x)\|$.

For $f :\mathbb{R}^d \rightarrow \mathbb{R}$ such that $\min_{\mathbb{R}^d}(f)$ exists, we denote $\text{argmin}(f) = \left\lbrace x \in \mathbb{R}^d : \ f(x) = \min_{\mathbb{R}^d}(f) \right\rbrace$.
We say that $f$ is coercive if $\lim_{|x| \rightarrow \infty} f(x) = + \infty$. 
If $f$ is Lipschitz continuous, we denote by $[f]_{\text{Lip}}$ its Lipschitz constant.
For $k \in \mathbb{N}$ and if $f$ is $\mathcal{C}^k$, we denote by $\nabla ^k f : \mathbb{R}^d \rightarrow (\mathbb{R}^d)^{\otimes k}$ its differential of order $k$.

For a random vector $X$, we denote by $[X]$ its law.

We denote the $L^p$-Wasserstein distance between two distributions $\pi_1$ and $\pi_2$ on $\mathbb{R}^d$:
$$ \mathcal{W}_p(\pi_1, \pi_2) = \inf \left\lbrace \left(\int_{\mathbb{R}^d} |x-y|^p \pi(dx,dy) \right)^{1/p} : \ \pi \in \mathcal{P}(\pi_1,\pi_2) \right\rbrace ,$$
where $\mathcal{P}(\pi_1,\pi_2)$ stands for the set of probability distributions on $(\mathbb{R}^d \times \mathbb{R}^d, \mathcal{B}or(\mathbb{R}^d)^{\otimes 2})$ with respective marginal laws $\pi_1$ and $\pi_2$. For $p=1$, let us recall the Kantorovich-Rubinstein representation of the Wasserstein distance of order 1 \cite[Equation (6.3)]{villani2009}:
$$ \mathcal{W}_1(\pi_1,\pi_2) = \sup \left\lbrace \int_{\mathbb{R}^d} f(x) (\pi_1-\pi_2)(dx) : \ f : \mathbb{R}^d \to \mathbb{R}, \ [f]_{\text{Lip}} = 1 \right\rbrace .$$

For $x \in \mathbb{R}^d$, we denote by $\delta_x$ the Dirac mass at $x$.
For $x$, $y \in \mathbb{R}^d$ we denote $(x,y) = \lbrace ux + (1-u)y, \ u \in [0,1] \rbrace$ the geometric segment between $x$ and $y$.

For $u$, $v \in \mathbb{R}$, we define $u \ \text{mod}(v) = u - v \lfloor u/v \rfloor$.

If $u_n$ and $v_n$ are two real-valued sequences, we write $u_n \asymp v_n$ meaning that $u_n = O(v_n)$ and $v_n = O(u_n)$.

In this paper, we use the notation $C$ to denote a positive real constant, which may change from line to line. The constant $C$ depends on the parameters of the problem: the coefficients of the SDE, the choice of $A$ in $a(t) = A \log^{-1/2}(t)$, the upper bound $\bar{\gamma}$ on the decreasing steps, but $C$ does not depend on $t$ nor $x$.

\section{Assumptions and main results}
\label{sec:main-results}

\subsection{Assumptions}

Let $V : \mathbb{R}^d \rightarrow (0,+\infty)$ be a $\mathcal{C}^2$ potential function such that $V$ is coercive and
\begin{equation}
\label{eq:def_A}
(x \mapsto |x|^2 e^{-2V(x)/A^2}) \in L^1(\mathbb{R}^d) \text{ for some } A>0.
\end{equation}
Then $V$ admits a minimum on $\mathbb{R}^d$. Moreover, let us assume that
\begin{align}
V^\star :=\min_{\mathbb{R}^d} V >0, \quad \argmin(V) = \lbrace x_1^\star, \ldots, x_{m^\star}^\star \rbrace, \quad \forall \ i =1,\ldots,m^\star, \ \nabla^2 V(x_i^\star) >0,
\owntag[eq:min_V]{$\mathcal{H}_{V1}$}
\end{align}
i.e. $\min_{\mathbb{R}^d} V$ is attained at a finite number $m^\star$ of points and in each point the Hessian matrix is positive definite.
We then define for $a \in (0,A]$ the Gibbs measure $\nu_{a}$ of density :
\begin{equation}
\label{eq:def_nu}
\nu_a(dx) = \mathcal{Z}_{a} e^{-2(V(x)-V^\star)/a^2} dx , \quad \mathcal{Z}_{a} = \left( \int_{\mathbb{R}^d} e^{-2(V(x)-V^\star)/a^2} dx \right)^{-1}
\end{equation}
Following \cite[Theorem 2.1]{hwang1980}, the measure $\nu_a$ converges weakly to $\nu^\star$ as $a \to 0$, where $\nu^\star$ is the weighted sum of Dirac measures:
\begin{equation}
\nu^\star = \left(\sum_{j=1}^{m^\star} \left(\det \nabla^2 V(x_j^\star) \right)^{-1/2} \right)^{-1} \sum_{i=1}^{m^\star} \left(\det \nabla^2 V(x_i^\star)\right)^{-1/2} \delta_{x_i^\star}.
\end{equation}

\medskip

We consider the following Langevin SDE in $\mathbb{R}^d$:
\begin{align}
\label{eq:def_Y}
& Y_0^{x_0} = x_0 \in \mathbb{R}^d, \\
& dY_t^{x_0} =  b_{a(t)}(Y_t^{x_0})dt + a(t) \sigma(Y_t^{x_0}) dW_t, \nonumber
\end{align}
where, for $a\ge 0$, the drift $b_a$ is given by
\begin{equation}
\label{eq:def_b}
b_a(x) = -(\sigma \sigma^\top \nabla V)(x) + a^2 \left[\sum_{j=1}^d \partial_j(\sigma \sigma^\top)_{ij}(x) \right]_{1 \le i \le d} =: -(\sigma \sigma^\top \nabla V)(x) + a^2 \Upsilon(x),
\end{equation}
where $W$ is a standard $\mathbb{R}^d$-valued Brownian motion defined on a probability space $(\Omega, \mathcal{A}, \mathbb{P})$, where $ \sigma : \mathbb{R}^d \to \mathcal{M}_{d}(\mathbb{R})$ is $C^2$ and
\begin{equation}
\label{eq:def_a}
a(t) = \frac{A}{\sqrt{\log(t+e)}} 
\end{equation}
where $A$ is defined in \eqref{eq:def_A} and with $\log(e)=1$.
This equation corresponds to a gradient descent on the potential $V$ with preconditioning $\sigma$ and multiplicative noise ; the second term in the drift \eqref{eq:def_b} is a correction term (see \cite[Proposition 2.5]{pages2020}) which is zero for constant $\sigma$.

\medskip

We make the following assumptions on the potential $V$:
\begin{align}
\lim_{|x| \rightarrow + \infty} V(x) = + \infty, \ \ |\nabla V|^2 \le CV \ \text{ and } \sup_{x \in \mathbb{R}^d} || \nabla^2 V(x)|| < + \infty ,
\owntag[eq:V_assumptions]{$\mathcal{H}_{V2}$}
\end{align}
which implies in particular that $V$ has at most a quadratic growth.
Let us also assume that
\begin{align}
\sigma \text{ is bounded and Lipschitz continuous,} \ \nabla^2 \sigma \text{ is bounded}, \ \nabla(\sigma\sigma^\top) \nabla V \text{ is bounded},
\owntag[eq:sigma_assumptions]{$\mathcal{H}_\sigma$}
\end{align}
and that $\sigma$ is uniformly elliptic, i.e.
\begin{equation}
\label{eq:ellipticity}
\exists \ubar{\sigma}_0 > 0, \ \forall x \in \mathbb{R}^d, \ (\sigma \sigma^\top) (x) \ge \ubar{\sigma}_0^2 I_d .
\end{equation}
Assumptions \eqref{Eq:eq:V_assumptions} and \eqref{Eq:eq:sigma_assumptions} imply that $\Upsilon$ is also bounded and Lipschitz continuous and that $b_a$ is Lipschitz continuous uniformly in $a \in [0,A]$. Let the minimal constant $[b]_{\text{Lip}}$ be such that:
$$ \forall a \in [0,A], \ b_a \text{ is } [b]_{\text{Lip}} \text{-Lipschitz continuous} .$$

We make the non-uniform dissipative (or convexity) assumption outside of a compact set: there exists $\alpha_0 >0$ and $R_0 >0$ such that
\begin{align}
\forall x, y \in \textbf{B}(0,R_0)^c, \ \left\langle \left( \sigma \sigma^\top \nabla V\right)(x)-\left(\sigma \sigma^\top \nabla V \right) (y), \ x-y \right\rangle \ge \alpha_0 |x-y|^2.
\owntag[eq:V_confluence]{$\mathcal{H}_{cf}$}
\end{align}

Taking $y \in \textbf{B}(0,R_0)^c$ fixed, letting $|x| \to \infty$ and using the boundedness of $\sigma$, \eqref{Eq:eq:V_confluence} implies that $|\nabla V|$ is coercive.
Using \eqref{Eq:eq:V_assumptions} and the boundedness of $\sigma$, there exists $C>0$ (depending on $A$) such that:
$$ \forall a \in [0,A], \ 1 + |b_a(x)| \le CV^{1/2}(x) .$$

\medskip

Let $(\gamma_n)_{n \ge 1}$ be a non-increasing sequence of varying positive steps. We define $\Gamma_n := \gamma_1 + \cdots + \gamma_n$ and for $t \ge 0$:
$$ N(t) := \min \lbrace k \ge 0 : \ \Gamma_{k+1} > t \rbrace = \max \lbrace k \ge 0 : \ \Gamma_k \le t \rbrace .$$
We make the classical assumptions on the step sequence, namely
\begin{equation}
\gamma_n \downarrow 0, \quad \sum_{n \ge 1} \gamma_n = + \infty \quad \text{and} \quad \sum_{n \ge 1} \gamma_n^2 < + \infty
\owntag[eq:gamma_assumptions]{$\mathcal{H}_{\gamma1}$}
\end{equation}
and we also assume that
\begin{align}
\varpi := \limsup_{n \to \infty} \frac{\gamma_n - \gamma_{n+1}}{\gamma_{n+1}^2} < \infty .
\owntag[eq:gamma_assumptions_2]{$\mathcal{H}_{\gamma2}$}
\end{align}
For example, if $\gamma_n = \gamma_1/n^\alpha$ with $\alpha \in (1/2,1)$ then $\varpi = 0$; if $\gamma_n = \gamma_1/n$ then $\varpi = \gamma_1$.

\medskip

In Stochastic Gradient algorithms, the true gradient is measured with a zero-mean noise $\zeta$, which law only depends on the current position. That is, let us consider a family of random fields $(\zeta_n(x))_{x \in \mathbb{R}^d, n \in \mathbb{N}}$ such that for every $n \in \mathbb{N}$, $(\omega, x) \in \Omega \times \mathbb{R}^d \mapsto \zeta_n(x,\omega)$ is measurable and for all $x \in \mathbb{R}^d$, the law of $\zeta_n(x)$ only depends on $x$ and $(\zeta_n(x))_{n \in \mathbb{N}}$ is an i.i.d. sequence independent of $W$. We make the following assumptions:
\begin{equation}
\label{eq:zeta_assumptions}
\forall x \in \mathbb{R}^d, \ \forall p \ge 1, \ \mathbb{E}[\zeta_{1}(x)] = 0 \quad \text{and} \quad \mathbb{E}[|\zeta_{1}(x)|^p] \le C_p V^{p/2}(x).
\end{equation}
We then consider the Euler-Maruyama scheme with decreasing steps associated to $(Y_t)$:
\begin{align}
\label{eq:def_Y_bar}
& \bar{Y}_0^{x_0} = x_0, \quad \bar{Y}_{\Gamma_{n+1}}^{x_0} = \bar{Y}_{\Gamma_n} + \gamma_{n+1} \left(b_{a(\Gamma_{n})}(\bar{Y}^{x_0}_{\Gamma_{n}}) + \zeta_{n+1}(\bar{Y}_{\Gamma_n}^{x_0}) \right) + a(\Gamma_{n}) \sigma(\bar{Y}_{\Gamma_n}^{x_0})(W_{\Gamma_{n+1}} - W_{\Gamma_n}),
\end{align}
We extend $\bar{Y}^{x_0}_{\boldsymbol{\cdot}}$ on $\mathbb{R}^+$ by considering its genuine continuous interpolation:
\begin{equation}
\label{eq:def_Y_bar_genuine}
\forall t \in [\Gamma_n, \Gamma_{n+1}), \  \bar{Y}^{x_0}_{t} = \bar{Y}^{x_0}_{\Gamma_n} + (t-\Gamma_n) \left(b_{a(\Gamma_n)}(\bar{Y}^{x_0}_{\Gamma_n}) + \zeta_{n+1}(\bar{Y}_{\Gamma_n}^{x_0}) \right) + a(\Gamma_n) \sigma(\bar{Y}^{x_0}_{\Gamma_n}) (W_t - W_{\Gamma_n}) .
\end{equation}

\subsection{Main results}

We now state our main results.

\begin{theorem}
\label{thm:main}
\begin{enumerate}[label=(\alph*)]
\item Let $Y$ be defined in \eqref{eq:def_Y}. Assume \eqref{Eq:eq:min_V}, \eqref{Eq:eq:V_assumptions}, \eqref{Eq:eq:sigma_assumptions}, \eqref{eq:ellipticity} and \eqref{Eq:eq:V_confluence}. If $A$ is large enough, then for every $x_0 \in \mathbb{R}^d$,
$$ \mathcal{W}_1([Y_t^{x_0}],\nu^\star) \underset{t \rightarrow \infty}{\longrightarrow} 0. $$
More precisely, for every $t >0$:
$$ \mathcal{W}_1([Y_t^{x_0}],\nu^\star) \le C \max(1+|x_0|,V(x_0))a(t) $$
and for every $\alpha \in (0,1)$ we have
$$ \mathcal{W}_1([Y_t^{x_0}],\nu_{a(t)}) \le C \max(1+|x_0|,V(x_0))t^{-\alpha} . $$

\item Let $\bar{Y}$ be defined in \eqref{eq:def_Y_bar}. Assume \eqref{Eq:eq:min_V}, \eqref{Eq:eq:V_assumptions}, \eqref{Eq:eq:sigma_assumptions}, \eqref{eq:ellipticity} and \eqref{Eq:eq:V_confluence}. Assume furthermore \eqref{Eq:eq:gamma_assumptions} and \eqref{Eq:eq:gamma_assumptions_2}, that $V$ is $\mathcal{C}^3$ with $\|\nabla^3 V\| \le CV^{1/2}$ and that $\sigma$ is $\mathcal{C}^3$ with $\|\nabla^3(\sigma \sigma^\top) \| \le CV^{1/2}$. If $A$ is large enough then for every $x_0 \in \mathbb{R}^d$,
$$ \mathcal{W}_1([\bar{Y}_t^{x_0}],\nu^\star) \underset{t \rightarrow \infty}{\longrightarrow} 0. $$
More precisely, for every $t >0$:
$$ \mathcal{W}_1([\bar{Y}_t^{x_0}],\nu^\star) \le C\max(1+|x_0|,V^2(x_0)) a(t) ,$$
and for every $\alpha \in (0,1)$ we have
$$ \mathcal{W}_1([\bar{Y}_t^{x_0}],\nu_{a(t)}) \le C\max(1+|x_0|,V^2(x_0))t^{-\alpha} . $$
\end{enumerate}
\end{theorem}

\begin{remark}
In particular, if $\argmin V = \lbrace x^\star \rbrace$ is reduced to a point, we can rewrite the conclusions of Theorem \ref{thm:main} as $\|Y_t^{x_0} - x^\star\|_1 \to 0$ and $\|\bar{Y}_t^{x_0} - x^\star\|_1 \to 0$ respectively and so on.
\end{remark}

\subsection{The degenerate case}

In this subsection we consider the case where some of the $\nabla^2 V(x_i^\star)$'s may be not definite positive but where the $x_i^\star$'s are strictly polynomial minima, i.e. is $V(x)-V(x_i^\star)$ is bounded below in a neighbourhood of $x_i^\star$ by a non-negative polynomial function null only in $x_i^\star$. This case can be treated in a similar way using the change of variable given in \cite{bras2021}.

First, let us restate the results from \cite[Theorem 4]{bras2021}. To simplify, let us assume that $\argmin(V)$ is reduced to a point.

\begin{theorem}
\label{thm:athreya}
Assume that $V$ is $C^{2p}$ with $p \ge 2$, is coercive, that $\argmin(V) = \lbrace x^\star \rbrace$, that $e^{-AV} \in L^1(\mathbb{R}^d)$ for some $A>0$ and that $x^\star$ is a strictly polynomial minimum of order $2p$ i.e. $p$ is the smallest integer such that
$$ \exists r >0, \ \forall h \in \textbf{B}(x^\star,r)\setminus \lbrace 0 \rbrace, \ \sum_{k=0}^{2p} \frac{1}{k!} \nabla^k V(x^\star) \cdot h^{\otimes k} > 0 .$$
Assume also the technical hypothesis \cite[(8)]{bras2021} if $p \ge 5$.
Then there exist $B \in \mathcal{O}_d(\mathbb{R})$, $\alpha_1, \ \ldots, \ \alpha_d \in \lbrace 1/2, \ldots, 1/(2p) \rbrace$ and a polynomial function $g : \mathbb{R}^d \to \mathbb{R}$ which is not constant in any of its variables such that
$$ \forall h \in \mathbb{R}^d, \ \frac{1}{s}[ V(x^\star + B \cdot (s^{\alpha_1}h_1,\ldots, s^{\alpha_d}h_d)) - V(x^\star)] \underset{s \to 0}{\longrightarrow} g(h) .$$
Moreover assume that $g$ is coercive. Then if $Z_s \sim \nu_{\sqrt{2s}}$,
$$ \left( \frac{(B^{-1}\cdot (Z_s-x^\star))_1}{s^{\alpha_1}}, \cdots, \frac{(B^{-1}\cdot (Z_s-x^\star))_d}{s^{\alpha_d}} \right) \overset{\mathscr{L}}{\longrightarrow} Z \quad \text{ as } s \to 0,$$
where $Z$ has density proportional to $\exp(-g)$.
\end{theorem}

\begin{theorem}
\label{thm:non_definite}
Let us make the same assumptions as in Theorem \ref{thm:athreya} and assume that $V^\star > 0$. Assume furthermore \eqref{Eq:eq:V_assumptions}, \eqref{Eq:eq:sigma_assumptions}, \eqref{eq:ellipticity} and \eqref{Eq:eq:V_confluence}. Assume furthermore \eqref{Eq:eq:gamma_assumptions} and \eqref{Eq:eq:gamma_assumptions_2}, that $V$ is $\mathcal{C}^3$ with $\|\nabla^3 V\| \le CV^{1/2}$ and that $\sigma$ is $\mathcal{C}^3$ with $\|\nabla^3(\sigma \sigma^\top) \| \le CV^{1/2}$.
Let us denote $\alpha_{\min} := \min(\alpha_1,\ldots,\alpha_d)$. Then for every $\alpha \in (0,1)$ we have
$$ \mathcal{W}_1([Y_t^{x_0}],\nu_{a(t)}) \le C \max(1+|x_0|,V(x_0))t^{-\alpha} \quad \textup{ and } \quad \mathcal{W}_1([\bar{Y}_t^{x_0}],\nu_{a(t)}) \le C \max(1+|x_0|,V^2(x_0))t^{-\alpha} $$
and
$$ \mathcal{W}_1([Y_t^{x_0}],\nu^\star) \le C \max(1+|x_0|,V(x_0))a(t)^{2\alpha_{\min}} \textup{ and } \mathcal{W}_1([\bar{Y}_t^{x_0}],\nu^\star) \le C \max(1+|x_0|,V^2(x_0))a(t)^{2\alpha_{\min}}. $$
\end{theorem}
The proof is given in the Supplementary Material.

\section{Application to optimization problems}
\label{sec:optimization}

\subsection{Potential function associated to a Neural Regression Problem}
\label{sec:neural_networks}

The setting described in Section \ref{sec:main-results} can first be applied to convex optimization problems where the potential function $V$ has a quadratic growth as $|x| \rightarrow \infty$. Classical examples are least-squares regression and logistic regression with quadratic regularization, that is:
$$ \min_{x \in \mathbb{R}^d} \frac{1}{M} \sum_{i=1}^M \log(1+e^{-v_i \langle u_i, x \rangle}) + \frac{\lambda}{2} |x|^2 ,$$
where $v_i \in \lbrace -1, +1 \rbrace$ and $u_i \in \mathbb{R}^d$ are the data samples associated with a binary classification problem and where $\lambda >0$ is the regularization parameter.


\medskip

We now consider a scalar regression problem with a fully connected neural network with quadratic regularization. Let $\varphi : \mathbb{R} \rightarrow \mathbb{R}$ be the sigmoid function. To simplify the proofs, we may consider instead a smooth function approximating the sigmoid function such that $\varphi'$ has compact support. Let $K \in \mathbb{N}$ be the number of layers and for $k=1$, $\ldots$, $K$, let $d_k \in \mathbb{N}$ be the size of the $k^{\text{th}}$ layer with $d_K=1$. For $u \in \mathbb{R}^{d_{k-1}}$ and for $\theta \in \mathcal{M}_{d_k,d_{k-1}}(\mathbb{R})$, we define $\varphi_{\theta}(u) := [\varphi([\theta \cdot u]_i)]_{1 \le i \le d_k}$. The output of the neural network is
\begin{align*}
& \psi : \mathbb{R}^{d_1,d_0} \times \cdots \times \mathbb{R}^{d_K,d_{K-1}} \times \mathbb{R}^{d_0} \to \mathbb{R} \\
& \psi (\theta_1,\ldots,\theta_K,u) = \psi(\theta,u) = \varphi_{\theta_K} \circ \ldots \circ \varphi_{\theta_1}(u).
\end{align*}
Let $u_i \in \mathbb{R}^{d_0}$ and $v_i \in \mathbb{R}$ be the data samples for $1 \le i \le M$. The objective is
$$ \underset{\theta_1, \ldots, \theta_K}{\text{minimize}} \quad V(\theta) := \frac{1}{2M} \sum_{i=1}^M (\psi(\theta_1, \ldots, \theta_K,u_i) -v_i)^2 + \frac{\lambda}{2}|\theta|^2 ,$$
where $\theta = (\theta_1, \ldots, \theta_K)$ and where $\lambda > 0$.

\begin{proposition}
Consider a neural network with a single layer : $\psi(\theta,u) = \varphi(\langle \theta, u \rangle)$. Assume that the data $u$ and $v$ are bounded and that $u$ admits a continuous density. Then $V$ satisfies \eqref{Eq:eq:V_assumptions} and for some $R_0, \ \alpha_0 >0$,
\begin{equation}
\label{eq:NN_condition}
\forall x, \ y \in \textbf{B}(0,R_0)^c, \ \langle \nabla V(x) - \nabla V(y), x-y \rangle \ge \alpha_0|x-y|^2
\end{equation}
\end{proposition}
\begin{proof}
Note that $\varphi$, $\varphi'$ and $\varphi''$ are bounded. The function $\psi$ is bounded so
$$ 2V(\theta) = \int (\varphi(\langle \theta, u \rangle) - v)^2 P(du,dv) + \lambda|\theta|^2 \sim \lambda|\theta|^2 \quad \text{ as } |\theta| \rightarrow \infty ,$$
so $V$ is coercive. Moreover, we have
\begin{align*}
\nabla V = \int u \varphi'(\langle \theta, u \rangle) (\varphi(\langle \theta, u \rangle) - v) P(du,dv) + \lambda \theta
\end{align*}
so $\nabla V(\theta) \sim \lambda\theta$ as $|\theta| \rightarrow \infty$ and $|\nabla V|^2 \le CV$.
Then, let us assume that the support of $\varphi$ is included in $[-1,1]$, that $u$ has its values in $\textbf{B}(0,1)$ and $v$ in $[-1,1]$. Then the set $ \lbrace u \in \textbf{B}(0,1), \ |\langle \theta,u \rangle| < 1 \rangle $ has Lebesgue measure no larger than $C/|\theta|$ so
$$ \left\|\nabla^2 \int (\varphi(\langle \theta, u \rangle) - v)^2 P(du,dv) \right\| \le C/|\theta|,$$
so outside the compact set $\lbrace |\theta| \le 2C/\lambda \rbrace$, we have $\| \nabla ^2 V \| \ge \lambda/2$ which guarantees \eqref{eq:NN_condition}.
\end{proof}

However, we cannot directly extend this proposition to multi-layers neural networks. Nevertheless, if we consider that the training stops if a parameter becomes too large and if we replace $\psi(\theta,u)$ by $\psi(\phi(\theta),u)$ where $\phi : \mathbb{R} \to \mathbb{R}$ is a smooth approximation of $x \mapsto \min(x,R)\mathds{1}_{x \ge 0} + \max(x,-R)\mathds{1}_{x<0}$ where $R>0$ is large and where $\phi$ is applied in order to avoid over-fitting coordinate by coordinate then the resulting potential $V$ with quadratic regularization satisfies \eqref{Eq:eq:V_assumptions} and \eqref{eq:NN_condition}.

\subsection{Practitioner's corner: choices for $\sigma$}
\label{subsec:practitioner}

In this section we briefly present general choices for the non-constant matrix $\sigma$ that are often used in the Stochastic Optimization and Machine Learning literature.

\cite{welling2011} introduced the Stochastic Gradient Langevin Dynamics (SGLD) with constant preconditioner matrix $\sigma$.
\cite{dauphin2015} adapted the well-known Newton method, which consists in considering $\sigma \sigma^\top = (\nabla^2 V)^{-1}$, to SGLD. Since the size of the Hessian matrix may be too large in practice, because inverting it is computationally costly and because the Hessian matrix may not be positive in every point, it is suggested to consider instead $|\text{diag}((\nabla^2V))^2|^{-1/2}$.
However, computing high-order derivatives may be cumbersome; \cite{simsekli2016} adapts the quasi-Newton method \cite{nocedal2006} to approximate the Hessian matrix to SGLD, yielding the Stochastic Quasi-Newton Langevin algorithm.

\cite{duchi2011} and \cite{li2015} give algorithms where the choice for $\sigma$ is $\sigma \simeq \text{diag}((\lambda + |\nabla V|)^{-1})$, where $\lambda >0$ guarantees numerical stability.
The idea of using geometry has been explored in \cite{patterson2013}, where $\sigma^{-2}$ defines the local curvature of a Riemannian  manifold, giving the Stochastic Gradient Riemaniann Langevin Dynamics algorithm where $\sigma$ is equal to $\mathcal{I}_x^{-1/2}$ where $\mathcal{I}_x$ is the Fischer information matrix, or to some other choices (see \cite[Table 1]{patterson2013}) as $\mathcal{I}_x$ may be intractable.
\cite{ma2015} extends the previous algorithm to Hamiltonian Monte Carlo methods, where a momentum variable is added in order to take into account the "inertia" of the trajectory, yielding the Stochastic Gradient Riemannian Hamiltonian Monte Carlo method.

Allowing the matrix $\sigma$ to depend on the position yields a faster convergence; we refer to the previous references where the simulations prove that these new methods greatly improve classical stochastic gradients algorithms. In particular, we refer to the simulations \cite[Figure 2]{simsekli2016}, \cite[Figure 2]{patterson2013} and \cite[Figure 3]{ma2015} where the different methods based on multiplicative noise are compared.

\section{Langevin equation with constant time coefficient}
\label{sec:langevin}

In this section, we consider the following $\mathbb{R}^d$-valued homogeneous SDE:
\begin{align}
\label{eq:X_SDE}
& X_0^x = x \in \mathbb{R}^d, \quad dX_t^x = b_a(X_t^x)dt + a\sigma(X_t^x)dW_t,
\end{align}
with $a \in(0,A]$ and where $b_a$ is defined in \eqref{eq:def_b}. The drift is specified in such a way that the Gibbs measure $\nu_{a}$ defined in \eqref{eq:def_nu} is the unique invariant distribution of $(X_t^x)$ (see \cite[Proposition 2.5]{pages2020}).

\subsection{Exponential contraction property}

We now prove contraction properties of the SDE \eqref{eq:X_SDE} under the uniform convex setting on the whole $\mathbb{R}^d$ or outside a compact set \eqref{Eq:eq:V_confluence}.
If the uniform dissipative assumption holds on $\mathbb{R}^d$ then we have the following contraction property.
\begin{proposition}
Let $Z$ be the solution of
\begin{align*}
& Z_0^x = x \in \mathbb{R}^d, \quad dZ_t^x = b^Z(Z_t^x)dt + \sigma^Z(Z_t^x) dW_t,
\end{align*}
where the coefficients $b^Z$ and $\sigma^Z$ are Lipschitz continuous.
Assume the uniform convexity i.e. there exists $\alpha >0$ such that
\begin{equation}
\label{eq:uniform_dissipative}
\forall x, y \in \mathbb{R}^d, \langle b^Z(x)-b^Z(y), \ x-y \rangle + \frac{1}{2}|| \sigma^Z(x)-\sigma^Z(y) ||^2 \le - \alpha |x-y|^2 .
\end{equation}
Then:
$$ \forall x,y \in \mathbb{R}^d, \ \mathcal{W}_1\left(\left[Z^x_t\right], \left[Z^y_t\right]\right) \le C|x-y|e^{-\alpha t} .$$
\end{proposition}
\begin{proof}
By the It\={o} lemma, $t \mapsto e^{2\alpha t}\left|Z^x_t - Z^y_t\right|^2$ is a super-martingale, so
$$ \mathbb{E} \left|Z^x_t - Z^y_t\right|^2 \le e^{-2\alpha t}|x-y|^2 ,$$
which yields the desired result.
\end{proof}

This proposition can be applied to $X$ under the assumption
$$ \forall x, y \in \mathbb{R}^d, \langle b_a(x)-b_a(y), \ x-y \rangle + \frac{a^2}{2}|| \sigma(x)-\sigma(y) ||^2 \le - \alpha |x-y|^2, $$
which may be hard to check because of the dependence in $a$.
In \cite[Corollary 2.4]{pages2020} is proved that this contraction property is still true under the uniform convexity outside a compact set \eqref{Eq:eq:V_confluence}. We make this statement more precise by expliciting the dependence in $a$.

\begin{theorem}
\label{thm:confluence}
Under the assumption \eqref{Eq:eq:V_confluence},
\begin{enumerate}[label=(\alph*)]
\item For every $x$, $y \in \mathbb{R}^d$,
\begin{align}
\mathcal{W}_1\left(\left[X^x_t\right], \left[X^y_t\right]\right) \le C e^{C_1/a^2} |x-y|e^{-\rho_a t}, \quad \rho_a := e^{-C_2/a^2}
\owntag[eq:W_confluence]{$\mathcal{P}_{cf}$}
\end{align}
where the constants $C$, $C_1$, $C_2$ do not depend on $a$.
\item For every $x \in \mathbb{R}^d$,
\begin{align}
\label{eq:confluence_nu_a}
& \mathcal{W}_1\left(\left[X^x_t\right],\nu_a\right) \le C e^{C_1/a^2} e^{-\rho_a t}\mathcal{W}_1(\delta_x, \nu_a).
\end{align}
\end{enumerate}
\end{theorem}
\begin{proof}
$(a)$ We refine the proof of \cite[Theorem 2.6]{wang2020} to enhance the dependence of the constants in the parameter $a$. First we remark as in \cite[Section 4.5]{pages2020} in the proof of Corollary 2.4, that Assumption (2.17) of \cite{wang2020}, stating that there exist constants $K_1$, $K_2$ and $r_0>0$ such that, with $\ubar{\sigma} := \sqrt{\sigma \sigma^\top - \ubar{\sigma}_0^2 I_d}$,
\begin{align}
& \frac{a^2}{2}\|\ubar{\sigma}(x)-\ubar{\sigma}(y)\|^2 - \frac{|a^2(\sigma(x)-\sigma(y))^\top (x-y)|^2}{2|x-y|^2} + \langle b_a(x) - b_a(y), x-y \rangle \nonumber \\
& \quad \le \left( (K_1+K_2)\mathds{1}_{|x-y|\le r_0}-K_2 \right)|x-y|^2, \ x, \ y \in \mathbb{R}^d,
\label{eq:wang_2.17}
\end{align}
is true, since $a \in (0,A]$ and $\sigma\sigma^\top$ bounded, as soon as there exist positive constants $\widetilde{K}_1$, $\widetilde{K}_2$ and $R_1$ such that
$$ \forall x, y \in \mathbb{R}^d, \ \langle b_a(x) - b_a(y), x-y \rangle \le \widetilde{K}_1 \mathds{1}_{|x-y|\le R_1} - \widetilde{K}_2|x-y|^2 ,$$
which is, up to changing the positive constants, equivalent to
$$ \forall x, y \in \mathbb{R}^d, \ \langle b_0(x) - b_0(y), x-y \rangle \le \widetilde{K}_1 \mathds{1}_{|x-y|\le R_1} - \widetilde{K}_2|x-y|^2 ,$$
which is in turn equivalent to \eqref{Eq:eq:V_confluence}.
Then we repeat the argument leading to (4.3) in \cite{wang2020}. We reformulate the assumption of ellipticity \eqref{eq:ellipticity} as:
$$ dX_t = b_a(X_t) dt + a(\ubar{\sigma}(X_t)dW^1_t + \ubar{\sigma}_0 dW^2_t) ,$$
where $\ubar{\sigma} \ge 0$ and where $(W^1_t)$ and $(W^2_t)$ are two independent Brownian motions in $\mathbb{R}^d$ (which can be expressed in terms of $W$). For $x \ne y$, let $X^x$ be the solution of this SDE with $X_0=x$ and let $Y^y$ solve the following coupled SDE for $Y^y_0 = y$ :
$$ dY^y_t = b_a(Y^y_t)dt + a\ubar{\sigma}(Y^y_t)dW^1_t + a\ubar{\sigma}_0 \left( dW^2_t - 2 \frac{\langle X^x_t-Y^y_t, dW^2_t\rangle(X^x_t-Y^y_t)}{|X^x_t-Y^y_t|^2} \right) .$$
The process $Y^y$ is in fact defined by orthogonally symmetrizing the component of the noise in $W^2$ w.r.t. $X^x_t-Y^y_t$ at every instant t.
This SDE has a unique solution up to the coupling time
$$ T_{x,y} := \inf \lbrace t \ge 0 : \ X^x_t = Y^y_t \rbrace ,$$
and for $t \ge T_{x,y}$ we set $Y^y_t = X^x_t$. Then $Y^y$ has the same distribution as $X^y$ i.e. is a weak solution of \eqref{eq:X_SDE} with starting value $y$ and it follows from \eqref{eq:wang_2.17} and from the It\=o formula applied to $|X^x-Y^y|$ that for every $0 \le u \le t \le T_{x,y}$,
$$ |X^x_t-Y^y_t| - |X^x_u-Y^y_u| \le M_t - M_u + \int_u^t \left((K_1 + K_2)\mathds{1}_{|X^x_s-Y^y_s|\le r_0} - K_2 \right)|X^x_s-Y^y_s|ds, $$
where
$$ M_t = \int_0^t \frac{a\langle 2 \ubar{\sigma}_0 dW^2_s + (\ubar{\sigma}(X_s) - \ubar{\sigma}(Y^y_s))dW^1_t, X^x_s-Y^y_s \rangle }{|X^x_s-Y^y_s|} $$
is a true Brownian martingale with bracket process satisfying
\begin{equation}
\label{eq:wang:M_inequality}
\langle M \rangle_t \ge 4 a^2 \ubar{\sigma}_0^2 t .
\end{equation}
We now set, still like in the proof of (4.3) in \cite{wang2020},
$$ p_t := \left|X^x_t-Y^y_t\right| \quad \text{and} \quad \bar{p}_t := \varepsilon p_t + 1 - e^{-Np_t},$$ where
$$ N := \frac{r_0}{a^2\ubar{\sigma}_0^2}(K_1+K_2) \quad \text{ and } \quad \varepsilon := Ne^{-Nr_0} .$$
Then we have :
$$ \varepsilon p_t \le \bar{p}_t \le (N+\varepsilon)p_t, \quad \text{and} \quad \forall r \in [0,r_0), \ \frac{2N^2}{r(\varepsilon e^{Nr}+N)} \ge \frac{K_1+K_2}{a^2\ubar{\sigma}_0^2} .$$
Then using \eqref{eq:wang:M_inequality} we derive for all $0 \le u \le t \le T_{x,y}$:
\begin{align*}
& \bar{p}_t - \bar{p}_u \le \int_u^t (\varepsilon + Ne^{-Np_s})dM_s + \int_u^t (\varepsilon + Ne^{-Np_s})\left((K_1+K_2)\mathds{1}_{p_s\le r_0} - K_2 - \frac{2N^2 a^2 \ubar{\sigma}_0^2}{p_s(\varepsilon e^{Np_s}+N)} \right) p_s ds \\
& \quad \le \tilde{M}_t - \tilde{M}_u - K_2 \int_u^t (\varepsilon+Ne^{-Np_s})p_s ds \le \tilde{M}_t - \tilde{M}_u - \varepsilon K_2 \int_u^t p_s ds \le \tilde{M}_t - \tilde{M}_u - \frac{\varepsilon K_2}{N+\varepsilon} \int_u^t \bar{p}_s ds.
\end{align*}
So that we have
$$ \mathbb{E}[\bar{p}_t - \bar{p}_u] = \mathbb{E}[(\bar{p}_t - \bar{p}_u) \mathds{1}_{t \le T_{x,y}}] \le - \frac{\varepsilon K_2}{N+\varepsilon} \int_u^t \mathbb{E}\bar{p}_s ds ,$$
so that
$$ \frac{d}{dt} \mathbb{E}[\bar{p}_t] \le - \frac{\varepsilon K_2}{N+\varepsilon} \mathbb{E}[\bar{p}_t] $$
and then
$$ \mathbb{E} \bar{p}_t \le \bar{p}_0 e^{-\frac{\varepsilon K_2}{N+\varepsilon} t}. $$
Noting that $\bar{p}_0 \le (N+\varepsilon)|x-y|$, we have
$$ \mathbb{E} p_t \le \frac{N+\varepsilon}{\varepsilon}|x-y| e^{-\frac{\varepsilon K_2}{N+\varepsilon} t},.$$
so that
$$ \mathcal{W}_1\left(\left[X^x_t\right],\left[X^y_t\right]\right) \le \frac{N+\varepsilon}{\varepsilon}|x-y|e^{-\frac{\varepsilon K_2}{N+\varepsilon}t} \le C e^{C_1/a^2} |x-y| e^{-e^{-C_2/a^2} t} .$$

\medskip

\noindent $(b)$ As $\nu_a$ is the invariant distribution of the diffusion \eqref{eq:X_SDE}, using \eqref{Eq:eq:W_confluence} we have
\begin{align*}
\mathcal{W}_1\left(\left[X^x_t\right],\nu_a\right) & = \int_{\mathbb{R}^d} \mathcal{W}_1\left(\left[X^x_t\right],\left[X^y_t\right]\right) \nu_a(dy) \le C e^{C_1/a^2} e^{-\rho_a t} \int_{\mathbb{R}^d} |x-y| \nu_a(dy) \\
& \le C e^{C_1/a^2}e^{-\rho_a t} \mathcal{W}_1(\delta_x, \nu_a) .
\end{align*}
%
%
\end{proof}

\subsection{Time schedule and Wasserstein distance between Gibbs measures}

For $C_{(T)}>0$ and for $\beta>0$, let us define the time schedule that will be used for the plateau SDE in the next section:
\begin{equation}
\label{eq:def_T_n}
T_n := C_{(T)}n^{1+\beta},
\end{equation}
and by a slight abuse of notation we define
\begin{equation}
\label{eq:def_a_n}
a_n := a(T_n) = \frac{A}{\sqrt{\log(T_n+e)}} \quad \text{and} \quad \rho_n := \rho_{a_n} = e^{-C_2/a_n^2}.
\end{equation}

\begin{lemma}
\label{lemma:app:a_n_diff}
The sequence $a_n = A \log^{-1/2}(T_n+e)$ satisfies
\begin{equation}
\label{eq:a_n_diff}
0 \le a_n - a_{n+1} \asymp (n \log^{3/2}(n))^{-1}.
\end{equation}
\end{lemma}
\begin{proof}
One straightforwardly checks that
\begin{align*}
a_n - a_{n+1} \sim -\frac{d}{dn}\left( \frac{A}{\sqrt{\log(C_{(T)}n^{1+\beta}{+}e)}} \right) = \frac{A\beta}{2\log^{3/2}(C_{(T)}n^{1+\beta}{+}e)\left(n{+}e/(C_{(T)}n^\beta)\right)} \asymp \frac{1}{n\log^{3/2}(n)}.
\end{align*}
\end{proof}

We prove the following result that will be useful to study the convergence of the plateau SDE.
\begin{proposition}
\label{prop:W_nu}
Let $\nu_a$, $a \in (0,A]$ be the Gibbs measure defined in \eqref{eq:def_nu}. Assume that $V$ is coercive, that $(x \mapsto |x|^2 e^{-2V(x)/A^2}) \in L^1(\mathbb{R}^d)$ and \eqref{Eq:eq:min_V}. Then for $n \in \mathbb{N}$,
$$ \mathcal{W}_1(\nu_{a_n},\nu_{a_{n+1}}) \le \frac{C}{n \log^{3/2}(n)} .$$
Moreover, for every $s$, $t \in [a_{n+1},a_n]$, we have
$$ \mathcal{W}_1(\nu_s,\nu_t) \le \frac{C}{n \log^{3/2}(n)} .$$
\end{proposition}
The proof of this proposition is given in the Supplementary Material. It relies on the following lemma.
\begin{lemma}
\label{lemma:acceptance_rejection}
Let $\mu$ and $\nu$ be two probability distributions on $\mathbb{R}^d$ with  densities $f$ and $g$ respectively with finite moments of order $p$. Assume that there exists $M \ge 1$ such that $f \le Mg$. Then
$$ \mathcal{W}_p(\mu,\nu)^p \le \mathbb{E}|X-Y|^p - \frac{1}{M}\mathbb{E}|X-\tilde{X}|^p,$$
where $X$ and $\tilde{X} \sim \mu$, $Y \sim \nu$ and $X$, $\tilde{X}$ and $Y$ are mutually independent.
\end{lemma}
\begin{proof}
We define a coupling on $\mu$ and $\nu$ inspired from the acceptance rejection sampling as follows. Let $X \sim \mu$, $Y \sim \nu$, $U \sim \mathcal{U}([0,1])$ and $X$, $Y$, $U$ are independent, and let
$$X'= Y \mathds{1}\lbrace U \le f(Y)/(Mg(Y))\rbrace + X \mathds{1}\lbrace U > f(Y)/(Mg(Y))\rbrace .$$
Then adapting the proof of the acceptance rejection method, $X' \sim \mu$ and we have:
\begin{align*}
\mathbb{E}|X'-Y|^p & = \mathbb{E}|Y-X|^p \mathds{1}\lbrace U > f(Y)/(Mg(Y))\rbrace \\
& = \int_{(\mathbb{R}^d)^2} |y-x|^p \left(\int_0^1 \mathds{1}\lbrace u > f(y)/(Mg(y))\rbrace du\right) f(x)g(y)dxdy \\
& = \int_{(\mathbb{R}^d)^2} |y-x|^p f(x)g(y)dxdy - \frac{1}{M} \int_{(\mathbb{R}^d)^2} |y-x|^p f(x)f(y)dx dy \\
& = \mathbb{E}|X-Y|^p - \frac{1}{M}\mathbb{E}|X-\tilde{X}|^p.
\end{align*}
\end{proof}

\begin{lemma}
\label{lemma:W_nu_a_nu_star}
We have
\begin{equation}
\mathcal{W}_1(\nu_{a_n}, \nu^\star) \le Ca_n .
\end{equation}
\end{lemma}
\begin{proof}
First let us prove that $\mathcal{W}_1(\nu_a,\nu^\star) \to 0$ as $a \to 0$. By Proposition \ref{prop:W_nu} and using that
$$ \textstyle \sum_{n \ge 2} (n\log^{3/2}(n))^{-1}< \infty ,$$
$(\nu_{a_n})$ is a Cauchy sequence in $(L^1(\mathbb{R}^d),\mathcal{W}_1)$ so converges to some limit measure $\widetilde{\nu}$.
But $(\nu_{a_n})$ also weakly converges to $\widetilde{\nu}$, so $\widetilde{\nu} = \nu^\star$.
Moreover, $\mathcal{W}_1(\nu_{a_n},\nu^\star)$ is bounded by the tail of the above series, which is of order $\log^{-1/2}(n)$.
\end{proof}

\section{Plateau case}
\label{sec:X}

We define $(X_t)$ as the solution the following SDE where the coefficients piecewisely depend on the time; $X$ is then said to be "by plateaux":
\begin{align}
\label{eq:def_X}
& X_0^{x_0} = x_0, \quad dX_t^{x_0} = b_{a_{k+1}}(X_t^{x_0})dt + a_{k+1} \sigma(X_t^{x_0})dW_t, \quad t \in [T_k,T_{k+1}],
\end{align}
where $b_a$ is defined in \eqref{eq:def_b}, $(T_n)$ is defined in \eqref{eq:def_T_n} and $(a_n)$ is defined in \eqref{eq:def_a_n}.
We note that although the coefficients are not continuous, the process $(X_t^{x_0})$ is well defined as it is the continuous concatenation of the solutions of the equations on the intervals $[T_k, T_{k+1}]$.
More generally, we define $(X^{x,n}_t)$ as the solution of
\begin{align*}
X^{x,n}_0 &= x, \quad dX_t^{x,n} = b_{a_{k+1}}(X_t^{x,n})dt + a_{k+1} \sigma(X_t^{x,n})dW_t, \quad t \in [T_k-T_n,T_{k+1}-T_n], \ k \ge n,
\end{align*}
i.e. $(X_t^{x,n})$ has the law of $(X_{T_n+t})_{t \ge 0}$ conditionally to $X_{T_n}=x$. We have $X_{t}^x = X_{t}^{x,0}$.

\begin{theorem}
\label{thm:conv_X}
Let $X$ be defined in \eqref{eq:def_X}. If
\begin{equation}
\label{eq:hyp_A}
A > \max \left( \sqrt{(1+\beta^{-1})C_2}, \sqrt{(1+\beta)C_1} \right),
\end{equation}
where $C_1$ and $C_2$ are given in in \eqref{Eq:eq:W_confluence}, then for every $x_0 \in \mathbb{R}^d$:
$$ \mathcal{W}_1([X^{x_0}_t], \nu^\star) \underset{t \rightarrow \infty}{\longrightarrow} 0.$$
More precisely, for $t \ge 0$ we have:
$$ \mathcal{W}_1([X^{x_0}_t], \nu^\star) \le Ca(t)(1+|x_0|), $$
for all $C' < C_{(T)}$, for all large enough $n \ge n(C_{(T)}')$, on the time schedule $(T_n)$ we have
$$ \mathcal{W}_1([X^{x_0}_{T_{n}}], \nu_{a_{n}}) \le Cn^{-1+(\beta+1)C_1/A^2} e^{-(C')^{1-C_2/A^2}(\beta+1)n^{\beta-(\beta+1)C_2/A^2}}(1+|x_0|) $$
and we have
\begin{equation*}
\mathcal{W}_1([X^{x_0}_t],\nu_{a(t)}) \le \frac{C(1+|x_0|)}{t^{(\beta+1)^{-1}-C_1/A^2} \log^{3/2}(t)}.
\end{equation*}
\end{theorem}
\begin{proof}
%
For fixed $x \in \mathbb{R}^d$ and using Theorem \ref{thm:confluence} we have:
$$ \mathcal{W}_1( [X^{x,n}_{T_{n+1}-T_n} ], \nu_{a_{n+1}}) \le C e^{C_1/a_{n+1}^2} e^{-\rho_{a_{n+1}} (T_{n+1}-T_n)} \mathcal{W}_1(\delta_x,\nu_{a_{n+1}}). $$
So integrating $x$ with respect to the law of $X^{x_0}_{T_n}$ (and using the existence of the optimal coupling, see for example \cite[Proposition 1.3]{wang2012}) yields:
\begin{equation}
\label{eq:proof:X:1}
\mathcal{W}_1([X^{x_0}_{T_{n+1}}], \nu_{a_{n+1}}) \le C e^{C_1/a_{n+1}^2} e^{-\rho_{a_{n+1}} (T_{n+1}-T_n)} \left( \mathcal{W}_1([X^{x_0}_{T_n}],\nu_{a_n}) + \mathcal{W}_1(\nu_{a_n},\nu_{a_{n+1}}) \right) .
\end{equation}
Iterating this relation yields
\begin{align}
\mathcal{W}_1([X^{x_0}_{T_{n+1}}], \nu_{a_{n+1}}) & \le \mu_{n+1} \mathcal{W}_1(\nu_{a_n},\nu_{a_{n+1}}) + \mu_{n+1} \mu_n \mathcal{W}_1(\nu_{a_{n-1}}, \nu_{a_n}) + \cdots + \mu_{n+1} \cdots \mu_1 \mathcal{W}_1(\nu_{a_0},\nu_{a_1}) \nonumber \\
& \quad + \mu_{n+1} \cdots \mu_1 \mathcal{W}_1(\delta_{x_0}, \nu_{a_0}).
\label{eq:proof:X:2}
\end{align}
where
\begin{align}
\mu_n & := C e^{C_1/a_n^2}e^{-\rho_{a_n}(T_n-T_{n-1})} = C (T_n+e)^{C_1/A^2} e^{-(T_n+e)^{-C_2/A^2}(T_n-T_{n-1})} \nonumber \\
& \le C(C_{(T)}n^{\beta+1}+e)^{C_1/A^2} e^{-(C_{(T)}n^{\beta+1}+e)^{-C_2/A^2}C_{(T)}(\beta+1)(n-1)^\beta} \nonumber \\
& \le Cn^{(\beta+1)C_1/A^2} e^{-(C')^{1-C_2/A^2}(\beta+1)n^{\beta-(\beta+1)C_2/A^2}},
\label{eq:def_mu}
\end{align}
where we have used \eqref{eq:def_T_n} and where the last inequality holds for large enough $n$. Note that $\mu_n$ is bounded by a sequence in the form of $n^\delta \exp(-Ln^\eta) = o(n^{-\ell})$ for every $\ell \ge 0$. Owing to \eqref{eq:hyp_A}, we have $\beta - (\beta+1)C_2/A^2 >0$.

On the other hand, if $Z \sim \nu_{a_0}$ then $\mathcal{W}_1(\delta_{x_0}, \nu_{a_0}) = \mathbb{E}|x_0 - Z| \le |x_0| + \mathbb{E}|Z|$.
Plugging this into \eqref{eq:proof:X:2} and using that $\mu_n \to 0$ so is bounded and smaller than $1$ for $n$ large enough and then $(\mu_{n-1} \cdots \mu_k)_{1 \le k \le n-1}$ is bounded ; using Proposition \ref{prop:W_nu} and that $\sum_{n} (n\log^{3/2}(n))^{-1} < \infty$ yields
\begin{align*}
\mathcal{W}_1([X^{x_0}_{T_{n+1}}], \nu_{a_{n+1}}) & \le \mu_{n+1} \mathcal{W}_1(\nu_{a_n},\nu_{a_{n+1}}) + C\mu_{n+1}\mu_n \left( \mathcal{W}_1(\nu_{a_{n-1}}, \nu_{a_n}) + \cdots + \mathcal{W}_1(\nu_{a_0}, \nu_{a_{1}}) \right) \\
& \quad + C\mu_{n+1}\mu_n\mathcal{W}_1(\delta_{x_0},\nu_{a_0}) \\
& \le \mu_{n+1} \mathcal{W}_1(\nu_{a_n},\nu_{a_{n+1}}) + C\mu_{n+1} \mu_n + C\mu_{n+1}\mu_n (1+|x_0|) \\
& \le C \frac{\mu_{n+1}}{n\log^{3/2}(n)}(1+|x_0|) \le C \mu_{n+1} a_{n+1}(1+|x_0|),
\end{align*}
where we used that $\mu_n = o(\mathcal{W}_1(\nu_{a_n},\nu_{a_{n+1}}))$.
Then using Lemma \ref{lemma:W_nu_a_nu_star} we have
$$ \mathcal{W}_1([X^{x_0}_{T_{n+1}}],\nu^\star) \le \mathcal{W}_1([X^{x_0}_{T_{n+1}}], \nu_{a_{n+1}}) + \mathcal{W}_1(\nu_{a_{n+1}}, \nu^\star) \le Ca_n (1+|x_0|) ,$$
where we used once again $\mu_n \to 0$.

\medskip

Now, let us prove that $\mathcal{W}_1([X^{x_0}_t],\nu^\star) \to  0$ as $t \to \infty$. For $t \in [0,T_{n+1}-T_n)$ we integrate \eqref{eq:confluence_nu_a} with respect to the law of $X^{x_0}_{T_n}$, giving
\begin{align}
\mathcal{W}_1([X^{x_0}_{T_n+t}], \nu_{a_{n+1}}) & \le Ce^{C_1 a_{n+1}^{-2}} e^{-\rho_{a_{n+1}} t} \mathcal{W}_1([X^{x_0}_{T_n}], \nu_{a_{n+1}}) \nonumber \\
& \le Ce^{C_1 a_{n+1}^{-2}}\left( \mathcal{W}_1([X^{x_0}_{T_n}], \nu_{a_n}) + \mathcal{W}_1(\nu_{a_n}, \nu_{a_{n+1}}) \right) \nonumber \\
& \le Ce^{C_1 a_{n+1}^{-2}} \mathcal{W}_1(\nu_{a_n}, \nu_{a_{n+1}})(1+|x_0|) \nonumber \\
& \le \frac{C(1+|x_0|)}{n^{1-(\beta + 1)C_1/A^2} \log^{3/2}(n)}.
\label{eq:proof:X_t_nu}
\end{align}
Now, for $t \ge 0$, let $n$ be such that $t \in [T_n,T_{n+1})$. Then $(n+1)\ge t^{1/(\beta+1)}$ and
\begin{align}
\mathcal{W}_1([X^{x_0}_{t}], \nu_{a(t)}) \le \mathcal{W}_1([X^{x_0}_{t}], \nu_{a_{n+1}}) + \mathcal{W}_1(\nu_{a_{n+1}}, \nu_{a(t)}) \le \frac{C(1+|x_0|)}{t^{(\beta+1)^{-1}-C_1/A^2} \log^{3/2}(t)},
\label{eq:proof:X_t_nu:2}
\end{align}
where we used the second claim of Proposition \ref{prop:W_nu}.

Furthermore owing to \eqref{eq:hyp_A} we have $(\beta + 1)C_1/A^2 < 1$, so that
$$ \mathcal{W}_1([X^{x_0}_{T_n+t}],\nu^\star) \le \mathcal{W}_1([X^{x_0}_{T_n+t}], \nu_{a_{n+1}}) + \mathcal{W}_1(\nu_{a_{n+1}},\nu^\star) \le Ca_n(1+|x_0|). $$

%
%
\end{proof}

\begin{remark}
We find again the classic schedule $a(t)$ of order $\log^{-1/2}(t)$. If for example we choose instead $a_n = \log(T_n)^{-(1+\varepsilon)/2}$ for some $\varepsilon > 0$, then we obtain
$$ \log(\mu_1 \cdots \mu_n) = n\log(C) + \frac{C_1}{A^2} \sum_{k=1}^n \log^{1+\varepsilon}(T_k) - \sum_{k=1}^n \frac{T_k-T_{k-1}}{T_k^{\log^\varepsilon(T_k) C_2/A^2}} . $$
Hence, as $T_n - T_{n-1} = o(T_n^{\log^\varepsilon(T_n) C_2/A^2})$, $\mu_1 \cdots \mu_n$ does not converge to $0$ whatever the value of $A>0$ is.
\end{remark}

\section{Continuously decreasing case}
\label{sec:Y}

We now consider $(Y_t)$ solution to \eqref{eq:def_Y} i.e. the Langevin equation where the time coefficient $a(t)$ before $\sigma$ is continuously decreasing. More generally, since $Y$ is solution to a non-homogeneous SDE, we define for every $x \in \mathbb{R}^d$ and for every fixed $u \ge 0$:
\begin{align}
\label{eq:def_Y:2}
Y_{0,u}^{x} & = x, \quad dY_{t,u}^{x} = b_{a(t+u)}(Y_{t,u}^{x})dt + a(t+u) \sigma(Y_{t,u}^{x}) dW_t,
\end{align}
so that $Y^x = Y^x_{\cdot, 0}$.
We define the kernel associated to $Y$ between the times $t$ and $t+u$ as $P^Y_{t,u}$ such that for all $f : \mathbb{R}^d \to \mathbb{R}^+$ measurable, $P^Y_{t,u} f(x) = \mathbb{E}[f(Y^x_{t,u})]$.
We also consider $X$ as defined in \eqref{eq:def_X} and its associated kernel denoted as $P^{X,n}_t$ such that for every $f: \mathbb{R}^d \to \mathbb{R}^+$ measurable, $P^{X,n}_t f(x) = \mathbb{E}[f(X^{x,n}_t)]$.

\subsection{Boundedness of the potential}

\begin{lemma}
\label{lemma:D.1a:cont}
Let $p>0$. Then there exists $C>0$ such that for every $n \ge 0$, for every $u \ge 0$ and for every $x \in \mathbb{R}^d$:
$$  \sup_{t \ge 0} \mathbb{E} V^p(X_t^{x,n}) \le CV^p(x) \ \text{ and } \ \sup_{t \ge 0} \mathbb{E} V^p(Y^x_{t,u}) \le C V^p(x) .$$
\end{lemma}
\begin{proof}
By the It\=o Lemma, we have for $k \ge n$ and for $t \in [T_k-T_n,T_{k+1}-T_n)$:
\begin{align*}
dV^p(X^{x,n}_t) & = p\nabla V(X^{x,n}_t)^\top \cdot V^{p-1}(X^{x,n}_t) \left( -\sigma \sigma^\top (X^{x,n}_t) \nabla V(X^{x,n}_t) + a_{k+1}^2 \Upsilon(X^{x,n}_t)\right) dt \\
& \quad  +  p\nabla V(X^{x,n}_t)^\top \cdot V^{p-1}(X^{x,n}_t) a_{k+1} \sigma(X^{x,n}_t) dW_t \\
& \quad + \frac{p}{2} \left(\nabla^2V(X^{x,n}_t)V^{p-1}(X^{x,n}_t) + (p-1)|\nabla V(X^{x,n}_t)|^2 \cdot V^{p-2}(X^{x,n}_t)\right)a_{k+1}^2 \sigma \sigma^\top(X^{x,n}_t) dt.
\end{align*}
Using the facts that $(a_k)$, $\Upsilon$, $\sigma$, $\nabla^2 V$ are bounded, that $|\nabla V| \le CV^{1/2}$ and that $V$, $|\nabla V|$ are coercive and $\sigma \sigma^\top \ge \ubar{\sigma}_0^2 I_d$, there exists $R>0$ such that if $|X^{x,n}_t| \ge R$ then the coefficient of $dt$ in the last equation is bounded above by
\begin{align*}
& pV^{p-1} \nabla V(X^{x,n}_t)^T \cdot \left( -\ubar{\sigma}_0^2 (X^{x,n}_t)\nabla V(X^{x,n}_t) + C \right) + CV^{p-1}(X^{x,n}_t)||\sigma||_\infty^2 \le 0,
\end{align*}
so that
$$ \mathbb{E}[V^p(X^{x,n}_t)] \le \max\left(\sup_{|z|\le R} V^p(z) , V^p(x) \right) .$$
The proof is the same for $Y$, replacing $a_{k+1}$ by $a(t)$.
\end{proof}

\subsection{Strong and weak error bounds}

In this subsection we adapt the proofs to bound weak and strong errors from \cite{pages2020} while paying attention to the dependence in $a_n$.

\begin{lemma}
\label{lemma:3.4.b:Y}
Let $p \ge 1$ and let $\bar{\gamma} > 0$. There exists $C>0$ such that for all $n \ge 0$, $u$, $t\ge 0$ such that $u \in [T_n,T_{n+1}]$, $u+t \in [T_{n},T_{n+1}]$ and $t \le \bar{\gamma}$,
$$|| X_{t}^{x,n} - Y_{t,u}^x ||_p \le C \sqrt{t} (a_n - a_{n+1}). $$
\end{lemma}
\begin{proof}
We first consider the case $p \ge 2$. Noting that $a_{n+1} \le a(u+s) \le a_{n}$ for all $s \in [0,t]$ and using Lemma \ref{lemma:BDG} in the Appendix, with in mind that $b_a = b_0 + a^2 \Upsilon$, we have
\begin{align*}
\| X_t^{x,n} - Y_{t,u}^x \|_p & \le \left\| \int_0^t (b_{a_{n+1}}(X_s^{x,n}) - b_{a(u+s)}(Y_{s,u}^x)) ds \right\|_p + \left\| \int_0^t (a_{n+1} \sigma(X_s^{x,n}) {-} a(u+s)\sigma(Y_{s,u}^x)) dW_s \right\|_p \\
& \le [b]_{\text{Lip}} \int_0^t ||X_s^{x,n} - Y_{s,u}^x||_p ds + \int_0^t ||a_{n+1}^2 \Upsilon(X_s^{x,n}) - a(u+s)^2 \Upsilon(Y_{s,u}^x)||_p ds \\
& \quad + C^{BDG}_p a_{n+1} [\sigma]_{\text{Lip}}\left( \int_0^t ||X_s^{x,n} - Y_{s,u}^x||^2_p ds \right)^{1/2} + \left|\left| \int_0^t \sigma(Y_{s,u}^x)(a_{n+1} {-} a(u{+}s)) dW_s \right|\right|_p \\
& \le [b]_{\text{Lip}} \int_0^t ||X_s^{x,n} - Y_{s,u}^x||_p ds + ||\Upsilon||_\infty (a_n^2 - a_{n+1}^2)t + a_{n+1}^2 [\Upsilon]_{\text{Lip}}  \int_0^t ||X_s^{x,n} - Y_{s,u}^x||_p ds \\
& \quad + C^{BDG}_p a_{n+1} [\sigma]_{\text{Lip}}\left( \int_0^t ||X_s^{x,n} - Y_{s,u}^x||^2_p ds \right)^{1/2} + ||\sigma||_\infty ||W_1||_p \sqrt{t} (a_n - a_{n+1}),
\end{align*}
where we used the generalized Minkowski inequality.
Set $\varphi(t) := \sup_{0 \le s \le t} ||X_s^{x,n} - Y_{s,u}^x||_p$ and $\psi(t) := ||\Upsilon||_\infty (a_n^2 - a_{n+1}^2)t + ||\sigma||_\infty ||W_1||_p \sqrt{t} (a_n - a_{n+1})$. Both functions are non-decreasing and
$$ \varphi(t) \le \psi(t) + ([b]_{\text{Lip}} + a_{n+1}^2[\Upsilon]_{\text{Lip}}) \int_0^t ||X_s^{x,n} - Y_{s,u}^x||_p ds + C^{BDG}_p a_{n+1} [\sigma]_{\text{Lip}}\left( \int_0^t ||X_s^{x,n} - Y_{s,u}^x||^2_p ds \right)^{1/2} .$$
Moreover, for every $\alpha>0$:
$$ \left( \int_0^t \varphi(s)^2 ds\right)^{1/2} \le \sqrt{\varphi(t)}\sqrt{\int_0^t \varphi(s) ds} \le \frac{\alpha}{2} \varphi(t) + \frac{1}{2\alpha} \int_0^t \varphi(s) ds.$$
Taking $\alpha = \left(C^{BDG}_p a_{n+1} [\sigma]_{\text{Lip}}\right)^{-1}$ yields:
$$ \varphi(t) \le 2 \psi(t) + \left(2[b]_{\text{Lip}} + 2a_{n+1}^2[\Upsilon]_{\text{Lip}} + (C^{BDG}a_{n+1}[\sigma]_{\text{Lip}})^2\right) \int_0^t \varphi(s) ds .$$
So the Gronwall Lemma yields for every $t \in [0,\bar{\gamma}]$
$$ \varphi(t) \le 2 e^{(2[b]_{\text{Lip}} + 2a_{n+1}^2[\Upsilon]_{\text{Lip}} + (C^{BDG}_p a_{n+1}[\sigma]_{\text{Lip}})^2)\bar{\gamma}} \psi(t) ,$$
which completes the proof for $p \ge 2$, noting that $a_n^2 - a_{n+1}^2 \le 2 a_n (a_n - a_{n+1}) = o(a_n-a_{n+1})$.
If $p \in [1,2)$, the inequality is still true remarking that $\| \cdot \|_p \le \| \cdot \|_2$.
\end{proof}

\begin{lemma}
\label{lemma:3.4.a}
Let $p \ge 1$ and let $\bar{\gamma} > 0$. There exists a real constant $C \ge 0$ such that for all $n \ge 0$,
$$ \forall t \in [0, \bar{\gamma}], \ ||X_t^{x,n} - x||_p \le CV^{1/2}(x)\sqrt{t} .$$
\end{lemma}
\begin{proof}
We perform a proof similar to the proof of Lemma \ref{lemma:3.4.b:Y}. For $p \ge 2$ we have
\begin{align*}
& ||X_t^{x,n} - x||_p \le \left\| \int_0^t b_{a_{n+1}}(X_s^{x,n}) ds \right\|_p + \left\| \int_0^t a_{n+1} \sigma(X_s^{x,n}) ds \right\|_p \\
& \le t|b_{a_{n+1}}(x)| + A\|\sigma\|_\infty \|W\|_1 \sqrt{t} + [b]_{\text{Lip}} \int_0^t ||X_s^{x,n} - x||_p ds + A[\sigma]_{\text{Lip}} C_p^{BDG} \left(\int_0^t ||X_s^{x,n} - x||_p^2 ds \right)^{1/2}.
\end{align*}
From here we use the Gronwall Lemma as in the proof of Lemma \ref{lemma:3.4.b:Y}. For $p \in [1,2)$, we have $\| \cdot \|_p \le \| \cdot \|_2$.
\end{proof}

\begin{proposition}
\label{prop:3.5:Y}
Let $\bar{\gamma}>0$. There exists $C>0$ such that for every $g : \mathbb{R}^d \to \mathbb{R}$ being $\mathcal{C}^2$, for every $\gamma \in (0,\bar{\gamma}]$, every $n \ge 0$ and every $u \ge 0$ such that $u \in [T_n,T_{n+1}]$ and $u+\gamma \in [T_n,T_{n+1}]$:
\begin{align*}
& | \mathbb{E}\left[g(Y^x_{\gamma,u})\right] - \mathbb{E}\left[g(X^{x,n}_\gamma)\right]| \le C \gamma (a_n - a_{n+1}) \Phi_g(x) \\
\text{with } \ & \Phi_g(x) = \max\left(|\nabla g(x)|, \left|\left|\sup_{\xi \in (X^{x,n}_\gamma, Y_{\gamma,u}^x)} || \nabla^2 g(\xi) || \right|\right|_2, V^{1/2}(x)\left|\left|\sup_{\xi \in (x, X_\gamma^{x,n})} || \nabla^2 g(\xi) || \right|\right|_2 \right) .
\end{align*}
\end{proposition}
\begin{proof}
By the second order Taylor formula, for every $y$, $z \in \mathbb{R}^d$:
$$ g(z) - g(y) = \langle \nabla g(y) | z-y \rangle + \int_0^1 (1-s)\nabla^2 g(sz + (1-s)y) ds (z-y)^{\otimes 2} .$$
Applying this expansion with $y = X^{x,n}_\gamma$ and $z=Y_{\gamma,u}^x$ yields:
\begin{align}
\mathbb{E}[g(Y_{\gamma,u}^x) - g(X^{x,n}_\gamma)] & = \langle \nabla g(x) | \mathbb{E}[Y_{\gamma,u}^x - X^{x,n}_\gamma] \rangle + \mathbb{E}[\langle \nabla g(X^{x,n}_\gamma) - \nabla g(x), Y_{\gamma,u}^x - X^{x,n}_\gamma \rangle] \nonumber \\
\label{eq:3.5:Y:proof}
& \quad + \int_0^1 (1-s)\mathbb{E}\left[\nabla^2 g(s Y_{\gamma,u}^x + (1-s)X^{x,n}_\gamma)(Y_{\gamma,u}^x-X_\gamma^{x,n})^{\otimes 2} \right] ds .
\end{align}
The first term is bounded by $|\nabla g(x)| \cdot |\mathbb{E}[Y_{\gamma,u}^x - X_\gamma^{x,n}]|$, with
\begin{align*}
|\mathbb{E}[Y_{\gamma,u}^x - X_\gamma^{x,n}]| & = \left|\mathbb{E} \left[ \int_0^\gamma (b_{a(s+u)}(Y_{s,u}^x) {-} b_{a(s+u)}(X_s^{x,n})) ds\right] + \mathbb{E} \left[ \int_0^\gamma (b_{a(u+s)}(X_s^{x,n}) {-} b_{a_{n+1}}(X_s^{x,n})) ds\right]\right| \\
& \le C[b]_{\text{Lip}} (a_n - a_{n+1}) \int_0^\gamma \sqrt{s}ds + ||\Upsilon||_\infty \gamma (a_n^2 - a_{n+1}^2) \le C \gamma (a_n - a_{n+1}),
\end{align*}
where we used Lemma \ref{lemma:3.4.b:Y}.
Using Lemma \ref{lemma:3.4.a} and Lemma \ref{lemma:3.4.b:Y} again, the second term in the right hand side of \eqref{eq:3.5:Y:proof} is bounded by
$$ C \left|\left|\sup_{\xi \in (x, X_\gamma^{x,n})} || \nabla^2 g(\xi) || \right|\right|_2 \sqrt{\gamma}V^{1/2}(x) \sqrt{\gamma} (a_n - a_{n+1}). $$
Using Lemma \ref{lemma:3.4.b:Y}, the third term is bounded by
$$ \frac{1}{2} C\gamma (a_n - a_{n+1})^2 \left|\left|\sup_{\xi \in (X_\gamma^{x,n}, Y_{\gamma,u}^x)} || \nabla^2 g(\xi) || \right|\right|_2  .$$

\end{proof}

\begin{proposition}
\label{prop:3.6:Y}
Let $T$, $\bar{\gamma}>0$. There exists $C>0$ such that for every Lipschitz continuous function $f : \mathbb{R}^d \to \mathbb{R}$ and every $t \in (0,T]$, for all $n \ge 0$, for all $\gamma < \bar{\gamma}$ and every $u \in [T_n,T_{n+1}]$ such that $u+t+\gamma \in [T_n,T_{n+1}]$,
$$ \left| \mathbb{E}\left[P_t^{X,n} f(Y_{\gamma,u}^x)\right]  - \mathbb{E}\left[P_t^{X,n} f(X_\gamma^{x,n})\right] \right| \le C a_{n+1}^{-2}(a_n-a_{n+1}) [f]_{\textup{Lip}} \gamma t^{-1/2} V(x). $$
\end{proposition}
\begin{proof}
We apply Proposition \ref{prop:3.5:Y} to $g_t := P^{X,n}_t f$ with $t >0$. Following \cite[Proposition 3.2(b)]{pages2020} while paying attention to the dependence in the ellipticity parameter $a$, we have
\begin{align*}
\Phi_{g_t}(x) & \le C[f]_{\text{Lip}} a_{n+1}^{-2} t^{-1/2} \max\left(V^{1/2}(x), \left|\left|\sup_{\xi \in (X^{x,n}_\gamma, Y_{\gamma,u}^x)} V^{1/2}(\xi) \right|\right|_2, V^{1/2}(x)\left|\left|\sup_{\xi \in (x, X_\gamma^{x,n})} V^{1/2}(\xi) \right|\right|_2 \right).
\end{align*}
But following \eqref{Eq:eq:V_assumptions}, $\nabla V/V^{1/2}$ is bounded so $V^{1/2}$ is Lipschitz continuous and then
\begin{align*}
& \left|\left|\sup_{\xi \in (x, X_\gamma^{x,n})} V^{1/2}(\xi) \right|\right|_2 \le \left|\left| V^{1/2}(x) + [V^{1/2}]_{\text{Lip}} |X_\gamma^{x,n}-x| \right|\right|_2 \le CV^{1/2}(x) \\
& \left|\left|\sup_{\xi \in (X^{x,n}_\gamma, Y_{\gamma,u}^x)} V^{1/2}(\xi) \right|\right|_2 \le \left|\left| V^{1/2}(x) + [V^{1/2}]_{\text{Lip}} \max(|X^{x,n}_\gamma-x|, |Y^x_{\gamma,u}-x|) \right|\right|_2 \le CV^{1/2}(x),
\end{align*}
where we used Lemmas \ref{lemma:3.4.a} and \ref{lemma:3.4.b:Y}. We thus obtain the desired result.
\end{proof}

\subsection{Proof of Theorem \ref{thm:main}.(a)}
\label{subsec:proof_Y}

\begin{figure}
	\centering
	
\begin{tikzpicture}
  \foreach \x in {0,1,2,...,10}
    {        
      \coordinate (A\x) at ($(0,0)+(\x*0.8cm,0)$) {};
      \draw ($(A\x)+(0,5pt)$) -- ($(A\x)-(0,5pt)$);
    }
  \draw[ultra thick,red] ($(A0)+(0,7pt)$) -- ($(A0)-(0,7pt)$);
  \node[red] at ($(A0)+(0,3ex)$) {$T_n$};
  \draw[ultra thick,red] ($(A7)+(0,7pt)$) -- ($(A7)-(0,7pt)$);
  \node[red] at ($(A10)+(0,3ex)$) {$T_{n+1}$};
  \draw[ultra thick,red] ($(A10)+(0,7pt)$) -- ($(A10)-(0,7pt)$);
  \node[red] at ($(A7)+(0,-3ex)$) {$T_{n+1}-T$};
  \draw[ultra thick] (A0) -- (A10);
  
  \node at ($(A2)+(2ex,-3ex)$) {$\gamma$};
  \draw[>=triangle 45, <->] ($(A2)+(0,-1ex)$) -- ($(A3)+(0,-1ex)$);

\end{tikzpicture}

	\caption{\textit{Intervals for the domino strategy.}}
	\label{fig:domino}
\end{figure}
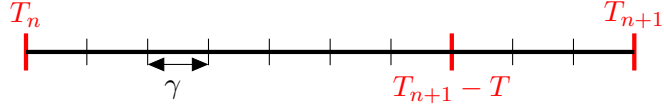

More precisely, we prove that for all $\beta >0$, if
\begin{equation}
\label{eq:hyp_A_Y}
A > \max\left(\sqrt{(\beta+1)(2C_1+C_2)}, \sqrt{(1+\beta^{-1})C_2} \right) ,
\end{equation}
then
$$ \mathcal{W}_1([Y^{x_0}_t], \nu_{a(t)}) \le \frac{C\max(1+|x_0|,V(x_0))}{\log^{3/2}(t)t^{(1+\beta)^{-1}-(2C_1+C_2)/A^2}} .$$

\begin{proof}
We apply the \textit{domino strategy} \eqref{eq:domino_strategy}.
Let us fix $T \in (0,T_1)$ and $\gamma \in (0,T_1-T)$. Here $\gamma$ is not linked to any Euler-Maruyama scheme but is an auxiliary tool for the proof. Let $n \ge 0$ and let $f : \mathbb{R}^d \to \mathbb{R}$ be Lipschitz continuous. We divide the two intervals $[T_n, T_{n+1}-T]$ and $[T_{n+1}-T, T_{n+1}]$ into smaller intervals of size $\gamma$ (see Figure \ref{fig:domino}) and for $x \in \mathbb{R}^d$ using the semi-group property of $P^{X,n}$ on $[T_n,T_{n+1})$ we write:
\begin{align*}
& \left| \mathbb{E}f(X_{T_{n+1}-T_n}^{x,n}) - \mathbb{E}f(Y_{T_{n+1}-T_n,T_n}^{x})\right| \\
& \quad \le \sum_{k=1}^{\lfloor(T_{n+1}-T_n-T)/\gamma\rfloor} \left| P^Y_{(k-1)\gamma,T_n} \circ (P^Y_{\gamma,T_n+(k-1)\gamma} - P^{X,n}_{\gamma}) \circ P^{X,n}_{T_{n+1}-T_n-k\gamma} f(x) \right| \\
& \quad \quad + \sum_{k=\lfloor(T_{n+1}-T_n-T)/\gamma\rfloor+1}^{\lfloor(T_{n+1}-T_n)/\gamma\rfloor-1} \left| P^Y_{(k-1)\gamma,T_n} \circ (P^Y_{\gamma,T_n+(k-1)\gamma} - P^{X,n}_{\gamma}) \circ P^{X,n}_{T_{n+1}-T_n-k\gamma} f(x) \right| \\
& \quad \quad + \left| P^Y_{\gamma(\lfloor (T_{n+1}-T_n)/\gamma \rfloor-1) ,T_n} \circ (P^Y_{\gamma + (T_{n+1}-T_n) \text{mod} \gamma, T_n + \gamma(\lfloor (T_{n+1}-T_n)/\gamma \rfloor-1)} - P^{X,n}_{\gamma + (T_{n+1}-T_n) \text{mod}(\gamma)}) f(x) \right| \\
& \quad =: (a) + (b) + (c).
\end{align*}
The term $(a)$ is the "ergodic term", for which the exponential contraction from Theorem \ref{thm:confluence} can be exploited. The terms $(b)$ and $(c)$ are the "error terms" where we bound the error on intervals of length no larger than $T$. The term $(c)$ is a remainder term due to the fact that $T_{n+1}-T_n$ is generally not a multiple of $\gamma$.

\medskip

$\bullet$ \textbf{Term $(a)$ :} It follows from Theorem \ref{thm:confluence} and Lemma \ref{lemma:3.4.b:Y} that
\begin{align*}
& |(P^Y_{\gamma,T_n+(k-1)\gamma} - P^{X,n}_{\gamma}) \circ P^{X,n}_{T_{n+1}-T_n-k\gamma} f(x) | \\
& \quad = |\mathbb{E}P^{X,n}_{T_{n+1}-T_n-k\gamma} f(X_{\gamma}^{x,n}) - \mathbb{E}P^X_{T_{n+1}-T_n-k\gamma,n} f(Y_{\gamma,T_n+(k-1)\gamma}^x)| \\
& \quad \le C e^{C_1 a_{n+1}^{-2}} e^{-\rho_{n+1}(T_{n+1}-T_n-k\gamma)} [f]_{\text{Lip}} \mathbb{E}|X_{\gamma}^{x,n} - Y_{\gamma,T_n+(k-1)\gamma}^x| \\
& \quad \le C e^{C_1 a_{n+1}^{-2}} e^{-\rho_{n+1}(T_{n+1}-T_n-k\gamma)} [f]_{\text{Lip}} \sqrt{\gamma} (a_n - a_{n+1})
\end{align*}
Integrating with respect to $P^Y_{(k-1)\gamma,T_n}$ and summing up yields
\begin{align*}
(a) & \le Ce^{C_1 a_{n+1}^{-2}} [f]_{\text{Lip}} \sqrt{\gamma} (a_n - a_{n+1}) \frac{e^{-\rho_{n+1} T} - e^{-\rho_{n+1} (T_{n+1}-T_n)}}{e^{\gamma\rho_{n+1}}-1} \\
& \le Ce^{C_1 a_{n+1}^{-2}} [f]_{\text{Lip}} \sqrt{\gamma} (a_n - a_{n+1})(\gamma \rho_{n+1})^{-1} .
\end{align*}

\medskip

$\bullet$ \textbf{Term $(b)$:} Applying Proposition \ref{prop:3.6:Y} yields:
$$ |(P^Y_{\gamma, T_n+(k-1)\gamma} - P^{X,n}_{\gamma}) \circ P^{X,n}_{T_{n+1}-T_n-k\gamma} f(x) | \le C a_{n+1}^{-2}(a_n-a_{n+1})[f]_{\text{Lip}} \frac{\gamma}{\sqrt{T_{n+1}-T_n-k\gamma}} V(x) .$$
Integrating with respect to $P^Y_{(k-1)\gamma,T_n}$ and using Lemma \ref{lemma:D.1a:cont} which guarantees that
$ P^Y_{(k-1)\gamma,T_n} V(x) \le C V(x) $
and summing with respect to $k$ implies
$$ (b) \le  C a_n^{-2}(a_n-a_{n+1})[f]_{\text{Lip}} \gamma V(x) \sum_{k=1}^{\lceil T/\gamma \rceil} (k\gamma)^{-1/2} \le C a_n^{-2}(a_n-a_{n+1})[f]_{\text{Lip}} T^{1/2} V(x)  .$$

\medskip

$\bullet$ \textbf{Term $(c)$:} Noting that $\gamma + (T_{n+1} - T_n) \ \text{mod} (\gamma) \le 2\gamma$, Lemma \ref{lemma:3.4.b:Y} yields
$$ (c) \le C[f]_{\text{Lip}}\sqrt{\gamma}(a_n-a_{n+1}). $$

\medskip

Now we sum up the terms $(a)$, $(b)$ and $(c)$. Since $\gamma$ is constant we have:
$$ \left| \mathbb{E}f(X_{T_{n+1}-T_n}^{x,n}) - \mathbb{E}f(Y_{T_{n+1}-T_n,T_n}^{x})\right| \le C e^{C_1 a_{n+1}^{-2}}(a_n-a_{n+1})\rho_{n+1}^{-1}[f]_{\text{Lip}}V(x) ,$$
so that for all $x \in \mathbb{R}^d$,
\begin{equation}
\label{eq:proof_Y:2}
\mathcal{W}_1([X_{T_{n+1}-T_n}^{x,n}], [Y_{T_{n+1}-T_n,T_n}^x]) \le Ce^{C_1 a_{n+1}^{-2}}(a_n-a_{n+1})\rho_{n+1}^{-1}V(x) .
\end{equation}
Temporarily setting $x_n := X^{x_0}_{T_n}$ and $y_n := Y^{x_0}_{T_n}$, we derive
\begin{align*}
& \mathcal{W}_1([X^{x_0}_{T_{n+1}}],[Y^{x_0}_{T_{n+1}}]) = \mathcal{W}_1([X_{T_{n+1}-T_n}^{x_n,n}], [Y_{T_{n+1}-T_n,T_n}^{y_n}]) \\
& \quad \quad \quad \quad \le \mathcal{W}_1([X_{T_{n+1}-T_n}^{x_n,n}],[X_{T_{n+1}-T_n}^{y_n,n}]) + \mathcal{W}_1([X_{T_{n+1}-T_n}^{y_n,n}], [Y_{T_{n+1}-T_n,T_n}^{y_n}]) \\
& \quad \quad \quad \quad \le Ce^{C_1 a_{n+1}^{-2}} e^{-\rho_{n+1} (T_{n+1}-T_n)} \mathcal{W}_1([X^{x_0}_{T_n}],[Y^{x_0}_{T_n}]) + Ce^{C_1 a_{n+1}^{-2}}(a_n-a_{n+1})\rho_{n+1}^{-1}\mathbb{E}V(Y^{x_0}_{T_n}),
\end{align*}
where we used Theorem \ref{thm:confluence} and \eqref{eq:proof_Y:2}.
We then apply Lemma \ref{lemma:D.1a:cont} which guarantees that $(\mathbb{E}V(Y^{x_0}_{T_n}))_n$ is bounded by $CV(x_0)$.
Let us denote
$$\lambda_n := Ce^{C_1 a_n^{-2}}(a_{n-1}-a_n) \rho_n^{-1} = Ce^{(C_1+C_2) a_n^{-2}}(a_{n-1}-a_n).$$
Owing to \eqref{eq:hyp_A_Y} we have $\lambda_n \to 0$.
Iterating this relation and using $(\mu_n)$ defined in \eqref{eq:def_mu} yields like in the proof of Theorem \ref{thm:conv_X}:
\begin{align*}
\mathcal{W}_1([X^{x_0}_{T_{n+1}}],[Y^{x_0}_{T_{n+1}}]) & \le CV(x_0) \left(\lambda_{n+1} + \mu_{n+1} \lambda_n + \mu_{n+1} \mu_n \lambda_{n-1} + \cdots + \mu_{n+1} \cdots \mu_2 \lambda_1 \right) \\
& \le CV(x_0) \left(\lambda_{n+1} + \mu_{n+1}\left(\lambda_n + \cdots + \lambda_1\right) \right) \\
& \le  CV(x_0) \left(\lambda_{n+1} + n\mu_{n+1} \right).
\end{align*}
But following \eqref{eq:def_mu} one checks that $n\mu_{n+1} = o(\lambda_{n+1})$ so that
$$ \mathcal{W}_1([X^{x_0}_{T_{n+1}}],[Y^{x_0}_{T_{n+1}}]) \le CV(x_0) \lambda_{n+1} \le \frac{CV(x_0)}{\log^{3/2}(n+1)(n+1)^{1-(\beta+1)(C_1+C_2)/A^2}} .$$
Moreover, owing to \eqref{eq:hyp_A_Y} and combining with Theorem \ref{thm:conv_X} we get
\begin{align*}
\mathcal{W}_1([Y^{x_0}_{T_n}],\nu_{a_n}) & \le \mathcal{W}_1([Y^{x_0}_{T_n}],[X^{x_0}_{T_n}]) + \mathcal{W}_1([X^{x_0}_{T_n}],\nu_{a_n}) \le \frac{C\max(1+|x_0|,V(x_0))}{\log^{3/2}(n)n^{1-(\beta+1)(C_1+C_2)/A^2}}
\end{align*}
and as the right hand side of these inequalities is in $o(a_n)$, we derive
$$ \mathcal{W}_1([Y^{x_0}_{T_n}],\nu^\star) \le \mathcal{W}_1([Y^{x_0}_{T_n}],[X^{x_0}_{T_n}]) + \mathcal{W}_1([X^{x_0}_{T_n}],\nu^\star) \le Ca_n \max(1+|x_0|,V(x_0)) .$$

\medskip

$\bullet$ \textbf{Convergence for $t \to \infty$ :}
Now let us prove that $\mathcal{W}_1([Y^{x_0}_{t}],\nu^\star) \to 0$ as $t \to \infty$. As before, let $T>0$. For $t \ge T$, then we perform the same domino strategy where we replace $T_{n+1}$ by $T_n + t$ and we consider the intervals $[T_n,T_n+t-T]$ and $[T_n+t-T,T_n+t]$. For $t < T$ then we only consider the terms $(b)$ and $(c)$ and we replace $T$ by $t$ in $(b)$. Doing so we obtain
$$ \mathcal{W}_1([X^{x,n}_t],[Y^{x}_{t,T_n}]) \le Ce^{C_1 a_{n+1}^{-2}}(a_n-a_{n+1})\rho_{n+1}^{-1}V(x). $$
So that, as before:
\begin{align*}
\mathcal{W}_1([X^{x_0}_{T_n+t}],[Y^{x_0}_{T_n+t}]) & \le Ce^{C_1 a_{n+1}^{-2}}\mathcal{W}_1([X^{x_0}_{T_n}],[Y^{x_0}_{T_n}]) + Ce^{C_1 a_{n+1}^{-2}}(a_n-a_{n+1})\rho_{n+1}^{-1}V(x_0) \\
& \le \frac{CV(x_0)}{\log^{3/2}(n)n^{1-(\beta+1)(2C_1+C_2)/A^2}}.
\end{align*}
Owing to \eqref{eq:hyp_A_Y} we have $1-(\beta+1)(2C_1+C_2)/A^2>1$, so that, using \eqref{eq:proof:X_t_nu},
$$ \mathcal{W}_1([Y^{x_0}_{T_n+t}],\nu_{a_{n+1}}) \le \mathcal{W}_1([Y^{x_0}_{T_n+t}],[X^{x_0}_{T_n+t}]) + \mathcal{W}_1([X^{x_0}_{T_n+t}],\nu_{a_{n+1}}) \le \frac{C\max(1+|x_0|,V(x_0))}{\log^{3/2}(n)n^{1-(\beta+1)(2C_1+C_2)/A^2}} .$$
We then prove the bound for $\mathcal{W}_1([Y^{x_0}_t], \nu_{a(t)})$ the same way as for \eqref{eq:proof:X_t_nu:2}, using the second claim of Proposition \ref{prop:W_nu}.

\end{proof}

\section{Continuously decreasing case : the Euler-Maruyama scheme}
\label{sec:bar-Y}

\begin{figure}
	\centering
\begin{tikzpicture}[scale = 1.]
\begin{axis}[grid=none,
			axis equal image,
            axis x line=center,
            axis y line=center,
            black,              
            xlabel={$t$},
            ylabel={$a(t)$},
            xlabel style={below right},
            ylabel style={above center},
            xmax=11,
            xmin=0,
            ymax=4,
            ymin=0,
            xtick={1.5,4,7,11},
            xticklabels={$T_1$,$T_2$,$T_3$,$T_4$},
            ytick={},
            yticklabels={},
            axis lines=middle,
            domain=0:11
            ]
\addplot[mark=none, red, thick] {2/sqrt(ln(x+2))} coordinate[pos=0.3] (aux);
\addplot[blue, domain=0:1.5, thick] {2/sqrt(ln(1.5+2))};
\addplot[blue, domain=1.5:4, thick] {2/sqrt(ln(4+2))};
\addplot[blue, domain=4:7, thick] {2/sqrt(ln(7+2))};
\addplot[blue, domain=7:11, thick] {2/sqrt(ln(11+2))};
\addlegendentry{Non plateau case}
\addlegendentry{Plateau case}
\end{axis}

\end{tikzpicture}
	\caption{\textit{Decreasing of the noise coefficient $a$ for the plateau and non plateau cases.}}
	\label{fig:a_decreasing}
\end{figure}
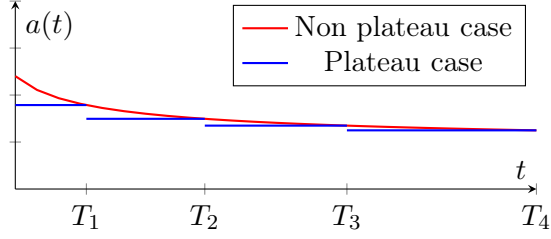

We now consider $(\bar{Y}_n)$ to be the Euler-Maruyama scheme of $(Y_t)$ with steps $(\gamma_n)$ defined in \eqref{eq:def_Y_bar} and we also consider its genuine interpolation defined in \eqref{eq:def_Y_bar_genuine}.
As with \eqref{eq:def_Y:2}, we define more generally for every $n \ge 0$, $(\bar{Y}^x_{t,\Gamma_n})_{t\ge 0}$, first at times $\Gamma_k-\Gamma_n$, $k \ge n$, by
\begin{align*}
\bar{Y}^x_{0,\Gamma_n} = x, \quad \bar{Y}^x_{\Gamma_{k+1}-\Gamma_n,\Gamma_n} & = \bar{Y}^x_{\Gamma_k -\Gamma_n, \Gamma_n} + \gamma_{k+1} \left(b_{a(\Gamma_k)}(\bar{Y}_{\Gamma_k-\Gamma_n,\Gamma_n}^{x}) + \zeta_{k+1}(\bar{Y}_{\Gamma_k-\Gamma_n,\Gamma_n}^{x}) \right) \\
& \quad + a(\Gamma_k)\sigma(\bar{Y}_{\Gamma_k-\Gamma_n,\Gamma_n}^{x})(W_{\Gamma_{k+1}} - W_{\Gamma_k}),
\end{align*}
then at every time $t$ by the genuine interpolation on the intervals $([\Gamma_k-\Gamma_n, \Gamma_{k+1}- \Gamma_n))_{k \ge n}$ as before. In particular $\bar{Y}^x = \bar{Y}^x_{\cdot,0}$.
Still more generally, we define $\bar{Y}^x_{t,u}$ where $u \in (\Gamma_n, \Gamma_{n+1})$ as
\begin{equation*}
\bar{Y}^x_{0,u} = x, \quad \bar{Y}^x_{t,u} = \left\lbrace \begin{array}{ll}
 x + t(b_a(x) + \zeta_{n+1}(x)) + a^2(u)\sigma(x)(W_t-W_{\Gamma_u}) & \text{ if } t \in [u, \Gamma_{n+1}] \\
 \bar{Y}^{\bar{Y}^x_{\Gamma_{n+1}-u,u}}_{t-(\Gamma_{n+1}-u),\Gamma_{n+1}} & \text{ if } t > \Gamma_{n+1} .
\end{array} \right.
\end{equation*}
For $n$, $k \ge 0$, for $u \in [\Gamma_k,\Gamma_{k+1})$ and $\gamma \in [0,\Gamma_{k+1}-u]$, let $P^{\bar{Y}}_{\gamma,u}$ be the transition kernel associated to $\bar{Y}_{\cdot,u}$ between the times $0$ and $\gamma$ i.e. for all $f : \mathbb{R}^d \to \mathbb{R}^+$ measurable, $P^{\bar{Y}}_{\gamma,u}f(x) = \mathbb{E}[f(\bar{Y}^x_{\gamma,u})]$.

\subsection{Boundedness of the potential}

\begin{lemma}
\label{lemma:D.1a}
Let $p \ge 1/2$. There exists a constant $C >0$ such that for every $k \ge 0$, for every $u \in [\Gamma_k, \Gamma_{k+1})$ and for every $x \in \mathbb{R}^d$:
$$ \sup_{n \ge k+1} \mathbb{E} V^p(\bar{Y}_{\Gamma_n-u,u}^x) \le CV^p(x). $$
\end{lemma}
\begin{proof}
We rework the proof of Lemma 2(b) in \cite{lamberton2002}. Let us assume directly that $u = \Gamma_k$. To simplify the notations, we define $\widetilde{y}_n := \bar{Y}^x_{\Gamma_n-\Gamma_k,\Gamma_k}$ for $n \ge k$ and $\Delta \widetilde{y}_{n+1} := \widetilde{y}_{n+1} - \widetilde{y}_n$. The Taylor formula applied to $V^p$ between $\widetilde{y}_n$ and $\widetilde{y}_{n+1}$ yields for some $\xi_{n+1} \in (\widetilde{y}_n, \widetilde{y}_{n+1})$ and with $\nabla^2 (V^p) = p(V^{p-1} \nabla^2 V + (p-1)V^{p-2} \nabla V \nabla V^T)$ :
\begin{align*}
V^p(\widetilde{y}_{n+1}) & = V^p(\widetilde{y}_n) + pV^{p-1} \langle \nabla V(\widetilde{y}_n), \Delta\widetilde{y}_{n+1} \rangle + \frac{1}{2}\nabla^2 (V^p)(\xi_{n+1}) \cdot (\Delta \widetilde{y}_{n+1})^{\otimes 2} \\
& = V^p(\widetilde{y}_n) + pV^{p-1} \nabla V(\widetilde{y}_n)^T \cdot \big( -\gamma_{n+1}\sigma \sigma^\top (\widetilde{y}_n)\nabla V(\widetilde{y}_n) + \gamma_{n+1} a^2(\Gamma_n)\Upsilon(\widetilde{y}_n)  \\
& \quad + \gamma_{n+1} \zeta_{n+1}(\widetilde{y}_n) + \sqrt{\gamma_{n+1}} a(\Gamma_n) \sigma (\widetilde{y}_n)U_{n+1}\big) + \frac{1}{2}\nabla^2 (V^p)(\xi_{n+1}) \cdot (\Delta \widetilde{y}_{n+1})^{\otimes 2},
\end{align*}
where $U_{n+1} \sim \mathcal{N}(0,I_d)$. Moreover using \eqref{Eq:eq:V_assumptions}, $\sqrt{V}$ is Lipschitz continuous so
\begin{equation}
\label{eq:potential_bounded_proof}
\textstyle \mathbb{E} \big[\sup_{z \in (\widetilde{y}_n, \widetilde{y}_{n+1})} V^{1/2}(z) | \widetilde{y}_1,\ldots,\widetilde{y}_n \big] \le V^{1/2}(\widetilde{y}_n) + [\sqrt{V}]_{\text{Lip}} \mathbb{E}[|\widetilde{y}_{n+1} - \widetilde{y}_n||\widetilde{y}_1,\ldots,\widetilde{y}_n] \le CV^{1/2}(\widetilde{y}_n),
\end{equation}
and in particular
$$ \mathbb{E}[\|\nabla^2(V^p)(\xi_{n+1})\| |\widetilde{y}_1,\ldots,\widetilde{y}_n] \le C\|\nabla^2(V^p)(\widetilde{y}_n)\| .$$
Moreover using that $\nabla^2 V$ is bounded and that $|\nabla V| \le CV^{1/2}$ we have
$$ \|\nabla^2(V^p)(\widetilde{y}_n)\| \le C\|(V^{p-1} \nabla^2 V + V^{p-2} \nabla V \nabla V^T)(\widetilde{y}_n)\| \le CV^{p-1}(\widetilde{y}_n). $$
Then using the facts that $a$, $\Upsilon$, $\sigma$, $\nabla^2 V$ are bounded and that $\gamma_n^2 = o(\gamma_n)$, that $V$, $\nabla V$ are coercive and $\sigma \sigma^\top \ge \ubar{\sigma}_0^2 I_d$ and \eqref{eq:zeta_assumptions}, there exists $R>0$ and $N \in \mathbb{N}$ such that if $|\widetilde{y}_n| \ge R$ and $n \ge N$ then
\begin{align*}
& \mathbb{E}[ V^p(\widetilde{y}_{n+1}) - V^p(\widetilde{y}_n) | \widetilde{y}_1,\ldots,\widetilde{y}_n ] \\
& \quad \le pV^{p-1} \nabla V(\widetilde{y}_n)^T \cdot \left( -\gamma_{n+1} \ubar{\sigma}_0^2 (\widetilde{y}_n)\nabla V(\widetilde{y}_n) + C\gamma_{n+1} \right) \\
& \quad \quad + C\|\nabla^2 (V^p)(\widetilde{y}_n) \| \cdot \left( \gamma_{n+1}^2\|\sigma\|_\infty^4 |\nabla V(\widetilde{y}_n)|^2 + C\gamma_{n+1}^2 + C\gamma_{n+1}^2 V(x) + C\gamma_{n+1}\mathbb{E}|\mathcal{N}(0,I_d)|^2 \right) \\
& \quad \le C\gamma_{n+1}V^{p-1}(\widetilde{y}_n) \left[ |\nabla V(\widetilde{y}_n)|\big(- |\nabla V(\widetilde{y}_n)| + 1\big) + \gamma_{n+1}(|\nabla V(\widetilde{y}_n)|^2 +1) + 1  \right] \le 0 .
\end{align*}
On the other side, if $|\widetilde{y}_n| \le R$ then
$$ \mathbb{E}[|V^p(\widetilde{y}_{n+1}) - V^p(\widetilde{y}_n)|| \widetilde{y}_1,\ldots,\widetilde{y}_n] \le C \gamma_{n+1} \textstyle \sup_{|x| \le R} V^p(x) .$$
Moreover for $n \in \lbrace k, \ldots, N \rbrace$ using \eqref{eq:potential_bounded_proof} we have
$$ \mathbb{E}[|V^p(\widetilde{y}_{n+1}) - V^p(\widetilde{y}_n)| | \widetilde{y}_1,\ldots,\widetilde{y}_n] \le C V^p(\widetilde{y}_n)$$
so that
$$ \textstyle \sup_{k \le n \le N+1} \mathbb{E}[V^p(\widetilde{y}_n)] \le C^{N-k} V^p(x) .$$
Finally we obtain
$$ \textstyle \sup_{n \ge k} \mathbb{E}[V^p(\widetilde{y}_n)] \le C V^p(x) .$$


\end{proof}

\subsection{Strong and weak error bounds for the Euler-Maruyama scheme}

\begin{lemma}
\label{lemma:3.4.b:Y:bar}
Let $p \ge 1$. There exists $C>0$ such that for every $n$, $k \ge 0$, for every $u \in [\Gamma_k,\Gamma_{k+1})$ and every $t>0$ such that $u \in [T_n,T_{n+1}]$, $t \le \Gamma_{k+1}-u$ and $u+t \in [T_{n},T_{n+1}]$,
$$|| X_t^{x,n} - \bar{Y}_{t,u}^x ||_p \le C \left(V^{1/2}(x)t + \sqrt{t} (a_n - a_{n+1})\right). $$
\end{lemma}
\begin{proof}
As in the proof of Lemma \ref{lemma:3.4.b:Y}, if $p \ge 2$ we have
\begin{align*}
& || X_t^{x,n} - \bar{Y}_{t,u}^x ||_p \le [b]_{\text{Lip}} \int_0^t ||X_s^{x,n} - x||_p ds + ||\Upsilon||_\infty (a_n^2 - a_{n+1}^2)t + a_{n+1}^2 [\Upsilon]_{\text{Lip}}  \int_0^t ||X_s^{x,n} - x||_p ds \\
& \quad \quad \quad + \|\zeta_1(x)\|_p t + C^{BDG} a_{n+1} [\sigma]_{\text{Lip}}\left( \int_0^t ||X_s^{x,n} - x||^2_p ds \right)^{1/2} + ||\sigma||_\infty ||W_1||_p \sqrt{t} (a_n - a_{n+1}).
\end{align*}
Plugging Lemma \ref{lemma:3.4.a} and \eqref{eq:zeta_assumptions} into this inequality yields:
\begin{align*}
|| X_t^{x,n} - \bar{Y}_{t,u}^x ||_p & \le CV^{1/2}(x)t^{3/2} + ||\Upsilon||_\infty (a_n^2 - a_{n+1}^2)t + CV^{1/2}(x)t + C \sqrt{t} (a_n - a_{n+1}),
\end{align*}
which completes the proof for $p \ge 2$. If $p \in [1,2)$, we remark that $\| \cdot \|_p \le \| \cdot \|_2$.
\end{proof}

\begin{proposition}
\label{prop:3.5:Y:bar}
For every $g : \mathbb{R}^d \to \mathbb{R}$ being $\mathcal{C}^3$, for every $n$, $k \ge 0$  and every $u \in [\Gamma_k, \Gamma_{k+1})$ such that $u \in [T_n,T_{n+1}]$, $\gamma \le \Gamma_{k+1}-u$ and $u +\gamma \in [T_n,T_{n+1}]$:
\begin{align*}
| \mathbb{E}\left[g(\bar{Y}^x_{\gamma,u})\right] - \mathbb{E}\left[g(X^{x,n}_{\gamma})\right]| & \le C V^{1/2}(x)\left(V^{1/2}(x) \gamma^2 + \gamma (a_n - a_{n+1})\right) \bar{\Phi}_{g,1}(x) \\
& \quad + C V(x)\left(V^{1/2}(x) \gamma^2 + \gamma^{3/2} (a_n - a_{n+1})\right)\bar{\Phi}_{g,2}(x) ,
\end{align*}
with
\begin{align*}
& \bar{\Phi}_{g,1}(x) = \max\left(|\nabla g(x)|, ||\nabla^2 g(x)||, \left|\left|\sup_{\xi \in (X_\gamma^{x,n}, \bar{Y}_{\gamma,u}^x)} || \nabla^2 g(\xi) || \right|\right|_2 \right), \\
& \bar{\Phi}_{g,2}(x) = \left|\left| \sup_{\xi \in (x, X^{x,n}_\gamma)} ||\nabla^3 g(\xi)|| \right|\right|_4.
\end{align*}
\end{proposition}
The proof is given in the Supplementary Material.

\begin{proposition}
\label{prop:3.6:Y:bar}
Let $T >0$. There exists $C > 0$ such that for every Lipschitz continuous function $f$ and every $t \in (0,T]$, for all $n$, $k \ge 0$, for all $u \in [\Gamma_k, \Gamma_{k+1})$, for all $\gamma$ such that $\Gamma_k \in [T_n,T_{n+1}]$, $\gamma \le \Gamma_{k+1}-u$ and $u+t+\gamma \in [T_n,T_{n+1}]$,
\begin{align*}
& \left| \mathbb{E}\left[P_t^{X,n} f(\bar{Y}_{\gamma,u}^x)\right] - \mathbb{E}\left[P_t^{X,n} f(X_{\gamma}^{x,n})\right] \right| \\
& \quad \le C[f]_{\textup{Lip}} V^{2}(x) \cdot \left(a_{n+1}^{-2}t^{-1/2}\left(\gamma^2 + (a_n-a_{n+1}) \gamma\right) + a_{n+1}^{-3}t^{-1}\left(\gamma^2 + \gamma^{3/2}(a_n-a_{n+1})\right) \right).
\end{align*}
\end{proposition}
\begin{proof}
The proof is the same as for Proposition \ref{prop:3.6:Y}.
\end{proof}

\subsection{Proof of Theorem \ref{thm:main}.(b)}
\label{subsec:proof_Y_bar}

More precisely, we prove that for all $\beta>0$, if
\begin{equation}
\label{eq:hyp_A_Y_bar}
A > \max\left(\sqrt{(\beta+1)(2C_1+C_2)}, \sqrt{(1+\beta^{-1})C_2} \right) ,
\end{equation}
then
$$ \mathcal{W}_1([\bar{Y}^{x_0}_t], \nu_{a(t)}) \le \frac{C\max\left(1+|x_0|,V^2(x_0)\right)}{t^{(1+\beta)^{-1}-(2C_1+C_2)/A^2}} .$$

\begin{proof}
We apply the same \textit{domino strategy} as in Section \ref{subsec:proof_Y}. Let $n \ge 0$ and let $f : \mathbb{R}^d \to \mathbb{R}$ be Lipschitz continuous.
Let us denote
$$\gamma^{\text{init}} := \Gamma_{N(T_n)+1}-T_n \le \gamma_{N(T_n)+1} \quad \text{ and } \quad \gamma^{\text{end}} := T_{n+1}-\Gamma_{N(T_{n+1})} \le \gamma_{N(T_{n+1})+1}. $$
For $x \in \mathbb{R}^d$ we write:
\begin{align*}
& \left| \mathbb{E}f(X_{T_{n+1}-T_n}^{x,n}) - \mathbb{E}f(\bar{Y}_{T_{n+1}-T_n,T_n}^{x})\right| \le \left|(P^{\bar{Y}}_{\gamma^{\text{init}},T_n} - P^{X,n}_{\gamma^{\text{init}}}) \circ P^{X,n}_{T_{n+1}-\Gamma_{N(T_n)+1}} f(x)\right| \\
& + \sum_{k=N(T_n)+2}^{N(T_{n+1}-T)} \left| P^{\bar{Y}}_{\gamma^{\text{init}},T_n} \circ P^{\bar{Y}}_{\gamma_{N(T_n)+2},\Gamma_{N(T_n)+1}} \circ \cdots \circ P^{\bar{Y}}_{\gamma_{k-1},\Gamma_{k-2}} \circ (P^{\bar{Y}}_{\gamma_{k},\Gamma_{k-1}} - P^{X,n}_{\gamma_k}) \circ P^{X,n}_{T_{n+1} - \Gamma_k} f(x) \right| \\
& + \sum_{k=N(T_{n+1}-T)+1}^{N(T_{n+1})-1} \left| P^{\bar{Y}}_{\gamma^{\text{init}},T_n} \circ P^{\bar{Y}}_{\gamma_{N(T_n)+2},\Gamma_{N(T_n)+1}} \circ \cdots \circ P^{\bar{Y}}_{\gamma_{k-1},\Gamma_{k-2}} \circ (P^{\bar{Y}}_{\gamma_{k},\Gamma_{k-1}} - P^{X,n}_{\gamma_k}) \circ P^{X,n}_{T_{n+1} - \Gamma_k} f(x) \right| \\
& {+} \left| P^{\bar{Y}}_{\gamma^{\text{init}},T_n} {\circ} P^{\bar{Y}}_{\gamma_{N(T_n)+2},\Gamma_{N(T_n)+1}} {\circ} \cdots {\circ} P^{\bar{Y}}_{\gamma_{N(T_{n+1})-1},\Gamma_{N(T_{n+1})-2}} {\circ} (P^{\bar{Y}}_{\gamma^{\text{end}}+\gamma_{N(T_{n+1})},\Gamma_{N(T_{n+1})-1}} {-} P^{X,n}_{\gamma^{\text{end}}+\gamma_{N(T_{n+1})}}) f(x)\right| \\
& =: (c^{\text{init}}) + (a) + (b) + (c^{\text{end}}).
\end{align*}

\medskip

$\bullet$ \textbf{Term $(a)$:} we have
\begin{align*}
& |(P^{\bar{Y}}_{\gamma_k,\Gamma_{k-1}} - P^{X,n}_{\gamma_k}) \circ P^{X,n}_{T_{n+1} - \Gamma_k} f(x)| \\
& = |P^{\bar{Y}}_{\gamma_k,\Gamma_{k-1}} \circ P^{X,n}_{T/2} \circ P^{X,n}_{T_{n+1}-\Gamma_k - T/2} f(x) - P^{X,n}_{\gamma_k,n} \circ P^{X,n}_{T/2} \circ P^{X,n}_{T_{n+1}-\Gamma_k - T/2} f(x)| \\
& \le | \mathbb{E} P^{X,n}_{T_{n+1}-\Gamma_k-T/2}(\Xi_k^x) - \mathbb{E}P^{X,n}_{T_{n+1}-\Gamma_k-T/2}(\bar{\Xi}_k^x)| \le Ce^{Ca_{n+1}^{-2}} e^{-\rho_{n+1}(T_{n+1}-\Gamma_k-T/2)}[f]_{\text{Lip}}\mathbb{E}|\Xi^x_k - \bar{\Xi}^x_k|,
\end{align*}
where $\Xi^x_k$ and $\bar{\Xi}^x_k$ are any random vectors with laws $\left[X_{T/2}^{X^{x,n}_{\gamma_k},n}\right]$ and $\left[X_{T/2}^{\bar{Y}^x_{\gamma_k,\Gamma_{k-1}},n}\right]$ respectively and where we used Theorem \ref{thm:confluence} to get the last inequality. Thus, it follows from the definition of the Wasserstein distance that
$$ |(P^{\bar{Y}}_{\gamma_k,\Gamma_{k-1}} - P^{X,n}_{\gamma_k}) \circ P^{X,n}_{T_{n+1} - \Gamma_k} f(x)| \le Ce^{Ca_{n+1}^{-2}}e^{-\rho_{n+1}(T_{n+1}-\Gamma_k)}[f]_{\text{Lip}} \mathcal{W}_1\left(X_{T/2}^{X^{x,n}_{\gamma_k},n}, X_{T/2}^{\bar{Y}^x_{\gamma_k,\Gamma_{k-1}},n}\right) .$$
On the other hand, the Kantorovich-Rubinstein representation of the $L^1$-Wasserstein distance (see \cite[Equation (6.3)]{villani2009}) reads
\begin{align*}
\mathcal{W}_1\left(X_{T/2}^{X^{x,n}_{\gamma_k},n}, X_{T/2}^{\bar{Y}^x_{\gamma_k,\Gamma_{k-1}},n}\right) & = \sup_{[g]_{\text{Lip}}=1} \mathbb{E}\left[ g\left(X_{T/2}^{X^{x,n}_{\gamma_k},n}\right) - g\left(X_{T/2}^{\bar{Y}^x_{\gamma_k,\Gamma_{k-1}},n}\right) \right] \\
& = \sup_{[g]_{\text{Lip}}=1} \mathbb{E}\left[ P^{X,n}_{T/2}g(X^{x,n}_{\gamma_k}) - P^{X,n}_{T/2}g(\bar{Y}^x_{\gamma_k,\Gamma_{k-1}}) \right].
\end{align*}
It follows from Proposition \ref{prop:3.6:Y:bar} and using $[g]_{\text{Lip}}=1$ that
$$ \mathbb{E}\left[ P^{X,n}_{T/2}g(X^{x,n}_{\gamma_k}) - P^{X,n}_{T/2}g(\bar{Y}^x_{\gamma_k,\Gamma_{k-1}}) \right] \le Ca_{n+1}^{-3}\left(\gamma_k^2 +(a_n-a_{n+1}) \gamma_k \right)V^2(x), $$
so that
$$ |(P^{\bar{Y}}_{\gamma_k,\Gamma_{k-1}} - P^{X,n}_{\gamma_k}) \circ P^{X,n}_{T_{n+1} - \Gamma_k} f(x)| \le Ce^{C_1 a_{n+1}^{-2}}e^{-\rho_{n+1}(T_{n+1}-\Gamma_k)}[f]_{\text{Lip}} a_{n+1}^{-3}\left(\gamma_k^2 + (a_n-a_{n+1}) \gamma_k\right) V^2(x) .$$
Finally, integrating with respect to $P^{\bar{Y}}_{\gamma^{\text{init}},T_n} \circ P^{\bar{Y}}_{\gamma_{N(T_n)+2},\Gamma_{N(T_n)+1}} \circ \cdots \circ P^{\bar{Y}}_{\gamma_{k-1},\Gamma_{k-2}}$ yields:
\begin{align*}
& \left| P^{\bar{Y}}_{\gamma^{\text{init}},T_n} \circ P^{\bar{Y}}_{\gamma_{N(T_n)+2},\Gamma_{N(T_n)+1}} \circ \cdots \circ P^{\bar{Y}}_{\gamma_{k-1},\Gamma_{k-2}} \circ (P^{\bar{Y}}_{\gamma_k,\Gamma_{k-1}} - P^{X,n}_{\gamma_k}) \circ P^{X,n}_{T_{n+1} - \Gamma_k} f(x) \right| \\
& \quad \le Ce^{C_1 a_{n+1}^{-2}}e^{-\rho_{n+1}(T_{n+1}-\Gamma_k)}[f]_{\text{Lip}} a_{n+1}^{-3}\left(\gamma_k^2 + (a_n-a_{n+1}) \gamma_k\right) \left(\sup_{\ell \ge N(T_n)+1} \mathbb{E} V^2(\bar{Y}_{\gamma^{\text{init}} + \Gamma_\ell - \Gamma_{N(T_n)+1},T_n}^x)\right) \\
& \quad \le Ce^{C_1 a_{n+1}^{-2}}e^{-\rho_{n+1}(T_{n+1}-\Gamma_k)}[f]_{\text{Lip}} a_{n+1}^{-3}\left(\gamma_k^2 + (a_n-a_{n+1}) \gamma_k \right) V^2(x),
\end{align*}
where we used Lemma \ref{lemma:D.1a}. Now, summing up over $k$ yields:
\begin{align*}
(a) & \le Ca_{n+1}^{-3}e^{C_1 a_{n+1}^{-2}}e^{-\rho_{n+1} T_{n+1}}[f]_{\text{Lip}} V^2(x) \sum_{k=N(T_n)+2}^{N(T_{n+1}-T)} ((a_n-a_{n-1}) + \gamma_k) \gamma_k e^{\rho_{n+1} \Gamma_k} \\
& \le Ca_{n+1}^{-3}e^{C_1 a_{n+1}^{-2}}e^{-\rho_{n+1} T_{n+1}}[f]_{\text{Lip}} ((a_n-a_{n-1}) + \gamma_{N(T_n)}) V^2(x) \sum_{k=N(T_n)+2}^{N(T_{n+1}-T)} \gamma_k e^{\rho_{n+1} \Gamma_{k-1}} \\
& \le Ca_{n+1}^{-3}e^{C_1 a_{n+1}^{-2}}e^{-\rho_{n+1} T_{n+1}}[f]_{\text{Lip}}((a_n-a_{n-1}) + \gamma_{N(T_n)})V^2(x) \int_{T_n}^{T_{n+1}-T} e^{\rho_{n+1} u} du \\
& \le Ca_{n+1}^{-3}e^{C_1 a_{n+1}^{-2}}[f]_{\text{Lip}}((a_n-a_{n-1}) + \gamma_{N(T_n)})V^2(x) \rho_{n+1}^{-1} \\
& \le Ca_{n+1}^{-3}e^{C_1 a_{n+1}^{-2}}[f]_{\text{Lip}}(a_n-a_{n-1})V^2(x) \rho_{n+1}^{-1},
\end{align*}
where we used that $(e^{\rho_{n+1} \gamma_k})_{n,k \ge 0}$ is bounded and Lemma \ref{lemma:app:gamma:1} in the last inequality.
We obtain likewise
$$ (c^{\text{init}}) \le Ce^{C_1 a_{n+1}^{-2}}e^{-\rho_{n+1}(T_{n+1}-T_n)}[f]_{\text{Lip}}a_{n+1}^{-3}(a_n-a_{n+1})\gamma_{N(T_n)+1} V^2(x) .$$

\medskip

$\bullet$ \textbf{Term $(b)$:} Applying Proposition \ref{prop:3.6:Y:bar} yields:
\begin{align*}
(b) & \le C a_{n+1}^{-3} \left(\gamma_{N(T_{n+1}-T)} + \sqrt{\gamma_{N(T_{n+1}-T)}}(a_n-a_{n+1})\right) [f]_{\text{Lip}} V^2(x) \sum_{k=N(T_{n+1}-T)+1}^{N(T_{n+1})-1} \frac{\gamma_k}{T_{n+1}-\Gamma_k} \\
& \quad + C a_{n+1}^{-2} \left(\gamma_{N(T_{n+1}-T)} + (a_n-a_{n+1})\right) [f]_{\text{Lip}} V^2(x) \sum_{k=N(T_{n+1}-T)+1}^{N(T_{n+1})-1} \frac{\gamma_k}{\sqrt{T_{n+1}-\Gamma_k}} \\
& \le C a_{n+1}^{-3} \left(\gamma_{N(T_{n+1}-T)} + \sqrt{\gamma_{N(T_{n+1}-T)}}(a_n-a_{n+1})\right) [f]_{\text{Lip}} V^2(x) \int_{T_{n+1}-T}^{T_{n+1}-\gamma_{N(T_{n+1})}} \frac{1}{T_{n+1}-u} du \\
& \quad + C a_{n+1}^{-2} \left(\gamma_{N(T_{n+1}-T)} + (a_n-a_{n+1})\right) [f]_{\text{Lip}} V^2(x) \int_{T_{n+1}-T}^{T_{n+1}-\gamma_{N(T_{n+1})}} \frac{1}{\sqrt{T_{n+1}-u}} du \\
& \le C a_{n+1}^{-3} \left(\gamma_{N(T_{n+1}-T)} + \sqrt{\gamma_{N(T_{n+1}-T)}}(a_n-a_{n+1})\right) [f]_{\text{Lip}} V^2(x) \log(1/\gamma_{N(T_{n+1})}) \\
& \quad + C a_{n+1}^{-2} (a_n-a_{n+1}) [f]_{\text{Lip}} V^2(x).
\end{align*}
Using Lemma \ref{lemma:app:gamma:2} in Appendix, $\sqrt{\gamma_{N(T_{n+1}-T)}} \log(1/\gamma_{N(T_{n+1})}) \le C \sqrt{\gamma_{N(T_{n+1})}} \log(1/\gamma_{N(T_{n+1})}) \to 0$ and using Lemma \ref{lemma:app:gamma:1} we also have
$$ \gamma_{N(T_{n+1}-T)} \log(1/\gamma_{N(T_{n+1})}) \le C\gamma_{N(T_{n+1})}^{1-\varepsilon} = o\left( n^{-1-\beta'} \right) = o(a_n - a_{n+1})$$
where $\beta'>0$ for small enough $\varepsilon$. So that
$$ (b) \le C a_{n+1}^{-3} (a_n-a_{n+1}) [f]_{\text{Lip}} V^2(x) .$$

\medskip

$\bullet$ \textbf{Term $(c^{\text{end}})$:} Using Lemma \ref{lemma:3.4.b:Y:bar} and $\gamma^{\text{end}} \le \gamma_{N(T_{n+1})+1} \le \gamma_{N(T_n)}$ yields:
\begin{align*}
|(P^{\bar{Y}}_{\gamma^{\text{end}}+\gamma_{N(T_{n+1})},\Gamma_{N(T_{n+1})-1}} - P^{X,n}_{\gamma^{\text{end}}+\gamma_{N(T_{n+1})}}) f(x)|
\le C[f]_{\text{Lip}} \left(\sqrt{\gamma_{N(T_n)}}(a_n-a_{n+1})+\gamma_{N(T_n)}\right) V^{1/2}(x).
\end{align*}
Then we integrate with respect to $P^{\bar{Y}}_{\gamma^{\text{init}},T_n} \circ P^{\bar{Y}}_{\gamma_{N(T_n)+2},\Gamma_{N(T_n)+1}} \circ \cdots \circ P^{\bar{Y}}_{\gamma_{k-1},\Gamma_{k-2}}$ and apply Lemma \ref{lemma:D.1a}.

\medskip

$\bullet$ So we have finally that $| \mathbb{E}f(X_{T_{n+1}-T_n}^{x,n}) - \mathbb{E}f(\bar{Y}_{T_{n+1}-T_n,T_n}^{x})|$ is bounded by
$$ C a_{n+1}^{-3} [f]_{\text{Lip}}(a_n-a_{n+1}) e^{C_1 a_{n+1}^{-2}}\rho_{n+1}^{-1} V^2(x) ,$$
which implies that, for every $x \in \mathbb{R}^d$,
$$ \mathcal{W}_1([X_{T_{n+1}-T_n}^{x,n}], [\bar{Y}_{T_{n+1}-T_n,T_n}^x]) \le C a_{n+1}^{-3} (a_n-a_{n+1}) e^{C_1 a_{n+1}^{-2}}\rho_{n+1}^{-1} V^2(x) .$$
We integrate this inequality with respect to the laws of $X^{x_0}_{T_n}$ and $\bar{Y}_{T_n}^{x_0}$ and obtain, temporarily setting $x_n := X^{x_0}_{T_n}$ and $\bar{y}_n := \bar{Y}_{T_n}^{x_0}$,
\begin{align*}
\mathcal{W}_1([X^{x_0}_{T_{n+1}}] & , [\bar{Y}^{x_0}_{T_{n+1}}]) = \mathcal{W}_1([X_{T_{n+1}-T_n}^{x_n,n}], [\bar{Y}_{T_{n+1}-T_n,T_n}^{\bar{y}_n}]) \\
& \le \mathcal{W}_1([X_{T_{n+1}-T_n}^{x_n,n}], [X_{T_{n+1}-T_n}^{\bar{y}_n,n}]) + \mathcal{W}_1([X_{T_{n+1}-T_n}^{\bar{y}_n,n}], [\bar{Y}_{T_{n+1}-T_n,T_n}^{\bar{y}_n}]) \\
& \le Ce^{C_1 a_{n+1}^{-2}} e^{-\rho_{n+1} (T_{n+1}-T_n)} \mathcal{W}_1([X^{x_0}_{T_n}], [\bar{Y}_{T_n}^{x_0}]) + Ca_{n+1}^{-3} (a_n-a_{n+1}) e^{C_1 a_{n+1}^{-2}}\rho_{n+1}^{-1} \mathbb{E} V^2(\bar{Y}_{T_n}^{x_0}) \\
& \le Ce^{C_1 a_{n+1}^{-2}} e^{-\rho_{n+1} (T_{n+1}-T_n)} \mathcal{W}_1([X^{x_0}_{T_n}], [\bar{Y}_{T_n}^{x_0}]) + Ca_{n+1}^{-3} (a_n-a_{n+1}) e^{C_1 a_{n+1}^{-2}}\rho_{n+1}^{-1}V^2(x_0) \\
& =: \mu_{n+1} \mathcal{W}_1([X^{x_0}_{T_n}], [\bar{Y}_{T_n}^{x_0}]) + v_{n+1}V^2(x_0),
\end{align*}
where $\mu_n$ is defined in \eqref{eq:def_mu} and where we used again Lemma \ref{lemma:D.1a}.
We use Lemma \ref{lemma:app:a_n_diff} to bound $(a_n - a_{n+1})$ and owing to \eqref{eq:hyp_A_Y_bar} we have $v_n \to 0$, so is bounded.
We iterate this inequality and obtain
\begin{align*}
\mathcal{W}_1([X_{T_{n+1}}^{x_0}],[\bar{Y}_{T_{n+1}}^{x_0}]) & \le CV^2(x_0) \left(v_{n+1} + \mu_{n+1} v_n + \mu_{n+1} \mu_n v_{n-1} + \cdots + \mu_{n+1} \cdots \mu_2 v_1 \right) \\
& \le CV^2(x_0) \left( v_{n+1} + Cn\mu_{n+1} \right).
\end{align*}
But following \eqref{eq:def_mu} we have $n \mu_n = O(v_n)$ so that
$$ \mathcal{W}_1([X_{T_{n+1}}^{x_0}],[\bar{Y}_{T_{n+1}}^{x_0}]) \le CV^2(x_0) v_{n+1} \le \frac{CV^2(x_0)}{(n+1)^{1-(\beta+1)(C_1+C_2)/A^2}} .$$
Moreover, owing to \eqref{eq:hyp_A_Y_bar} and combining with Theorem \ref{thm:conv_X} we get
$$ \mathcal{W}_1([\bar{Y}^{x_0}_{T_n}],\nu_{a_n}) \le \mathcal{W}_1([\bar{Y}^{x_0}_{T_n}],[X^{x_0}_{T_n}]) + \mathcal{W}_1([X^{x_0}_{T_n}],\nu_{a_n}) \le \frac{C\max(1+|x_0|,V^2(x_0))}{n^{1-(\beta+1)(C_1+C_2)/A^2}} $$
and
$$  \mathcal{W}_1([\bar{Y}^{x_0}_{T_n}],\nu^\star) \le \mathcal{W}_1([\bar{Y}^{x_0}_{T_n}],[X^{x_0}_{T_n}]) + \mathcal{W}_1([X^{x_0}_{T_n}],\nu^\star) \le Ca_n \max(1+|x_0|,V^2(x_0)) .$$

\medskip

Finally, to prove that $\mathcal{W}_1([\bar{Y}^{x_0}_t],\nu^\star) \rightarrow 0$ as $t \to \infty$, we conclude as in the end of Section \ref{subsec:proof_Y}.

\end{proof}

\section{Convergence of the Euler-Maruyama scheme with plateau}
\label{sec:bar-X}

In this section, we consider the Euler-Maruyama scheme for $(X_t)$, that is
\begin{align*}
\bar{X}_0^{x_0} = x_0, \quad \bar{X}_{\Gamma_{k+1}}^{x_0} = \bar{X}_{\Gamma_k} + \gamma_{k+1} b_{a_{n+1}}(\bar{X}_{\Gamma_k}) + a_{n+1} \sigma(\bar{X}_{\Gamma_k})(W_{\Gamma_{k+1}} - W_{\Gamma_k})
\end{align*}
for $k \in \lbrace N(T_n), \ldots, N(T_{n+1}) - 1 \rbrace$.
We also define as in Section \ref{sec:bar-Y} the genuine time-continuous scheme and the Euler-Maruyama scheme for $(X^{x,n}_t)_t$ so that $\bar{X}^{x_0,0} = \bar{X}^x$.

Although we already proved the convergence of the Euler-Maruyama scheme for $(Y_t)$, we shall also prove the convergence of the present scheme, since this algorithm is also used by practitioners within the framework of batch methods.

\begin{theorem}
\label{thm:main:3}
Assume \eqref{Eq:eq:min_V}, \eqref{Eq:eq:V_assumptions}, \eqref{Eq:eq:sigma_assumptions}, \eqref{eq:ellipticity} and \eqref{Eq:eq:V_confluence}. Assume furthermore \eqref{Eq:eq:gamma_assumptions} and \eqref{Eq:eq:gamma_assumptions_2}, that $V$ is $\mathcal{C}^3$ with $\|\nabla^3 V\| \le CV^{1/2}$ and that $\sigma$ is $\mathcal{C}^3$ with $\|\nabla^3(\sigma \sigma^\top)\| \le CV^{1/2}$. Then for large enough $A>0$ and for every $x_0 \in \mathbb{R}^d$,
$$ \mathcal{W}_1(\bar{X}^{x_0}_t, \nu^\star) \underset{t \rightarrow \infty}{\longrightarrow} 0 .$$
\end{theorem}
The proof of this theorem is given in the Supplementary Material.

\section{Experiments}
\label{sec:experiments}

In this section, we compare the performances of adaptive Langevin-Simulated Annealing algorithms versus vanilla SGLD, that is the Langevin algorithm with constant (additive) $\sigma$\footnote{Our code is available at \url{https://github.com/Bras-P/langevin-simulated-annealing}.}.
We train an artificial neural network on the MNIST dataset \cite{mnist}, which is composed of grayscale images of size $28 \times 28$ of handwritten digits (from 0 to 9). The goal is to recognize the handwritten digit and to classify the images. 60000 images are used for training and 10000 images are used for test.

We consider a feedforward neural network with two hidden dense layers with 128 units each and with ReLU activation.
For the adaptive Langevin algorithms, we choose the function $\sigma$ as a diagonal matrix which is the square root of the preconditioner in RMSprop \cite{li2015}, in Adam \cite{adam} and in Adadelta \cite{adadelta} respectively (see also Section \ref{subsec:practitioner}), giving L-RMSprop, L-Adam and L-Adadelta respectively.
The results are given in Figure \ref{fig:MNIST:langevin} and in Table \ref{table:MNIST:langevin}.

As pointed out in the literature (see the references Section \ref{subsec:practitioner}), the preconditioned Langevin algorithms show significant improvement compared with the vanilla SGLD algorithm. The convergence is faster and they achieve a lower error on the test set. We also display the value of the loss function on the train set during the training to show that the better performances of the preconditioned algorithms are not due to some overfitting effect.


\begin{figure}
\centering

\begin{tikzpicture}
\begin{axis}[
	xlabel=Epochs,
	ylabel=Test accuracy,
	xmin=0,
	ymin=0.9,
	ymax=0.99,
	legend style={at={(1,0)},anchor=south east},
	grid=both,
	minor grid style={gray!25},
	major grid style={gray!25},
	width=0.85\linewidth,
	height=0.16\paperheight,
	line width=1pt,
	mark size=2.5pt,
	mark options={solid},
	]
\addplot[color=black, mark=*] %
	table[x=time,y=f,col sep=comma]{./data_langevin1/SGLDOptimizer.csv};
\addlegendentry{SGLD};
\addplot[color=blue, mark=*, style=densely dotted] %
	table[x=time,y=f,col sep=comma]{./data_langevin1/pSGLDOptimizer.csv};
\addlegendentry{L-RMSprop};
\addplot[color=red, mark=diamond*, style=densely dashed] %
	table[x=time,y=f,col sep=comma]{./data_langevin1/aSGLDOptimizer.csv};
\addlegendentry{L-Adam};
\addplot[color=olive, mark=square*, style=densely dashdotted] %
	table[x=time,y=f,col sep=comma]{./data_langevin1/pAdaSGLDOptimizer.csv};
\addlegendentry{L-Adadelta};
\end{axis}
\end{tikzpicture}

\begin{tikzpicture}
\begin{axis}[
	xlabel=Epochs,
	ylabel=Train loss,
	xmin=0,
	ymin=0,
	ymax=0.5,
	grid=both,
	minor grid style={gray!25},
	major grid style={gray!25},
	width=0.85\linewidth,
	height=0.16\paperheight,
	line width=1pt,
	mark size=2.5pt,
	mark options={solid},
	]
\addplot[color=black, mark=*] %
	table[x=time,y=f,col sep=comma]{./data_langevin1/SGLDOptimizer_loss.csv};
\addlegendentry{SGLD};
\addplot[color=blue, mark=*, style=densely dotted] %
	table[x=time,y=f,col sep=comma]{./data_langevin1/pSGLDOptimizer_loss.csv};
\addlegendentry{L-RMSprop};
\addplot[color=red, mark=diamond*, style=densely dashed] %
	table[x=time,y=f,col sep=comma]{./data_langevin1/aSGLDOptimizer_loss.csv};
\addlegendentry{L-Adam};
\addplot[color=olive, mark=square*, style=densely dashdotted] %
	table[x=time,y=f,col sep=comma]{./data_langevin1/pAdaSGLDOptimizer_loss.csv};
\addlegendentry{L-Adadelta};
\end{axis}
\end{tikzpicture}

\caption{\textit{Performance of preconditioned Langevin algorithms compared with vanilla SGLD on the MNIST dataset. The values of the hyperparameters are $a(n) = A \log^{-1/2}(c_1n+e)$ with $A=2.10^{-3}$ and where $c_1n=1$ after $5$ epochs; $\gamma_n=\gamma_1/(1+c_2n)$ where $c_2 n=1$ after $5$ epochs and where for SGLD, $\gamma_1=0.001$ for L-RMSprop and L-Adam and $\gamma_1=0.1$ for L-Adadelta.}}
\label{fig:MNIST:langevin}
\end{figure}
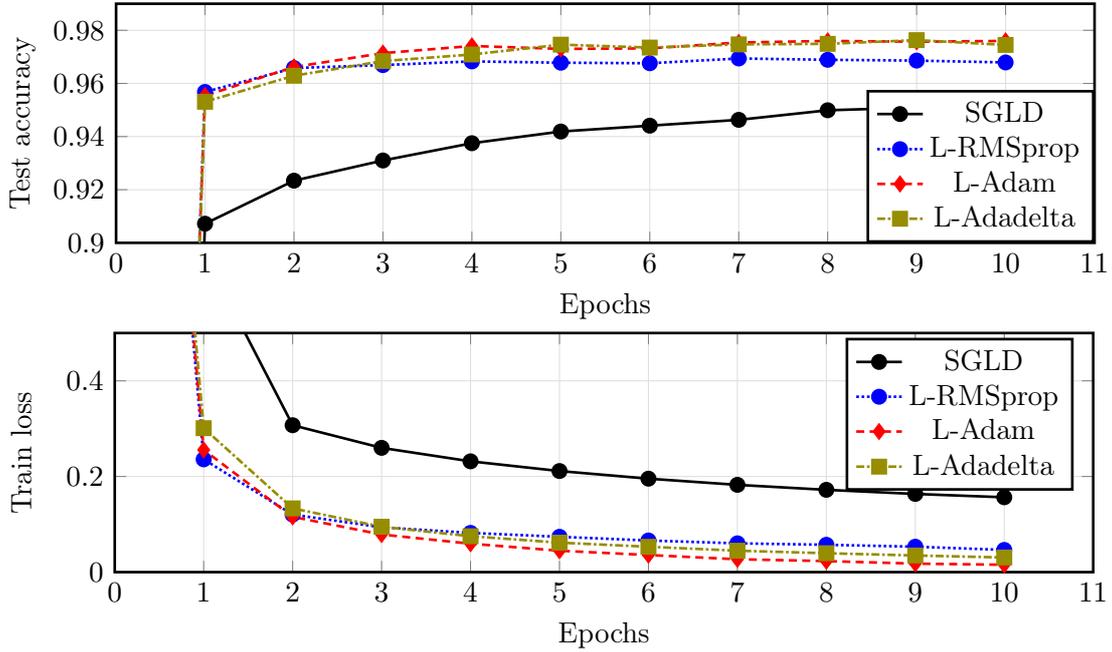

\begin{table}
\centering
\begin{tabular}{ccccc}
\hline
Preconditioner & SGLD & L-RMSprop & L-Adam & L-Adadelta  \\ 
\hline 
Best accuracy & 95,24 \% & 96,94 \% & 97,60 \% & 97,63 \%
\end{tabular}
\caption{\textit{Best accuracy performance on the MNIST test set after 10 epochs.}}
\label{table:MNIST:langevin}
\end{table}

We also compare preconditioned Langevin algorithms with their respective non-Langevin counterpart. For shallow neural networks, adding an exogenous noise does not seem to improve significantly the performances of the optimization algorithm. However, for deep neural networks, which are highly non-linear and which loss function has many local minima, the Langevin version is competitive with the currently widely used non-Langevin algorithms and can even lead to improvement. The results are given in Figure \ref{fig:MNIST:deep} where we used a deep neural network with 20 hidden layers with 32 units each and with ReLU activation.

\begin{figure}
\centering

\begin{tikzpicture}
\begin{axis}[
	xlabel=Epochs,
	ylabel=Test accuracy,
	xmin=0,
	ymin=0.7,
	ymax=1,
	legend style={at={(1,0)},anchor=south east},
	grid=both,
	minor grid style={gray!25},
	major grid style={gray!25},
	width=0.45\linewidth,
	height=0.15\paperheight,
	line width=1pt,
	mark size=2.5pt,
	mark options={solid},
	]
\addplot[color=black, mark=*] %
	table[x=time,y=f,col sep=comma]{./data_langevin1/SGD_deep.csv};
\addlegendentry{SGD};
\addplot[color=blue, mark=*, style=densely dotted] %
	table[x=time,y=f,col sep=comma]{./data_langevin1/aSGLDOptimizer_deep.csv};
\addlegendentry{Adam};
\addplot[color=red, mark=diamond*, style=densely dashed] %
	table[x=time,y=f,col sep=comma]{./data_langevin1/aSGLDOptimizer_deep_langevin.csv};
\addlegendentry{L-Adam};
\end{axis}
\end{tikzpicture}%
~
\begin{tikzpicture}
\begin{axis}[
	xlabel=Epochs,
	xmin=0,
	ymin=0.7,
	ymax=1,
	legend style={at={(1,0)},anchor=south east},
	grid=both,
	minor grid style={gray!25},
	major grid style={gray!25},
	width=0.45\linewidth,
	height=0.15\paperheight,
	line width=1pt,
	mark size=2.5pt,
	mark options={solid},
	]
\addplot[color=black, mark=*] %
	table[x=time,y=f,col sep=comma]{./data_langevin1/SGD_deep.csv};
\addlegendentry{SGD};
\addplot[color=blue, mark=*, style=densely dotted] %
	table[x=time,y=f,col sep=comma]{./data_langevin1/pSGLDOptimizer_deep.csv};
\addlegendentry{RMSprop};
\addplot[color=red, mark=diamond*, style=densely dashed] %
	table[x=time,y=f,col sep=comma]{./data_langevin1/pSGLDOptimizer_deep_langevin.csv};
\addlegendentry{L-RMSprop};
\end{axis}
\end{tikzpicture}

\caption{\textit{Side-by-side comparison of optimization algorithms with their respective Langevin counterparts for the training of a deep neural network on the MNIST dataset. We display the performance of SGD for reference. The values of the hyperparameters are $a(n) = A \log^{-1/2}(c_1n+e)$ with $A=1.10^{-3}$ for L-Adam and $A=5.10^{-4}$ for L-RMSprop and where $c_1n=1$ after $5$ epochs; $\gamma_n=\gamma_1/(1+c_2n)$ where $c_2 n=1$ after $5$ epochs and where $\gamma_1=0.01$ for SGLD and $\gamma_1=0.001$ for the others.}}
\label{fig:MNIST:deep}
\end{figure}
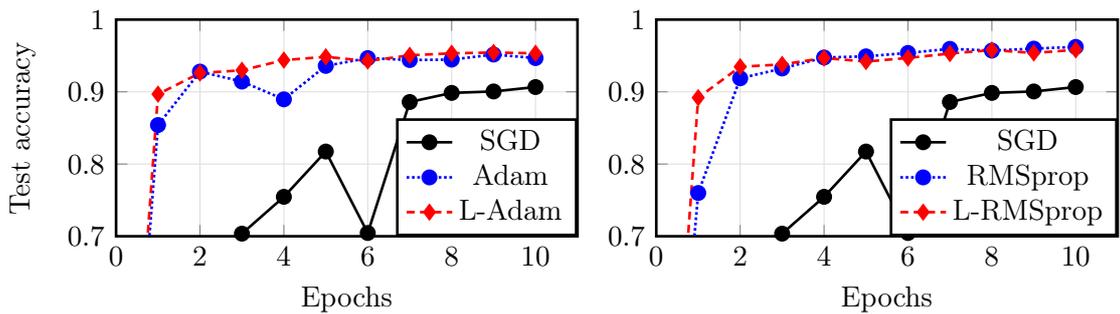

In order to understand how sensitive are these methods to poor initialization, we run an experiment on the previous deep neural network where all the weights are initialized to zero, as in \cite[Section 4.1]{neelakantan2015}. We plot the accuracy on the test set in Figure \ref{fig:MNIST:zeros}.
We observe that the non-Langevin optimizer needs some time before escaping from the neighbourhood of the initial point whereas in its Langevin version, the Gaussian noise is effective to rapidly escape from highly degenerated saddle points of the loss.


\begin{figure}
\centering

\begin{tikzpicture}
\begin{axis}[
	xlabel=Epochs $\times 10$,
	ylabel=Test accuracy,
	xmin=0,
	ymin=0,
	ymax=1,
	legend style={at={(1,0)},anchor=south east},
	grid=both,
	minor grid style={gray!25},
	major grid style={gray!25},
	width=0.85\linewidth,
	height=0.16\paperheight,
	line width=1pt,
	mark size=2.5pt,
	mark options={solid},
	]
\addplot[color=black, mark=*] %
	table[x=time,y=f,col sep=comma]{./data_langevin2/adam_zeros_data.csv};
\addlegendentry{Adam};
\addplot[color=blue, mark=square*] %
	table[x=time,y=f,col sep=comma]{./data_langevin2/aSGLD_zeros_data.csv};
\addlegendentry{L-Adam};
\end{axis}
\end{tikzpicture}

\caption{\textit{Performance of the Adam optimizer compared with its Langevin version at the beginning of the training of a deep neural network on the MNIST dataset with poor initialization. We record the accuracy on the test set 10 times per epoch.}}
\label{fig:MNIST:zeros}
\end{figure}
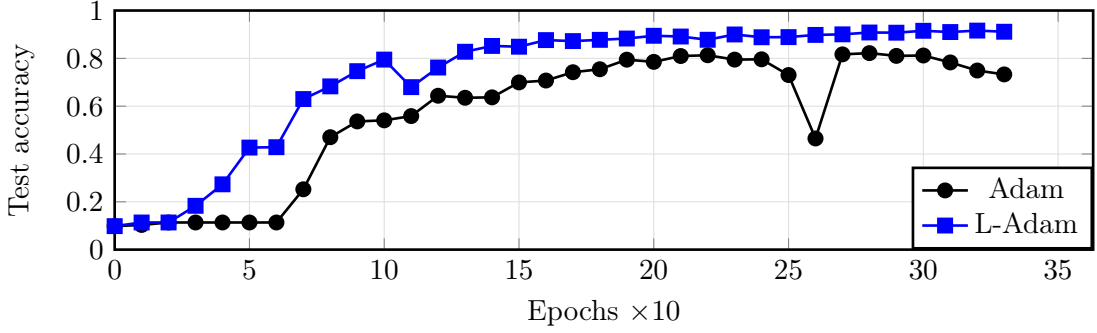

\appendix

\section{Appendix}

\begin{lemma}
\label{lemma:BDG}
Let $Z$ and $\widetilde{Z}$ be two continuous diffusion processes. Then for all $t \ge 0$ and for all $p \ge 2$:
$$ \left\| \int_0^t (\sigma(Z_s) - \sigma(\widetilde{Z}_s)) dW_s \right\|_p \le C^{BDG}_p [\sigma]_{\textup{Lip}} \left(\int_0^t \| Z_s - \widetilde{Z}_s \|_p^2 ds\right)^{1/2} ,$$
where $C^{BDG}_p$ is a constant which only depends on $p$.
\end{lemma}
\begin{proof}
It follows from the generalized Minkowski and the Burkholder-Davis-Gundy inequalities that
$$ \left\| \int_0^t (\sigma(Z_s) - \sigma(\widetilde{Z}_s)) dW_s \right\|_p \le C^{BDG}_p [\sigma]_{\text{Lip}} \left\| \int_0^t |Z_s - \widetilde{Z}_s|^2 ds \right\|_{p/2}^{1/2} \le C^{BDG}_p [\sigma]_{\text{Lip}} \left( \int_0^t \|Z_s - \widetilde{Z}_s\|^2_p ds \right)^{1/2}$$
\end{proof}

We now give some results on the step sequence $(\gamma_n)$ associated to the Euler-Maruyama scheme. Let us recall that the sequence $(T_n)$ is defined in \eqref{eq:def_T_n}.

\begin{lemma}
\label{lemma:app:tauber:1}
Let $(u_n)$ be a positive and non-increasing sequence such that $\sum_n u_n < \infty$. Then $u_n = o(n^{-1})$.
\end{lemma}
\begin{proof}
We have $Nu_{2N} \le \sum_{n=N}^{2N} u_n\to 0 $ as $N \to \infty$.
\end{proof}

\begin{lemma}
\label{lemma:app:gamma:1}
We have
\begin{equation}
\gamma_{N(T_n)} = o\left(n^{-(1+\beta)}\right).
\end{equation}
\end{lemma}
\begin{proof}
Using the previous lemma, $\gamma_n = o(n^{-1/2})$ so that $\Gamma_n = o(n^{1/2})$ and then $x^2 = o(N(x))$ as $x \rightarrow \infty$ and then
$ \gamma_{N(T_n)} = o\left(N(T_n)^{-1/2}\right) = o\left(n^{-(1+\beta)}\right).$
\end{proof}

\begin{lemma}
\label{lemma:app:gamma:2}
The sequence $(\gamma_{N(T_{n+1}-T)}/\gamma_{N(T_{n+1})})$ is bounded.
\end{lemma}
\begin{proof}
Using \eqref{Eq:eq:gamma_assumptions_2}, we have for $\varpi'>\varpi$ and for large enough $k$, $(\gamma_k-\gamma_{k+1})/\gamma_{k+1}^2 \le \varpi'$ so that $\gamma_k/\gamma_{k+1} \le 1 + \varpi' \gamma_{k+1}$ and then
\begin{align*}
\log\left(\frac{\gamma_{N(T_{n+1}-T))}}{\gamma_{N(T_{n+1})}}\right) & = \sum_{k=N(T_{n+1}-T)}^{N(T_{n+1})-1} \log\left(\frac{\gamma_k}{\gamma_{k+1}}\right) \le C \sum_{k=N(T_n)}^{N(T_{n+1})-1} \gamma_k = C\left(\Gamma_{N(T_{n+1})}-\Gamma_{N(T_{n+1}-T)}\right) \\
& \le C(T_{n+1}-(T_{n+1}-T)).
\end{align*}
\end{proof}

\newpage

\section{Supplementary Material}

\subsection{Proof of Proposition \ref{prop:W_nu}}

\begin{proof}
We have
$$\frac{\nu_{a_{n+1}}(x)}{\nu_{a_n}(x)} = \frac{\mathcal{Z}_{a_{n+1}}}{\mathcal{Z}_{a_n}} e^{-2(V(x)-V^\star)(a_{n+1}^{-2} - a_{n}^{-2})} \le \frac{\mathcal{Z}_{a_{n+1}}}{\mathcal{Z}_{a_n}} =: M_n.$$
We now consider $(P_i)_{1 \le i \le m^\star}$ a partition of $\mathbb{R}^d$ such that for all $i$, $x_i^\star \in \mathring{P_i}$.
Let us prove that for all $1 \le i \le m^\star$,
\begin{equation}
\label{eq:Z_a_equivalent:3}
\mathcal{Z}_{a,i}^{-1} := \int_{\mathbb{R}^d} e^{-2(V(x)-V^\star)/a^2} \mathds{1}_{x \in P_i} dx \underset{a \to 0}{\sim} a^d \int_{\mathbb{R}^d} e^{-x^\top \nabla^2 V(x_i^\star) x} dx .
\end{equation}
Let $r>0$ ; let us consider $\widetilde{V}_i$ defined as
$$ \widetilde{V}_i(x) = \left\lbrace \begin{array}{ll}
V(x) & \text{ if } x \in \textbf{B}(x_i^\star,r) \\
|x-x_i^\star|^2 + V^\star & \text{ otherwise} .
\end{array} \right.  $$
We also define $\widetilde{\mathcal{Z}}_{a,i}^{-1} := \int_{\mathbb{R}^d} e^{-2(\widetilde{V}_i(x)-V^\star)/a^2}\mathds{1}_{x \in P_i}dx$.
Then, owing to $V^\star >0$ and \eqref{Eq:eq:V_assumptions},
\begin{equation}
\label{eq:Z_a_equivalent:2}
\forall x \in \mathbb{R}^d, \ C|x-x_i^\star|^2 \le \widetilde{V}_i(x) - V^\star \le C'|x-x_i^\star|^2
\end{equation}
and then
$$ \widetilde{\mathcal{Z}}_{a,i}^{-1} = a^d \int_{\mathbb{R}^d} e^{-2(\widetilde{V}_i(ax+x_i^\star)-V^\star)/a^2} \mathds{1}_{x \in a^{-1}(P_i-x_i^\star)} dx 
\underset{a \rightarrow 0}{\sim} a^d \int_{\mathbb{R}^d} e^{-x^\top \nabla^2 V(x_i^\star) x} dx ,$$
where we get the equivalence by dominated convergence ; the domination comes from \eqref{eq:Z_a_equivalent:2}. Then
\begin{align*}
& \mathcal{Z}_{a,i}^{-1} - \widetilde{\mathcal{Z}}_{a,i}^{-1} = \int_{\textbf{B}(x_i^\star,r)^c} e^{-2(V(x)-V^\star)/a^2}\mathds{1}_{x \in P_i}dx - \int_{\textbf{B}(x_i^\star,r)^c} e^{-2(\widetilde{V}_i(x)-V^\star)/a^2}\mathds{1}_{x \in P_i}dx =: I_1 - I_2, \\
& I_2 = a^d \int_{\textbf{B}(0,r/a)^c} e^{-2|x|^2} \mathds{1}_{x \in a^{-1}(P_i-x_i^\star)} dx \le a^d \int_{\textbf{B}(0,r/a)^c} e^{-2|x|^2}dx = o(a^d) = o\left(\widetilde{\mathcal{Z}}_{a,i}^{-1}\right).
\end{align*}
Moreover using \cite[Proposition 1]{bras2021} we have $\mathcal{Z}_{a,i} I_1 \to 0$ as $a\to 0$, so that
$$ \mathcal{Z}_{a,i}^{-1} = \widetilde{\mathcal{Z}}_{a,i}^{-1} + o\left(\mathcal{Z}_{a,i}^{-1}\right) + o\left( \widetilde{\mathcal{Z}}_{a,i}^{-1}\right) \sim  \widetilde{\mathcal{Z}}_{a,i}^{-1} ,$$
which proves \eqref{eq:Z_a_equivalent:3} and then
\begin{align}
\label{eq:Z_a_equivalent}
\mathcal{Z}_a^{-1} \underset{a \rightarrow 0}{\sim} a^d \sum_{i=1}^{m^\star} \int_{\mathbb{R}^d} e^{-x^\top \nabla^2 V(x_i^\star) x} dx .
\end{align}

We now prove that
\begin{equation}
\label{eq:app:2_a_diff}
\mathcal{Z}_{a_n}^{-1} - \mathcal{Z}_{a_{n+1}}^{-1} \le C a_{n+1}^{d-1}(a_n-a_{n+1}) .
\end{equation}
Indeed, by convexity we have for all $z \in \mathbb{R}$
\begin{align}
\label{eq:app:exp_x_2}
\left|e^{-2z/a_{n}^2} - e^{-2z/a_{n+1}^2} \right| & \le 2e^{-2z/a_{n}^2}z \left|\frac{1}{a_n^2} - \frac{1}{a_{n+1}^2}\right| \le 4 e^{-2z/a_{n}^2} \frac{z}{a_{n+1}^2} \frac{(a_n-a_{n+1})}{a_n}.
\end{align}
and then
\begin{align*}
& \mathcal{Z}_{a_n,i}^{-1} - \mathcal{Z}_{a_{n+1},i}^{-1} \\
& \quad = a_{n+1}^d \int_{\mathbb{R}^d} \left(e^{-2(V(a_{n+1}x+x_i^\star)-V^\star)/a_n^2} \mathds{1}_{x \in a_{n+1}^{-1} (P_i-x_i^\star)} - e^{-2(V(a_{n+1}x+x_i^\star)-V^\star)/a_{n+1}^2}\mathds{1}_{x \in a_{n+1}^{-1} (P_i-x_i^\star)}\right) dx \\
& \quad \le 4a_{n+1}^{d-1}(a_n - a_{n+1}) \underbrace{\int_{\mathbb{R}^d} e^{-2(V(a_{n+1}x+x_i^\star)-V^\star)/a_{n}^2} \frac{V(a_{n+1}x+x_i^\star)-V^\star}{a_{n+1}^2}\mathds{1}_{x \in a_{n+1}^{-1} (P_i-x_i^\star)} dx}_{:= I_3}.
\end{align*}
Let us also define
$$ \widetilde{I}_3 := \int_{\mathbb{R}^d} e^{-2(\widetilde{V}_i(a_{n+1}x+x_i^\star)-V^\star)/a_{n}^2} \frac{\widetilde{V}_i(a_{n+1}x+x_i^\star)-V^\star}{a_{n+1}^2}\mathds{1}_{x \in a_{n+1}^{-1} (P_i-x_i^\star)} dx .$$
Then $\widetilde{I_3}$ converges by dominated convergence and $|I_3 - \widetilde{I}_3|$ is bounded by
\begin{align*}
& \left|\int_{\mathbb{R}^d} \left(e^{-2(V(a_{n+1}x+x_i^\star)-V^\star)/a_{n}^2} \frac{V(a_{n+1}x+x_i^\star)-V^\star}{a_{n+1}^2} - e^{-2(\widetilde{V}_i(a_{n+1}x+x_i^\star)-V^\star)/a_{n}^2} \frac{\widetilde{V}_i(a_{n+1}x+x_i^\star)-V^\star}{a_{n+1}^2}\right) \right. \\
& \quad \quad \quad \mathds{1}_{x \in a_{n+1}^{-1} (P_i-x_i^\star)} dx \Big| \\
& \quad \le a_{n+1}^{-d-2} \int_{\textbf{B}(x_i^\star,r)^c} e^{-2(V(x)-V^\star)/a_{n}^2} (V(x)-V^\star) \mathds{1}_{x \in P_i} dx \\
& \quad \quad + \int_{\textbf{B}(0,r/a_{n+1})^c} e^{-2(\widetilde{V}_i(a_{n+1}x+x_i^\star)-V^\star)/a_{n}^2} \frac{\widetilde{V}_i(a_{n+1}x+x_i^\star)-V^\star}{a_{n+1}^2} \mathds{1}_{x \in a_{n+1}^{-1} (P_i-x_i^\star)} dx.
\end{align*}
The second integral converges to $0$ by dominated convergence by similar arguments as for $I_2$. Moreover we have for every $x \in \textbf{B}(x_i^\star,r)^c \cap P_i$, $V(x)-V^\star \ge \varepsilon$ for some $\varepsilon > 0$ and then for $n$ such that $a_n \le A/\sqrt{2}$:
\begin{align*}
& a_{n+1}^{-d-2} \int_{\textbf{B}(x_i^\star,r)^c} e^{-2(V(x)-V^\star)/a_{n}^2} (V(x)-V^\star) \mathds{1}_{x \in P_i} dx \\
& \quad \le Ca_{n+1}^{-d-2} \int_{\textbf{B}(x_i^\star,r)^c} e^{-2(V(x)-V^\star)/a_{n}^2} |x-x_i^\star|^2 \mathds{1}_{x \in P_i} dx \\
& \quad \le Ca_{n+1}^{-d-2} e^{-\varepsilon/a_n^2} \int_{\textbf{B}(x_i^\star,r)^c} e^{-(V(x)-V^\star)/a_{n}^2} |x-x_i^\star|^2 \mathds{1}_{x \in P_i} dx \\
& \quad \le Ca_{n+1}^{-d-2} e^{-\varepsilon/a_n^2} \int_{\mathbb{R}^d} e^{-2(V(x)-V^\star)/A^2} |x-x_i^\star|^2 dx \underset{n \to \infty}{\longrightarrow} 0,
\end{align*}
where we used that $(x \mapsto |x|^2 e^{-2(V(x)-V^\star)/A^2}) \in L^1(\mathbb{R}^d)$. Then we obtain that $I_3$ converges to $\widetilde{I}_3$, which proves \eqref{eq:app:2_a_diff}.
Then we have
$$ 1 - M_n^{-1} = \frac{\mathcal{Z}_{a_n}^{-1}-\mathcal{Z}_{a_{n+1}}^{-1}}{\mathcal{Z}_{a_n}^{-1}} \le C\frac{a_n-a_{n+1}}{a_n} \le \frac{C}{n\log(n)}. $$
%
On the other hand, if $X \sim \nu_{a_{n+1}}$, $\tilde{X} \sim \nu_{a_{n+1}}$, $Y \sim \nu_{a_{n}}$ and $X$, $\tilde{X}$ and $Y$ are mutually independent then
\begin{align*}
& \left| \mathbb{E}|X-Y| - \mathbb{E}|X - \tilde{X}| \right| \\
& = \Big| a_{n+1}^d \mathcal{Z}_{a_n} a_{n+1}^d \mathcal{Z}_{a_{n+1}} \sum_{i,j=1}^{m^\star} \int \int a_{n+1}|x - y| e^{-2(V(a_{n+1}x+x_i^\star)-V^\star)/a_{n+1}^2} e^{-2(V(a_{n+1}y+x_i^\star)-V^\star)/a_{n}^2} \\
& \qquad \qquad \qquad \qquad \qquad \qquad \qquad \mathds{1}_{x \in a_{n+1}^{-1} (P_i-x_i^\star)} \mathds{1}_{y \in a_{n+1}^{-1} (P_j-x_i^\star)} dxdy \\
& \quad \quad - (a_{n+1}^d \mathcal{Z}_{a_{n+1}})^2 \sum_{i,j=1}^{m^\star} \int \int a_{n+1} |x - y| e^{-2(V(a_{n+1}x+x_i^\star)-V^\star)/a_{n+1}^2} e^{-2(V(a_{n+1}y+x_i^\star)-V^\star)/a_{n+1}^2} \\
& \qquad \qquad \qquad \qquad \qquad \qquad \qquad \mathds{1}_{x \in a_{n+1}^{-1} (P_i-x_i^\star)} \mathds{1}_{y \in a_{n+1}^{-1} (P_j-x_i^\star)} dxdy \Big| \\
%
%
%
%
& \quad = a_{n+1}^{2d+1}\mathcal{Z}_{a_{n+1}} \sum_{i,j=1}^{m^\star} \int \int |x-y| e^{-2(V(a_{n+1}x+x_i^\star)-V^\star)/a_{n+1}^2} \\
& \quad \quad \cdot \left|\mathcal{Z}_{a_n} e^{-2(V(a_{n+1}y+x_i^\star)-V^\star)/a_{n}^2} - \mathcal{Z}_{a_{n+1}} e^{-2(V(a_{n+1}y+x_i^\star)-V^\star)/a_{n+1}^2} \right| \mathds{1}_{x \in a_{n+1}^{-1} (P_i-x_i^\star)} \mathds{1}_{y \in a_{n+1}^{-1} (P_j-x_i^\star)} dxdy \\
& \quad \le a_{n+1}\left(a_{n+1}^{2d}\mathcal{Z}_{a_{n+1}}^2\right) \sum_{i,j=1}^{m^\star} \int \int |x-y| e^{-2(V(a_{n+1}x+x_i^\star)-V^\star)/a_{n+1}^2} \\
& \quad \quad \quad \quad \cdot \left|e^{-2(V(a_{n+1}y+x_i^\star)-V^\star)/a_{n}^2} - e^{-2(V(a_{n+1}y+x_i^\star)-V^\star)/a_{n+1}^2} \right| \mathds{1}_{x \in a_{n+1}^{-1} (P_i-x_i^\star)} \mathds{1}_{y \in a_{n+1}^{-1} (P_j-x_i^\star)} dxdy \\
& \quad \quad + a_{n+1} \left(a_{n+1}^{2d}\mathcal{Z}_{a_{n+1}}^2\right) \sum_{i,j=1}^{m^\star} \int \int |x-y| e^{-2(V(a_{n+1}x+x_i^\star)-V^\star)/a_{n+1}^2} e^{-2(V(a_{n+1}y+x_i^\star)-V^\star)/a_{n}^2} \\
& \quad \quad \quad \quad \cdot \left|1 - \frac{\mathcal{Z}_{a_{n}}}{\mathcal{Z}_{a_{n+1}}}\right| \mathds{1}_{x \in a_{n+1}^{-1} (P_i-x_i^\star)} \mathds{1}_{y \in a_{n+1}^{-1} (P_j-x_i^\star)} dxdy.
\end{align*}
So using \eqref{eq:app:exp_x_2}, dominated convergence as for the proof of \eqref{eq:Z_a_equivalent}, \eqref{eq:Z_a_equivalent} itself with \eqref{eq:a_n_diff} and the bound for $1-\mathcal{Z}_{a_n}/\mathcal{Z}_{a_{n+1}} = 1-M_n^{-1}$ we have
\begin{align*}
& \limsup_{n \to \infty}\left[ n\log^{3/2}(n) \left| \mathbb{E}|X-Y| - \mathbb{E}|X - \tilde{X}| \right| \right] \\
& \quad \le C \sum_{i=1}^{m^\star} \int \int |x-y| e^{-x^\top \nabla^2 V(x_i^\star) x} e^{-y^\top \nabla^2 V(x_i^\star) y} \left(1+y^\top \nabla^2 V(x_i^\star) y\right) dx dy.
\end{align*}
So that using Lemma \ref{lemma:acceptance_rejection} and the fact that $\mathbb{E}|X-\tilde{X}|$ is of order $a_n$ we have
\begin{align*}
\mathcal{W}_1(\nu_{a_n},\nu_{a_{n+1}}) & \le \mathbb{E}|X-Y| - \frac{1}{M_n}\mathbb{E}|X-\tilde{X}| \le \mathbb{E}|X-Y| - \mathbb{E}|X - \tilde{X}| + \frac{C}{n\log(n)}\mathbb{E}|X-\tilde{X}| \\
& \le \frac{C}{n\log^{3/2}(n)}.
\end{align*}

The proof for the second claim is similar.
\end{proof}

\subsection{Proof of Proposition \eqref{prop:3.5:Y:bar}}

\begin{proof}
As in the proof of \cite[Proposition 3.5]{pages2020}, we split $ | \mathbb{E}[g(\bar{Y}^x_{\gamma,u})] - \mathbb{E}\left[g(X^{x,n}_{\gamma})\right]|$ into four terms $A_1$, $A_2$, $A_3$ and $A_4$, that is, by the Taylor formula, for every $y$, $z \in \mathbb{R}^d$,
$$
g(z)-g(y)=\langle \nabla g(y)|z-y \rangle +\int_0^1 (1-u)\nabla^2g\left(uz+(1-u)y  \right)du  (z-y)^{\otimes 2}.
$$
For a given $x\in\mathbb{R}^d$, it follows that
\begin{align*}
g(z)-g(y) & = \langle \nabla g(x)|z-y \rangle + \langle \nabla g(y)-\nabla g(x)|z-y \rangle +\int_0^1 (1-u)\nabla^2g\left(uz+(1-u)y  \right)  (z-y)^{\otimes 2}du\\
&=\langle\nabla g(x)|z-y \rangle +\langle \nabla^2 g(x) (y-x) |z-y \rangle\\
&\quad + \int_0^1(1-u)\nabla^3 g(uy+(1-u)x)(y-x)^{\otimes 2} (z-y)du\\
&\quad + \int_0^1 (1-u)\nabla^2 g \left(uz+(1-u)y  \right)du  (z-y)^{\otimes 2}.
\end{align*}
Applying this expansion with $y=X^{x,n}_{\gamma}$ and $z=\bar{Y}^x_{\gamma,u}$, this yields:
\begin{align*}
\mathbb{E}[&g(\bar{Y}^x_{\gamma,u})-g(X^{x,n}_{\gamma})]=\underbrace{\langle \nabla g(x)|\mathbb{E}[\bar{Y}^x_{\gamma,u} -X^{x,n}_{\gamma}]\rangle}_{=:A_1} + \underbrace{\mathbb{E}\left[\langle \nabla^2 g(x) (X^{x,n}_{\gamma}-x) |\bar{Y}^x_{\gamma,u}-X^{x,n}_{\gamma} \rangle \right]}_{=:A_2}\\
 &+\underbrace{\mathbb{E}\left[\int_0^1(1-u)\nabla^3 g(u X^{x,n}_{\gamma}+(1-u)x) (X^{x,n}_{\gamma}-x)^{\otimes 2} (\bar{Y}^x_{\gamma,u}-X^{x,n}_{\gamma})du\right]}_{=:A_3}\\
& +\underbrace{\int_0^1 (1-u) \mathbb{E} \left[\nabla^2g\left(u\bar{Y}^x_{\gamma,u}+(1-u)X^{x,n}_{\gamma}  \right)  (\bar{Y}^x_{\gamma,u}-X^{x,n}_{\gamma})^{\otimes 2}\right]du}_{=:A_4}.
 \end{align*}

$\bullet$ \textbf{Term} $A_1$:
The term $A_1$ is bounded by $|\nabla g(x)| \cdot |\mathbb{E}[\bar{Y}^x_{\gamma,u}-X^{x,n}_\gamma]|$, with
\begin{align*}
\mathbb{E}[\bar{Y}^x_{\gamma,u}-X^{x,n}_\gamma] & = \mathbb{E}\left[ \int_0^\gamma b_{a(u)}(x)-b_{a(u)}(X^{x,n}_s))ds \right] + \mathbb{E}\left[ \int_0^\gamma (b_{a(u)}(X^{x,n}_s) - b_{a_{n+1}}(X^{x,n}_s))ds \right] \\
& =: A_{11} + A_{12}.
\end{align*}
We have $|A_{12}| \le \gamma ||\Upsilon||_\infty (a_n^2 - a_{n+1}^2)$ and
\begin{align*}
|A_{11}| & = \left| \int_0^\gamma \int_0^s \mathbb{E}\left[\nabla b_{a(u)}(X^{x,n}_v)b_{a(u)}(X^{x,n}_v) + \frac{1}{2} \nabla^2 b_{a(u)}(X^{x,n}_v)a_{n+1}^2\sigma \sigma^\top (X^{x,n}_v) \right] dv \right| \\
& \le C\gamma^2 \sup_{v \in [0,\gamma]} \mathbb{E}[V^{1/2}(X^{x,n}_v)] \le C\gamma^2 V^{1/2}(x),
\end{align*}
where we used that $|\nabla b_a| \le C$ and $\|\nabla^2 b_a \| \le CV^{1/2}$ because we assumed $\|\nabla^3 V\| \le CV^{1/2}$ and $\|\nabla^3(\sigma \sigma^\top) \| \le CV^{1/2}$.

\medskip

$\bullet$ \textbf{Term} $A_2$: We have:
$$ |A_2| \le \sum_{1\le i,j \le d}|\partial_{ij}g(x)| |\mathbb{E}[(X^{x,n}_\gamma-x)_i(X_\gamma^{x,n}-\bar{Y}^x_{\gamma,u})_j]| $$
and we have
$$ \mathbb{E}[(X^{x,n}_\gamma-x)_i(X_\gamma^{x,n}-\bar{Y}^x_{\gamma,u})_j] = \mathbb{E}[(X^{x,n}_\gamma-\bar{Y}^x_{\gamma,u})_i(X_\gamma^{x,n}-\bar{Y}^x_{\gamma,u})_j] + \mathbb{E}[(\bar{Y}^x_{\gamma,u}-x)_i(X_\gamma^{x,n}-\bar{Y}^x_{\gamma,u})_j].  $$
Using Lemma \ref{lemma:3.4.b:Y:bar}, the first term of the right-hand side is bounded by $C(V^{1/2}(x)\gamma + \sqrt{\gamma}(a_n-a_{n+1}))^2$ and in the second term we write $(\bar{Y}^x_{\gamma,u}-x)_i = \left(\gamma b_{a(u)}(x) + \gamma \zeta_{k+1}(x) + a(u)\sigma(x)W_\gamma \right)_i$ and we have
$$ |\mathbb{E}[(\gamma b_{a(u)}(x) + \gamma \zeta_{k+1}(x))_i(X_\gamma^{x,n}-\bar{Y}^x_{\gamma,u})_j]| \le \gamma V^{1/2}(x) (V^{1/2}(x)\gamma + \sqrt{\gamma}(a_n-a_{n+1})) $$
and using that the increments of $\zeta$ and $W$ are independent,
$$ |\mathbb{E}[\left(a(u)\sigma(x)W_\gamma \right)_i (\gamma \zeta_{k+1}(x))_j]| = 0 $$
and using the It\=o isometry:
\begin{align*}
& \left| \mathbb{E}\left[ \left(a(u)\sigma(x)W_\gamma \right)_i \left( \int_0^\gamma(b_{a_{n+1}}(X^{x,n}_s) - b_{a_{n+1}}(x) + b_{a_{n+1}}(x) - b_{a(u)}(x))_j ds \right.\right.\right. \\
& \quad \left.\left.\left. + \int_0^\gamma ((a_{n+1}\sigma(X^{x,n}_s) - a_{n+1}\sigma(x) + a_{n+1}\sigma(x) - a(u)\sigma(x))dW_s)_j \right) \right] \right| \\
& \le C[b]_{\text{Lip}} \int_0^\gamma ||W_\gamma||_2 ||X^{x,n}_s -x||_2 ds + C(a_n^2-a_{n+1}^2)||W_\gamma||_1 \gamma ||\Upsilon||_\infty \\
& \quad + C\left|\sum_{k=1}^d \int_0^\gamma \mathbb{E}[\sigma_{ik}(x)(\sigma_{jk}(X^{x,n}_s)-\sigma_{jk}(x)] ds \right| + C(a_n - a_{n+1})\mathbb{E}[W_\gamma^2] \\
& \le CV^{1/2}(x)\gamma^2 + C(a_n-a_{n+1})\gamma^{3/2} + CV^{1/2}(x)\gamma^2 + C(a_n-a_{n+1})\gamma,
\end{align*}
where we used an argument similar to $A_{11}$ to bound the third term, using that $\nabla \sigma$ and $\nabla^2 \sigma$ are bounded.


\medskip

$\bullet$ \textbf{Term} $A_3$: Using the three fold Cauchy-Schwarz inequality, Lemma \ref{lemma:3.4.a} and Lemma \ref{lemma:3.4.b:Y:bar}, $A_3$ is bounded by
$$ C \left|\left| \sup_{\xi \in (x, X^{x,n}_\gamma)} ||\nabla^3 g(\xi)|| \right|\right|_4 V(x)\gamma \left(V^{1/2}(x)\gamma + \sqrt{\gamma}(a_n-a_{n+1})\right) .$$

\medskip

$\bullet$ \textbf{Term} $A_4$:
Using Lemma \ref{lemma:3.4.b:Y:bar}, $A_4$ is bounded by
$$ C \left(V^{1/2}(x) \gamma + \sqrt{\gamma} (a_n - a_{n+1})\right)^2 \left|\left|\sup_{\xi \in (X_\gamma^{x,n}, \bar{Y}_{\gamma,u}^x)} || \nabla^2 g(\xi) || \right|\right|_2 .$$

\end{proof}

\subsection{Proof of Theorem \ref{thm:non_definite}}

\begin{proof}
We remark that according to \cite[Proposition 3]{bras2021}, for all $\kappa > 0$ we have $e^{-\kappa g} \in L^1(\mathbb{R}^d)$.
We first prove that
\begin{equation}
\label{eq:W_nu_degen}
\mathcal{W}_1(\nu_{a_n},\nu_{a_{n+1}}) \le \frac{C}{n \log^{1+\alpha_{\min}}(n)} ,
\end{equation}
so that $(\mathcal{W}_1(\nu_{a_n},\nu_{a_{n+1}}))$ is still a converging Bertrand series.
To do so, we directly adapt the proof of Proposition \ref{prop:W_nu}, replacing the change of variables in the integrals in $ax$ by the change of variables in $B \cdot (a^{2\alpha_1}x_1, \ldots, a^{2\alpha_d}x_d)$. Still using \eqref{eq:app:exp_x_2}, we successively obtain
\begin{align*}
& \mathcal{Z}_a^{-1} \underset{a \to 0}{\sim} a^{2\alpha_1+\cdots+2\alpha_d} \int_{\mathbb{R}^d} e^{-2g(x)}dx \\
& \mathcal{Z}_{a_n}^{-1} - \mathcal{Z}_{a_{n+1}}^{-1} \le 4 a_{n+1}^{2\alpha_1 + \cdots + 2\alpha_d - 1}(a_n-a_{n+1}) \int_{\mathbb{R}^d} e^{-2g(x)}g(x) dx \\
& 1 - M_n^{-1} \le \frac{C}{n \log(n)} \\
& \left| \mathbb{E}|X-Y| - \mathbb{E}|X - \tilde{X}| \right| \le \frac{C a_{n+1}^{2\alpha_{\min}}}{n \log(n)}.
\end{align*}
Then, using \eqref{eq:W_nu_degen} we prove that $\mathcal{W}_1(\nu_n, \nu^\star) \le Ca_n^{2\alpha_{\min}}$ the same way as in Lemma \ref{lemma:W_nu_a_nu_star}.

The next parts of the proof are the same as for the definite positive case.
%
%
\end{proof}

\subsection{Proof of Theorem \ref{thm:conv_X}}
To prove Theorem \ref{thm:main:3}, we proceed as for the proof of Theorem \ref{thm:main}.

%
%
%

In the following, for $\gamma>0$ we denote by $(\bar{X}^{x,n,\gamma}_t)_{t \in [0,\gamma]}$ the Euler-Maruyama scheme over one step with coefficient $a_{n+1}$.
We first recall \cite[Lemma 3.4(b), Proposition 3.5(a)]{pages2020} giving bounds for the weak and strong errors for the one-step Euler-Maruyama scheme, which do not depend on the ellipticity parameter $a_n$.
\begin{lemma}
\label{lemma:3.4.b:X:bar}
Let $p \ge 1$ and let $\bar{\gamma}>0$. There exists $C>0$ such that for every $n \ge 0$, for every $\gamma \in (0,\bar{\gamma}]$ and every $t \in [0,\gamma]$:
$$ ||X_t^{x,n} - \bar{X}_t^{x,n,\gamma}||_p \le CV^{1/2}(x)t .$$
\end{lemma}

\begin{proposition}
\label{prop:3.5.a:X:bar}
Let $\bar{\gamma}>0$. Then for every $g : \mathbb{R}^d \to \mathbb{R}$ being $\mathcal{C}^3$ and for every $0 \le \gamma \le \gamma' \le \bar{\gamma}$:
$$ \left| \mathbb{E}\left[g(\bar{X}^{x,n,\gamma'}_\gamma)\right] - \mathbb{E}\left[g(X^{x,n}_\gamma)\right] \right| \le C V^{3/2}(x) \gamma^2 \Phi_g(x) ,$$
where
$$ \Phi_g(x) = \max \left( |\nabla g(x)|, || \nabla^2 g(x) ||, \left|\left| \sup_{\xi \in (X^{x,n}_\gamma, \bar{X}^{x,n,\gamma'}_\gamma)} ||\nabla^2 g(\xi)|| \right|\right|_2, \left|\left| \sup_{\xi \in (x, X^{x,n}_\gamma)} ||\nabla^3 g(\xi)|| \right|\right|_4 \right) .$$
\end{proposition}
%

\begin{proposition}
\label{prop:3.6:X:bar}
Let $T$, $\bar{\gamma}>0$. Then for every Lipschitz continuous function $f : \mathbb{R}^d \to \mathbb{R}$, for every $n \ge 0$ and every $t \in (0,T]$ and every $0 \le \gamma \le \gamma' \le \bar{\gamma}$:
$$ \left| \mathbb{E}\left[P_t f(\bar{X}_\gamma^{x,n,\gamma'})\right]  - \mathbb{E}\left[P_t f(X_\gamma^{x,n})\right] \right| \le C a_n^{-3} [f]_{\textup{Lip}} \gamma^2 t^{-1} V^2(x). $$
\end{proposition}
\begin{proof}
The proof is the same as in \cite[Proposition 3.6]{pages2020}. When applying \cite[Proposition 3.2(b)]{pages2020}, we remark that the lowest exponent of $\ubar{\sigma}_0$ is $-3$.
\end{proof}

Moreover, by the same proof as in Lemma \ref{lemma:D.1a} we get
$$ \textstyle \sup_{m \ge k+1} \mathbb{E} V^p(\bar{X}_{\Gamma_m-\Gamma_k}^{x,n}) \le CV^p(x). $$

We now prove Theorem \ref{thm:main:3}.
\begin{proof}
Let us write:
$$ \mathcal{W}_1([\bar{X}_{T_n}^{x_0}],\nu^\star) \le \mathcal{W}_1([\bar{X}_{T_n}^{x_0}], [X_{T_n}^{x_0}]) + \mathcal{W}_1([X_{T_n}^{x_0}],\nu^\star). $$
Temporarily setting $\bar{x}_n := \bar{X}_{T_n}^{x_0}$ and $x_n := X_{T_n}^{x_0}$, we have
\begin{align*}
\mathcal{W}_1([\bar{X}_{T_{n+1}}^{x_0}], [X_{T_{n+1}}^{x_0}]) & = \mathcal{W}_1([\bar{X}_{T_{n+1}-T_n}^{\bar{x}_{n},n}], [X_{T_{n+1}-T_{n}}^{x_{n},n}]) \\
& \le \mathcal{W}_1([\bar{X}_{T_{n+1}-T_n}^{\bar{x}_{n},n}], [X_{T_{n+1}-T_{n}}^{\bar{x}_{n},n}]) + \mathcal{W}_1([X_{T_{n+1}-T_{n}}^{\bar{x}_{n},n}], [{X}_{T_{n+1}-T_{n}}^{x_{n},n}]),
\end{align*}
and we find a bound on the first term using the same proof as in \cite[Section 4.2]{pages2020}. For $x \in \mathbb{R}^d$, we split $| \mathbb{E} f(\bar{X}_{T_{n+1}-T_n}^{x,n}) - \mathbb{E} f(X_{T_{n+1}-T_n}^{x,n}) |$ into three terms $(a)$, $(b)$ and $(c)$. We however pay attention to the dependence in $a_n$ when applying Lemma \ref{lemma:3.4.b:X:bar}, Proposition \ref{prop:3.6:X:bar} and Theorem \ref{thm:confluence}. We then have:
\begin{align*}
& (c) \le C [f]_{\text{Lip}} \gamma_{N(T_n)} V^{1/2}(x), \\
& (b) \le Ca_{n+1}^{-3} \gamma_{N(T_n)} \log\left(\frac{T + ||\gamma||_\infty}{\gamma_{N(T_n)}}\right), \\
& (a) \le Ca_{n+1}^{-3} e^{C_1 a_{n+1}^{-2}} V(x) \gamma_{N(T_n)} \rho_{n+1}^{-1}.
\end{align*}
Then we establish a recursive relation and prove the convergence as in the proof of Theorem \ref{thm:main}.

\end{proof}


\begin{thebibliography}{10}

\bibitem{bally1996}
{\sc Bally, V., and Talay, D.}
\newblock The law of the {E}uler scheme for stochastic differential equations.
  {I}. {C}onvergence rate of the distribution function.
\newblock {\em Probab. Theory Related Fields 104}, 1 (1996), 43--60.

\bibitem{bras2021}
{\sc {Bras}, P.}
\newblock {Convergence rates of Gibbs measures with degenerate minimum}.
\newblock {\em arXiv e-prints\/} (2021), arXiv:2101.11557.

\bibitem{chaing1987}
{\sc Chiang, T.-S., Hwang, C.-R., and Sheu, S.~J.}
\newblock Diffusion for global optimization in {${\bf R}^n$}.
\newblock {\em SIAM J. Control Optim. 25}, 3 (1987), 737--753.

\bibitem{dalalyan2014}
{\sc Dalalyan, A.~S.}
\newblock Theoretical guarantees for approximate sampling from smooth and
  log-concave densities.
\newblock {\em J. R. Stat. Soc. Ser. B. Stat. Methodol. 79}, 3 (2017),
  651--676.

\bibitem{dauphin2015}
{\sc Dauphin, Y., de~Vries, H., and Bengio, Y.}
\newblock Equilibrated adaptive learning rates for non-convex optimization.
\newblock In {\em Neural Information Processing Systems\/} (2015).

\bibitem{dauphin2014}
{\sc Dauphin, Y.~N., Pascanu, R., Gulcehre, C., Cho, K., Ganguli, S., and
  Bengio, Y.}
\newblock {Identifying and Attacking the Saddle Point Problem in
  High-Dimensional Non-Convex Optimization}.
\newblock In {\em Proceedings of the 27th International Conference on Neural
  Information Processing Systems - Volume 2\/} (Cambridge, MA, USA, 2014),
  NIPS'14, MIT Press, p.~2933–2941.

\bibitem{duchi2011}
{\sc Duchi, J., Hazan, E., and Singer, Y.}
\newblock {Adaptive Subgradient Methods for Online Learning and Stochastic
  Optimization}.
\newblock {\em Journal of Machine Learning Research 12\/} (2011), 2121–2159.

\bibitem{gelfand-mitter}
{\sc Gelfand, S.~B., and Mitter, S.~K.}
\newblock Recursive stochastic algorithms for global optimization in {${\bf
  R}^d$}.
\newblock {\em SIAM J. Control Optim. 29}, 5 (1991), 999--1018.

\bibitem{hwang1980}
{\sc Hwang, C.-R.}
\newblock Laplace's method revisited: weak convergence of probability measures.
\newblock {\em Ann. Probab. 8}, 6 (1980), 1177--1182.

\bibitem{adam}
{\sc Kingma, D.~P., and Ba, J.}
\newblock Adam: {A} method for stochastic optimization.
\newblock In {\em 3rd International Conference on Learning Representations,
  {ICLR} 2015, San Diego, CA, USA, May 7-9, 2015, Conference Track
  Proceedings\/} (2015), Y.~Bengio and Y.~LeCun, Eds.

\bibitem{lamberton2002}
{\sc Lamberton, D., and Pag\`es, G.}
\newblock Recursive computation of the invariant distribution of a diffusion.
\newblock {\em Bernoulli 8}, 3 (2002), 367--405.

\bibitem{lazarev1992}
{\sc Lazarev, V.~A.}
\newblock Convergence of stochastic approximation procedures in the case of
  several roots of a regression equation.
\newblock {\em Problemy Peredachi Informatsii 28}, 1 (1992), 75--88.

\bibitem{mnist}
{\sc Lecun, Y., Bottou, L., Bengio, Y., and Haffner, P.}
\newblock Gradient-based learning applied to document recognition.
\newblock {\em Proceedings of the IEEE 86}, 11 (1998), 2278--2324.

\bibitem{li2015}
{\sc Li, C., Chen, C., Carlson, D., and Carin, L.}
\newblock Preconditioned stochastic gradient langevin dynamics for deep neural
  networks.
\newblock In {\em Proceedings of the Thirtieth AAAI Conference on Artificial
  Intelligence\/} (2016), AAAI'16, AAAI Press, p.~1788–1794.

\bibitem{ma2015}
{\sc Ma, Y., Chen, T., and Fox, E.~B.}
\newblock {A Complete Recipe for Stochastic Gradient MCMC}.
\newblock In {\em Neural Information Processing Systems\/} (2015).

\bibitem{miclo1992}
{\sc Miclo, L.}
\newblock Recuit simul\'{e} sur {${\bf R}^n$}. \'{E}tude de l'\'{e}volution de
  l'\'{e}nergie libre.
\newblock {\em Ann. Inst. H. Poincar\'{e} Probab. Statist. 28}, 2 (1992),
  235--266.

\bibitem{neelakantan2015}
{\sc {Neelakantan}, A., {Vilnis}, L., {Le}, Q.~V., {Sutskever}, I., {Kaiser},
  L., {Kurach}, K., and {Martens}, J.}
\newblock {Adding Gradient Noise Improves Learning for Very Deep Networks}.
\newblock {\em arXiv e-prints\/} (Nov. 2015), arXiv:1511.06807.

\bibitem{nocedal2006}
{\sc Nocedal, J., and Wright, S.~J.}
\newblock {\em Numerical optimization}, second~ed.
\newblock Springer Series in Operations Research and Financial Engineering.
  Springer, New York, 2006.

\bibitem{pages2020}
{\sc {Pagès}, G., and {Panloup}, F.}
\newblock {Unajusted Langevin algorithm with multiplicative noise: Total
  variation and Wasserstein bounds}.
\newblock {\em arXiv e-prints\/} (2020), arXiv:2012.14310.

\bibitem{patterson2013}
{\sc Patterson, S., and Teh, Y.~W.}
\newblock {Stochastic Gradient Riemannian Langevin Dynamics on the Probability
  Simplex}.
\newblock In {\em Advances in Neural Information Processing Systems\/} (2013),
  C.~J.~C. Burges, L.~Bottou, M.~Welling, Z.~Ghahramani, and K.~Q. Weinberger,
  Eds., vol.~26, Curran Associates, Inc.

\bibitem{royer1989}
{\sc Royer, G.}
\newblock A remark on simulated annealing of diffusion processes.
\newblock {\em SIAM J. Control Optim. 27}, 6 (1989), 1403--1408.

\bibitem{sagun2016}
{\sc {Sagun}, L., {Bottou}, L., and {LeCun}, Y.}
\newblock {Eigenvalues of the Hessian in Deep Learning: Singularity and
  Beyond}.
\newblock {\em arXiv e-prints\/} (2016), arXiv:1611.07476.

\bibitem{sagun2017}
{\sc {Sagun}, L., {Evci}, U., {Ugur Guney}, V., {Dauphin}, Y., and {Bottou},
  L.}
\newblock {Empirical Analysis of the Hessian of Over-Parametrized Neural
  Networks}.
\newblock {\em arXiv e-prints\/} (2017), arXiv:1706.04454.

\bibitem{simsekli2016}
{\sc Simsekli, U., Badeau, R., Cemgil, A.~T., and Richard, G.}
\newblock {Stochastic Quasi-Newton Langevin Monte Carlo}.
\newblock In {\em Proceedings of the 33rd International Conference on
  International Conference on Machine Learning - Volume 48\/} (2016), ICML'16,
  p.~642–651.

\bibitem{talay1990}
{\sc Talay, D., and Tubaro, L.}
\newblock Expansion of the global error for numerical schemes solving
  stochastic differential equations.
\newblock {\em Stochastic Anal. Appl. 8}, 4 (1990), 483--509 (1991).

\bibitem{laarhoven1987}
{\sc van Laarhoven, P. J.~M., and Aarts, E. H.~L.}
\newblock {\em Simulated annealing: theory and applications}, vol.~37 of {\em
  Mathematics and its Applications}.
\newblock D. Reidel Publishing Co., Dordrecht, 1987.

\bibitem{villani2009}
{\sc Villani, C.}
\newblock {\em Optimal transport}, vol.~338 of {\em Grundlehren der
  Mathematischen Wissenschaften [Fundamental Principles of Mathematical
  Sciences]}.
\newblock Springer-Verlag, Berlin, 2009.
\newblock Old and new.

\bibitem{wang2012}
{\sc Wang, F.-Y.}
\newblock Coupling and applications.
\newblock In {\em Stochastic analysis and applications to finance}, vol.~13 of
  {\em Interdiscip. Math. Sci.} World Sci. Publ., Hackensack, NJ, 2012,
  pp.~411--424.

\bibitem{wang2020}
{\sc Wang, F.-Y.}
\newblock Exponential contraction in {W}asserstein distances for diffusion
  semigroups with negative curvature.
\newblock {\em Potential Anal. 53}, 3 (2020), 1123--1144.

\bibitem{welling2011}
{\sc Welling, M., and Teh, Y.~W.}
\newblock {Bayesian Learning via Stochastic Gradient Langevin Dynamics}.
\newblock In {\em Proceedings of the 28th International Conference on
  International Conference on Machine Learning\/} (2011), ICML'11, Omnipress,
  p.~681–688.

\bibitem{adadelta}
{\sc {Zeiler}, M.~D.}
\newblock {ADADELTA: An Adaptive Learning Rate Method}.
\newblock {\em arXiv e-prints\/} (Dec. 2012), arXiv:1212.5701.

\bibitem{zitt2008}
{\sc Zitt, P.-A.}
\newblock Annealing diffusions in a potential function with a slow growth.
\newblock {\em Stochastic Process. Appl. 118}, 1 (2008), 76--119.

\end{thebibliography}
\end{document}